\documentclass[11pt]{amsart}
\usepackage[colorlinks = true,
            linkcolor = red,
            urlcolor  = blue,
            citecolor = blue,
            anchorcolor = blue]{hyperref}

\makeatletter
\def\@tocline#1#2#3#4#5#6#7{\relax
  \ifnum #1>\c@tocdepth 
  \else
    \par \addpenalty\@secpenalty\addvspace{#2}%
    \begingroup \hyphenpenalty\@M
    \@ifempty{#4}{%
      \@tempdima\csname r@tocindent\number#1\endcsname\relax
    }{%
      \@tempdima#4\relax
    }%
    \parindent\z@ \leftskip#3\relax \advance\leftskip\@tempdima\relax
    \rightskip\@pnumwidth plus4em \parfillskip-\@pnumwidth
    #5\leavevmode\hskip-\@tempdima
      \ifcase #1
       \or\or \hskip 1em \or \hskip 2em \else \hskip 3em \fi%
      #6\nobreak\relax
    \dotfill\hbox to\@pnumwidth{\@tocpagenum{#7}}\par
    \nobreak
    \endgroup
  \fi}
\makeatother

\usepackage{tikz,amsthm,amsmath,amstext,amssymb,amscd,epsfig,euscript, mathrsfs, dsfont,pspicture,multicol, mathtools,graphpap,graphics,graphicx,times,enumerate,subfig,sidecap,wrapfig,color}

\usepackage{ hyperref}
\usepackage{enumerate}
\usepackage{esint}
\usepackage{verbatim} 

 \numberwithin{equation}{section}

\def\bC{{\mathbb{C}}}

\def\NN{{\mathbb{N}}}

\def\mS{{\mathcal{S}}}
\def\WW{{\mathcal{W}}}
\def\PP{{\mathcal{P}}}

\def\calN{{\mathcal{N}}}
\def\calC{{\mathcal{C}}}
\def\calB{{\mathcal{B}}}

\def\calA{{\mathcal{A}}}

\def\ve{\varepsilon}
\def\vp{\varphi}

\def\lec{\lesssim}
\def\gec{\gtrsim}

\def\wt{\widetilde}

\def\ch{\mathop\mathrm{ch}} 

\def\Rn1{\mathbb{R}^{n+1}}
\def\rrn{\mathbb{R}^{n+1}}
\def\rn{\mathbb{R}^{n}}
 
\def\oom{\overline \Omega} 					

\DeclareMathOperator{\diam}{diam}

\newcommand{\vd}{{\mathrm{D^L_v}}}
\newcommand{\vpd}{{\mathrm{PD^L_v}}}
\newcommand{\dirprime}{{\mathrm{D}^{L^*}_{p'}}}

\newcommand{\rpr}{{\mathrm{R}^L_p}}

\newcommand{\pd}{{\mathrm{PD}}}
\newcommand{\pr}{{\mathrm{PR}^L_p}}

\newcommand{\dqr}{{\mathrm{D}^L_q}}
\newcommand{\pre}{{\mathrm{PR}}}
\newcommand{\tdqr}{{\wt{\mathrm{D}}^L_q}}

\newcommand{\sss}{{\mathsf {Stop}}}
\newcommand{\ttt}{{\mathsf {Top}}}

\newcommand{\tree}{{\rm Tree}}

\def\car{\mathop\mathrm{Car}}

\newcommand{\nn}{{\mathrm{N}}}
\newcommand{\cc}{{\mathrm{C}}}

\newcommand{\ccis}{{\mathrm{C}^\infty_s}}
\newcommand{\nnp}{{\mathrm{N}^p}}

\newcommand{\nni}{{\mathrm{N}^\infty}}
\newcommand{\nnis}{{\mathrm{N}_{\sharp}^\infty}}
\newcommand{\ccpqs}{{\mathrm{C}^p_{s, q}}}
\newcommand{\ccpqsd}{{\mathrm{C}^{p'}_{s, q'}}}
\newcommand{\ccpsi}{{\mathrm{C}^p_{s, \infty}}}
\newcommand{\ccisi}{{\mathrm{C}^\infty_{s, \infty}}}
\newcommand{\nnpq}{{\mathrm{N}^p_q}}

\def\TT{\mathop\mathrm{T}}
\def\AR{\mathop\mathrm{AR}}

\def\JC{\mathop\mathrm{JC}}
\def\tr{\mathop\mathrm{Tr}}
\def\BMO{\mathop\mathrm{BMO}} 					
\def\VMO{\mathop\mathrm{VMO}} 					
\def\Lip{\mathop\mathrm{Lip}} 						

\def\dist{\textup{dist}} 						
\def\supp{\mathop\mathrm{supp}}					
\def\loc{\mathop\mathrm{loc}}						
\DeclareMathOperator*{\qntlim}{qnt-lim}			
\DeclareMathOperator*{\ntlim}{nt-lim}			

\renewcommand{\div}{\mathop\mathrm{div }}			

\def\XXint#1#2#3{{\setbox0=\hbox{$#1{#2#3}{\int}$ }
\vcenter{\hbox{$#2#3$ }}\kern-.58\wd0}}

\theoremstyle{plain}
\newtheorem{theorem}{Theorem}

\newtheorem{corollary}[theorem]{Corollary}

\newtheorem{lemma}[theorem]{Lemma}

\newtheorem{proposition}[theorem]{Proposition}

\theoremstyle{definition}

\newtheorem{definition}[theorem]{Definition}
\newtheorem{remark}[theorem]{Remark}

\numberwithin{equation}{section}
\numberwithin{theorem}{section}

  \DeclareFontFamily{U}{mathb}{\hyphenchar\font45} 
\DeclareFontShape{U}{mathb}{m}{n}{
      <5> <6> <7> <8> <9> <10> gen * mathb
      <10.95> mathb10 <12> <14.4> <17.28> <20.74> <24.88> mathb12
      }{}
\DeclareSymbolFont{mathb}{U}{mathb}{m}{n}

\DeclareMathSymbol{\toitself}      {3}{mathb}{"FD}  

\def\HH{\mathcal{H}}
\newcommand{\vv}{\vspace{2mm}}
\newcommand{\vvv}{\vspace{4mm}}

\newcommand{\dv}{\mathop{\rm div}}
\def\R{\mathbb{R}}
\def\vphi{\varphi}
\def\om{\Omega}

\def\hm{\omega}
\def\pom{{\partial\Omega}}
\newcommand{\DD}{{\mathcal D}}

  \usepackage{etoolbox}
\appto\appendix{\addtocontents{toc}{\protect\setcounter{tocdepth}{1}}}

  \makeatletter
\newcommand{\setword}[2]{%
  \phantomsection
  #1\def\@currentlabel{\unexpanded{#1}}\label{#2}%
}
\makeatother

\begin{document}

\title[Varopoulos extensions  in domains with Ahlfors-regular boundaries]{Varopoulos extensions in domains with Ahlfors-regular boundaries and applications to Boundary Value Problems for elliptic systems with $L^\infty$ coefficients}

\author[Mihalis Mourgoglou]{Mihalis Mourgoglou}
\address{Departamento de Matem\'aticas, Universidad del Pa\' is Vasco, Barrio Sarriena s/n 48940 Leioa, Spain and\\
Ikerbasque, Basque Foundation for Science, Bilbao, Spain.}
\email{michail.mourgoglou@ehu.eus}

\author{Thanasis Zacharopoulos}
\address{Departamento de Matem\'aticas, Universidad del Pa\' is Vasco, Barrio Sarriena s/n 48940 Leioa, Spain. }
\email{athanasios.zacharopoulos@ehu.eus}

\thanks{M.M. was supported  by IKERBASQUE and partially supported by the grant PID2020-118986GB-I00 of the Ministerio de Econom\'ia y Competitividad (Spain) and by the grant IT-1615-22 (Basque Government).  T.Z.  was supported by the FPI grant PRE2018-084984
of the Ministerio de Ciencia e Innovaci\'on (Spain) and  partially supported by the grant PID2020-118986GB-I00 of the Ministerio de Econom\'ia y Competitividad (Spain).}

\newcommand{\mih}[1]{\marginpar{\color{red} \scriptsize \textbf{Mi:} #1}}
\newcommand{\than}[1]{\marginpar{\color{blue} \scriptsize \textbf{Thanasis:} #1}}
\maketitle

\begin{abstract}
Let $\Omega \subset \mathbb{R}^{n+1}$, $n \geq 1$,  be an open set with $s$-Ahlfors regular boundary $\partial \Omega$, for some $s \in(0,n]$,  such that either  $s=n$ and  $\Omega$ is a corkscrew domain with the pointwise John condition,  or $s<n$ and $\Omega= \mathbb{R}^{n+1} \setminus E$,  for some $s$-Ahlfors regular set $E \subset \mathbb{R}^{n+1}$.  In this paper we provide a unifying method to construct Varopoulos type extensions of $\BMO$ and $L^p$  boundary functions.  In particular,  we show that a)  if $ f \in \BMO(\partial \Omega)$,  there exists $ F\in C^\infty(\Omega)$ such that $\dist(x, \om^c)|\nabla  F(x)|$ is uniformly bounded in $\om$ and  the Carleson functional of $\dist(x,\Omega^c)^{s-n}|\nabla  F(x)|$  as well  the sharp non-tangential maximal function of $ F$ are uniformly bounded on $\partial \Omega$ with norms controlled by the $\BMO$-norm of $ f$,  and $ F \to  f$ in a certain non-tangential sense  $\mathcal H^s|_{\partial \Omega}$-almost everywhere;  b)  if $\bar f \in L^p(\partial \Omega)$, $1 <p \leq \infty$,   there exists $\bar F \in C^\infty(\Omega)$ such that the non-tangential maximal functions of $\bar F$ and $\dist(\cdot, \om^c)|\nabla \bar F|$ as well as  the Carleson functional of $\dist(\cdot,\Omega^c)^{s-n}|\nabla \bar  F|$  are  in $L^p(\partial \Omega)$ with norms controlled by the $L^p$-norm of $\bar f$,  and $\bar F \to \bar f$ in some non-tangential sense $\mathcal H^s|_{\partial \Omega}$-almost everywhere.  If,  in addition,  the boundary function is Lipschitz with compact support,  then both $F$ and $\bar F$ can be constructed so that they are also Lipschitz on $\overline \Omega$ and converge to the boundary data continuously.  The latter results hold without  the additional assumption of the pointwise John condition.  Finally,  for elliptic systems of equations in divergence form with merely bounded complex-valued coefficients, we show some connections between the solvability of Poisson problems with interior data in the appropriate Carleson  or tent spaces and the solvability of Dirichlet problem with $L^p$ and $\BMO$ boundary data. 
\end{abstract}

\tableofcontents

\section{Introduction}

In the present manuscript we are concerned with open sets  $\om \subset \rrn$, $n \geq 1$, which satisfy  one of the following assumptions:
\begin{itemize}
\item[(a)] $\om$ satisfies the corkscrew condition and  its boundary $\pom$ is $n$-Ahlfors regular (see Definitions \ref{def:sregular} and \ref{def:corkscrew}),  or
\item[(b)] $\om=\rrn \setminus E$,  for some $s$-Ahlfors regular set $E \subset \rrn$ with $s<n$.
\end{itemize}
We will call such domains {\it $\AR(s)$ domains} for $s \in (0,n]$.  We also define  $\sigma_s:= \HH^s|_{\pom}$ to be the ``surface" measure of $\om$, where $\HH^s$ is the $s$-dimensional Hausdorff measure. 

\vv

Our first   goal is to construct, in $\AR(s)$ domains,  smooth  extensions $u:\om \to \R$ of  boundary functions that  are  in $\BMO(\sigma_s)$  (resp.  in  $L^p(\sigma_s)$ for $ p \in (1, \infty]$) so that their  sharp non-tangential maximal function defined in  \eqref{def: sharpntmaxfunction} (resp.  their non-tangential maximal functions defined in \eqref{def: ntmaxfunction}) and the modified Carleson functionals (see \eqref{eq: carleson sup} for the definition) of their ``weighted" gradients are uniformly bounded  (resp.  in $L^p(\sigma_s)$) with norms controlled by the $\BMO(\sigma_s)$  (resp.  $L^p(\sigma_s)$) norms of the boundary functions.  The identification on the boundary is in the non-tangential convergence sense (up to a set of measure zero on the boundary).  To do so,  when $s=n$, we assume that  $\om$ satisfies the {\it pointwise John condition} (see Definition \ref{def: pwJohn}), while no additional connectivity  assumption is required for $s<n$. This is the first time that such results are proved in so general geometric  setting and also when the  co-dimension is larger than $1$.

Our second goal is to construct such extensions of Lipschitz functions  with compact support on the boundary of an $\AR(s)$ domain so that they are  Lipschitz on $\oom$ and in the weighted Sobolev space $\dot W^{1,2}(\om;\hm_s)$, where $\hm_s(x):=\delta_\om(x)^{s-n}$\footnote{When $s=n$, this is the standard  homogeneous Sobolev space $\dot W^{1,2}(\om)$.}.   In fact,  the construction of those extensions is even more important due to their applications to Boundary Value Problems given in Section \ref{sec:applications} and also in \cite{GaMT23, MP24}.  Finally, we also prove similar extensions of  boundary  functions in the Campanato space $\Lambda_\beta ( \pom)$ for $\beta \in (0,1)$.

\vv

The original result of Varopoulos  \cite{VAR77,  VAR78} entails an extension property for $\text{BMO}$ functions in $\mathbb{R}^n$. Specifically, for $f \in \text{BMO}(\mathbb{R}^n)$, there exists an extension $F$ defined in the upper half-space $\mathbb{R}^{n+1} := \{(x, t) \in \mathbb{R}^n \times (0, \infty)\}$, such that $F(x, t) = f$ non-tangentially, with $F \in C^\infty(\mathbb{R}^{n+1})$, satisfying a size condition  given by
\begin{equation}\label{eq:size}
\sup_{{(x,t) \in  \R^{n+1}_+}} t\, {|\nabla F(x, t)|} <\infty,
\end{equation}
and a $L^1$ Carleson measure condition
\begin{equation}\label{eq:carleson-bis}
\sup_{{r > 0, x \in \mathbb{R}^n}} r^{-n} \iint_{\{|x-y| < r\} \cap \R^{n+1}_+}  |\nabla F(y, t)| \, dy \, dt \leq C \|f\|_{\text{BMO}},
\end{equation}
where the implicit constants are purely dimensional. It is crucial to note that the Carleson condition \eqref{eq:carleson-bis} involves an $L^1$ bound, in contrast to the more standard $L^2$ estimate (with respect to the weighted measure $tdydt$):
\begin{equation}\label{eq:square function}
\sup_{{r > 0, x \in \mathbb{R}^n}} r^{-n} \iint_{\{|x-y| < r\} \cap \R^{n+1}_+} |\nabla F(y, t)|^2 \, dy \, td t \leq C \|f\|_{\text{BMO}}^2.
\end{equation}
In the presence of the size condition \eqref{eq:size},  it is evident that \eqref{eq:carleson-bis} implies \eqref{eq:square function}. However, the former is strictly stronger. An example illustrating this distinction is provided by Garnett in  \cite{GAR81}. In this example,  \eqref{eq:square function} holds, but \eqref{eq:carleson-bis} fails. Garnett's example involves the standard Poisson extension of a suitably constructed (bounded) function $f$.  In this case, Fefferman and Stein \cite{FS} established \eqref{eq:square function} for the Poisson extension of $f$,  and this bound played a crucial role in his proof of the $H^1$-$\BMO$ duality.

Varopoulos initially sought to establish his extension result with the aim of extending Carleson's Corona Theorem \cite{Car62} to $\mathbb{C}^n$ with $n \geq 2$. While this particular endeavor did not succeed, the outcome, known as the “Varopoulos extension", has proven to be of significant interest for other reasons. Notably, employing \eqref{eq:carleson-bis} instead of \eqref{eq:square function} allows for a simplification in one of the key steps of Fefferman's duality theorem proof.

Varopoulos based his extension theorem on a deep property of harmonic functions (or solutions of divergence form elliptic equations more broadly), now termed $\varepsilon$-{\it approximability}. This property asserts that for any bounded harmonic function $u$  (initially defined in the half-space but extendable to more general settings),  normalized so that $\|u\|_\infty \leq 1$, and for a fixed $\varepsilon > 0$, there exists an “$\varepsilon$-approximator" $\varphi_\ve \in C^\infty(\mathbb{R}^{n+1}_+)$ such that
\begin{equation}\label{eq:eapprox}
 \|u - \varphi_\ve\|_{L^\infty(\mathbb{R}^{n+1}_+)} < \varepsilon
 \end{equation}
and $\displaystyle \|\mathscr{C}_n (\nabla \vphi_\ve)\|_{L^\infty(\pom)} <  C_{\ve}$.  In particular, he demonstrated that there exists  $\varepsilon \in(0, 1)$ such that every  harmonic function $u$,  as described above, can be $\varepsilon$-approximated. Garnett later established that this holds for every $0 < \varepsilon < 1$. Thus, even though Garnett's example demonstrates that \eqref{eq:carleson-bis} can fail for bounded harmonic functions, these functions can still be approximated closely in the $L^\infty$ norm by a function $\varphi_\ve$ that does satisfy this property.

It is noteworthy that, for example, in Lipschitz domains, the $\varepsilon$-approximability of bounded solutions for divergence form elliptic equations is connected to the solvability of the Dirichlet problem with $L^p$ data\footnote{The relationships outlined have been explored and utilized in the works of  Dahlberg \cite{DAH80} and Kenig, Koch, Pipher, and Toro \cite{KKPT00}, with significant  contributions from Garnett \cite{GAR81}.}. Furthermore, in the case of harmonic functions in more general domains, the $\varepsilon$-approximation property has significant geometric implications.  In particular, if $\om \in \AR(n)$,  it {\it characterizes}  uniform rectifiability of the boundary\footnote{Uniform rectifiability is a quantitative,  scale-invariant version of the classical notion of rectifiability whose theory has been developed extensively in the deep work of David and Semmes \cite{DS1, DS2}.}. This equivalence is derived from the combined work of Hofmann, Martell, and Mayboroda \cite{HMM16} and Garnett,  Tolsa, and the first named author  \cite{GMT18} (see also  \cite{AGMT}).

Varopoulos established his extension theorem by first iterating $\varepsilon$-approximability to obtain the extension property for  $f \in L^\infty(\R^n)$. Subsequently, he used a ``Corona" type decomposition of $\text{BMO}(\R^n)$ functions, credited to Garnett \cite{GAR81}, to establish the general case.  Varopoulos's approach is both powerful and innovative; however, it has a fundamental limitation. Due to its reliance on the $\varepsilon$-approximability property of bounded harmonic functions, the Varopoulos method cannot treat domains whose boundaries fail to be uniformly rectifiable (see \cite{HT20} for the construction of Varopoulos extensions  in such domains).

\vv

In the present manuscript,  we overcome this geometric obstacle and clarify the true nature of the important extension property of Varopoulos, and the ingredients that go into its proof.   We also disprove a conjecture of Hofmann and Tapiola (originally stated in 
the preprint version of \cite{HT20} on arXiv), which was saying that the existence of Varopoulos extensions of $\BMO$ functions in $\AR(n)$ domains implies uniform $n$-rectifiability of the boundary.  Indeed, one can easily find  sets which are $n$-Ahlfors regular and  purely unrectifiable so that their complements are uniform domains (i.e.,  they satisfy Definitions \ref{def:corkscrew} and \ref{def:harnackchain}) and so they satisfy the local John (and thus the pointwise John) condition (see Definitions \ref{def: pwJohn} and \ref{def: local John}).  The $4$-corner Cantor set of Garnett is such an example in $\R^2$.  

\vv

Our first main result is the following. 

\begin{theorem}\label{th: introtheorem3}
Let $ \om \in \AR(s)$ for $s \in (0, n]$,  which,  for $s=n$, satisfies the pointwise John condition.  If $ f\in \BMO(\sigma_s)$, there exist  $ u : \om \to \R $ and $c_0\in(0,\frac{1}{4}]$ such that, for any $c\leq c_0$,  it holds that 
\begin{itemize}
\item[(i)] $u\in C^\infty(\om)$,
\item[(ii)] $\displaystyle  \sup_{\xi \in \pom} \calN_{\sharp,c}(u)(\xi)+ \sup_{x\in \om} \delta_\om(x) | \nabla u(x)|   \lec \| f \|_{\BMO(\sigma_s)} $, 
\item[(iii)] $ \displaystyle \sup_{\xi \in \pom}  {\calC}_{s,c}(\nabla u)(\xi)  \lec \|f\|_{\BMO(\sigma_s)}  $,
\item[(iv)] For $\sigma_s$-almost every $\xi\in\pom$,
$$
\displaystyle \fint_{B(x, \delta_\om(x)/2)} u(y)\,dy   \to 
\begin{dcases}
\displaystyle {f(\xi)\,\,\,\textup{non-tangentially}}, &\textup{if}\,\, s<n\\
\displaystyle {f(\xi)\,\,\,\textup{quasi-non-tangentially}}, &\textup{if}\,\, s=n.
\end{dcases}
$$
\end{itemize} 
Moreover, if $ f\in L^\infty(\sigma_s)$, we also obtain the stronger estimate
\begin{itemize}
\item[(v)]  $\displaystyle \sup_{x\in \om} |u(x)|  \lec \| f \|_{L^\infty(\sigma_s)}$.
\end{itemize} 
The constant $c_0$ only depends on dimension, the Ahlfors regularity constants, and the corkscrew condition.   Moreover, in the case that $s=n$ and $\om$ satisfies the {\it pointwise John condition}  but not the {\it local John condition},  we also assume that $f$ is compactly supported.
\end{theorem}

\begin{remark}
Note that the second term on the left hand side of the estimate  (ii) of  Theorem \ref{th: introtheorem3} can also be written in terms of the non-tangential maximal function.  Namely,  (ii) is equivalent to the estimate
$$
 \sup_{\xi \in \pom} \calN_{\sharp,c}(u)(\xi) +  \sup_{\xi \in \pom} \calN( \delta_\om \nabla u )(\xi)  \lec  \| f \|_{\BMO(\sigma_s)}.
$$
\end{remark}

Our proof of Theorem \ref{th: introtheorem3}, under the additional assumption that the local John condition is satisfied when $s=n$, does not involve a decomposition of the boundary data, as in the Varopoulos argument.  Instead, it relies on a direct approach, which is novel in the case of $L^\infty$ and $\BMO$ functions, and hinges on three crucial components. First, we apply Theorem \ref{thm:e-approxbmo}, demonstrating that the regularized dyadic extension of $f \in \BMO(\sigma_s)$ (see \eqref{eq:dyadext} for its definition) is \textit{uniformly $\varepsilon$-approximable} (see subsection \ref{subsec:approx} for the definition).  Second, we establish a trace theorem (see Definition \ref{def:ntconv} and Proposition \ref{lem: non-tangential convergence}).  Finally, we employ an iteration argument inspired by Varopoulos.  

It is worth noting that the local John condition, representing a  scale-invariant connectivity condition between boundary points and corkscrew points at all scales and locations, is necessary only for the trace theorem and specifically in the case where $\Omega \in \AR(s)$ and $s=n$, while it is always satisfied in $\AR(s)$ for $s<n$. A significant challenge in the iteration argument is the need to introduce the Sobolev-type space $\nn_{\textup{sum}}(\Omega)$ (see \eqref{eq: Nsum} for its definition) and demonstrate its sequential completeness. This constitutes a non-trivial task, occupying the majority of the Appendix (see Corollary \ref{cor: the sum space is complete}).

The aforementioned method does not seem to work when $s=n$ and the domain satisfies a  connectivity condition weaker than the local John condition  called the {\it pointwise John condition}.  So in the case that $\om$ satisfies the pointwise John condition but {\it not} the local John condition,  we resort to a proof closer in spirit to that of Varopoulos. Specifically, we decompose the boundary data $f$ into the sum of a `good' function $g \in L^\infty$ and a `bad' function $b = \sum_j a_j \chi_{Q_j}$, where $\sup_{j \geq 1}|a_j|\lesssim \|f\|_{\BMO(\sigma)}$ and $\{Q_j\}_{j \geq 1}$ is a family of dyadic cubes on the boundary satisfying a Carleson packing condition (see \eqref{eq:packing} for the definition). The construction of the extension of $b$ is more standard,  although, in rough domains, it is technically involved and requires quite some effort to be accomplished (see  \cite[Proposition 1.3]{HT20}). The challenging part is how to build the extension of $g$ without relying on the $\varepsilon$-approximability property of the harmonic extension of $g$ (and, thus, avoiding the restriction to domains with uniformly rectifiable boundaries). Interestingly, the (direct) proof described above works for $L^\infty$ functions in domains merely satisfying the pointwise John condition, allowing us to construct a Varopoulos extension of $g$.  The reason this occurs is that, rather than obtaining an estimate for $\calN_{\sharp,c}(u)$ for the approximating function of $\upsilon_g$,  we are able to establish a stronger estimate like (v).  The extension of $f$ is just the sum of the extensions of $g$ and $b$.

We also prove a  version of Theorem \ref{th: introtheorem3} for boundary functions that belong to  $L^p$.  This  result  was previously shown in $\rrn_+$ by Hyt\"{o}nen and Ros\'{e}n \cite{HR18}.  
\begin{theorem}\label{th: introtheorem2}
Let $ \om \in \AR(s)$ for $s \in (0, n]$,  which,  for $s=n$, satisfies the pointwise John condition.  If  $ f\in L^p(\sigma_s) $ with 
$ p\in (1, \infty)$,  there exist $ u :\om \to \R $ and $c_0\in(0,\frac{1}{4}]$ such that, for any $c\in(0,c_0]$,  it holds that 
\begin{itemize}
\item[(i)] $ u \in C^\infty(\om) $,
\item[(ii)] $\displaystyle \| \calN (u) \|_{L^p(\sigma_s)} + \|\calN(\delta_\om\nabla u)\|_{L^p(\sigma_s)}  \lec  \| f \|_{L^p(\sigma_s)} $,
\item[(iii)] $\displaystyle  \| {\calC}_{s,c}(\nabla u) \|_{L^p(\sigma_s)} \lec \| f \|_{L^p(\sigma_s)} $,
\item[(iv)]  For $\sigma_s$-almost every $\xi\in\pom$,\footnote{For the definitions of non-tangential and quasi-non-tangential convergence, see Definition \ref{def:ntconv}.}
$$
\displaystyle \fint_{B(x, \delta_\om(x)/2)} u(y)\,dy \to 
\begin{dcases}
\displaystyle{f(\xi)\,\,\,\textup{non-tangentially}}, &\textup{if}\,\, s<n\\
\displaystyle{f(\xi)\,\,\,\textup{quasi-non-tangentially}}, &\textup{if}\,\, s=n.
\end{dcases}
$$
\end{itemize}
The constant $c_0$ only depends on dimension, the Ahlfors regularity constants, and the corkscrew condition. 
\end{theorem}

\vv
In the proof of Theorem \ref{th: introtheorem2},   we utilize the regularized dyadic extension of the boundary function and establish its $\varepsilon$-approximability in $L^p$ for $p \in (1,\infty)$ (see Subsection \ref{subsec:approx} for the definition). Subsequently,  we demonstrate a trace theorem (see Proposition \ref{lem: non-tangential convergence}) and, finally,  emloy an iteration argument.  This approach follows the general scheme presented in \cite{HR18}, where the $\varepsilon$-approximability in $L^p$ for the dyadic average extension operator in $L^p$ for $p \in (1,\infty)$\footnote{
The concept  of $\ve$-approximability in $L^p$ for $p \in (1,\infty)$ was  introduced  by Hyt\"{o}nen and Ros\'{e}n in \cite{HR18} who 
showed that the dyadic average extension operator as well as  any weak solution to certain elliptic PDEs in $ \R^{n+1}_+ $ are $\ve$-approximable in $L^p$ for every 
$\ve \in (0, 1)$ and $ p\in (1, \infty)$.  The second  part of that work was extended by Hofmann and Tapiola in \cite{HT17} to harmonic functions in $\om = \R^{n+1} \setminus E$ where $E \subset \R^{n+1}$ is a uniformly $n$-rectifiable set.   The converse direction was shown by Bortz and Tapiola in \cite{BT19}. } first appeared.

In fact, this is a key feature of our approach  in both Theorems \ref{th: introtheorem3} and  \ref{th: introtheorem2},   which is  inspired by, but significantly advances Hyt\"{o}nen and Ros\'{e}n's work (even in $\rrn_+$ for the endpoint spaces).  Let us highlight that we encounter significant challenges due to the geometry of our domains. For example, in \cite{HR18}, it is crucially used the separation of variables $(x,t) \in \mathbb{R}^n \times \mathbb{R}_+$ to reduce the case to estimating ${\mathscr C}_n(\partial_t w)$, where $\partial_t w$ stands for the partial derivative in the transversal direction.  In higher co-dimensions, even if $\Omega=\mathbb{R}^3 \setminus \mathbb{R}$, such a reduction does not seem to work, let alone in ``rough" domains. Instead, we resort to multiscale analysis to construct the approximating functions.  An important component is the proof of the packing condition of the top cubes, which was not shown in \cite{HR18} (see Proposition \ref{pro:packing Top}), while  the trace theorem and  the iteration,  also require subtle arguments in our case.

\vv

Our second main goal is to establish Varopoulos extensions of Lipschitz functions with compact supports, which are Lipschitz on $\overline{\Omega}$ and also belong to $\dot{W}^{1,2}(\Omega; \omega_{s})$. This is crucial due to its applications in Boundary Value Problems for elliptic equations (and systems) with merely bounded coefficients (see Theorem \ref{thm:applications}).  
In fact, the second part of the theorem has been employed without explicit verification in connection with Boundary Value Problems (see, e.g., \cite{DaKe} and \cite{MiTa}). To the best of our knowledge, a comprehensive proof of this theorem is not available in the literature. However, it should not be considered folklore since establishing it is far from trivial, at least in our setting, and it has neither been documented nor been known among experts.

\begin{theorem}\label{th: introtheorem4}
Let $ \om \in \AR(s)$ for $s \in (0, n]$.  If $ f\in \Lip_c(\pom) $, there exist a function $ F : \oom \to \R $ and $c_0 \in(0,1/2]$, such that for  any $ c \in (0,c_0]$,  it holds that
\begin{itemize}
\item[(i)] $ F \in  C^\infty(\om)\cap \Lip(\oom) \cap\dot W^{1,2}(\om; \omega_{s}) $, 
\item[(ii)] $ \| \calN(F) \|_{L^p(\sigma_s)}+  \| \calN(\delta_\om  \nabla  {F} )\| _{L^p(\sigma_s)} \lec \| f \|_{L^p(\sigma_s)}$,  for $p\in (1, \infty]$,  
\item[(iii)]  $ \| \calC_{s,c} (\nabla F) \|_{L^p(\sigma_s)}  \lec  \| f \|_{L^p(\sigma_s)}$,
\item[(iv)] $ F|_{\pom} = f $ continuously, 
\item[(v)]   $\| [\calC_s(\delta_\om |\nabla F|^2)]^{1/2}\|_{L^q(\sigma_s)} \lec \|f\|_{L^q(\sigma_s)}$, \quad for \,  $q\in[2, \infty)$.
\end{itemize}
Moreover, there exist a function  $ \bar{F} : \oom \to \R $ and $c_0 \in(0,1/2]$ such that for any $ c \in (0,c_0]$,  it holds that
\begin{itemize}
\item[(i)] $ \bar{F} \in  C^\infty(\om)\cap \Lip(\oom)\cap \dot W^{1,2}(\om; \omega_{s})$, 
\item[(ii)] $ \displaystyle  \sup_{\xi \in \pom} \calN_{\sharp,c}(\bar{F})(\xi)+   \displaystyle  \sup_{x\in \om} \delta_\om(x) | \nabla \bar{F}(x) | \lec 
\| f \|_{\BMO(\sigma_s)} $,   
\item[(iii)]$ \sup_{\xi \in \pom} \calC_{s,c}(\nabla \bar{F})(\xi)  \lec  \| f \|_{\BMO(\sigma_s)}$,
\item[(iv)] $ \bar{F}|_{\pom} = f $ continuously,
\item[(v)] $\displaystyle \sup_{\xi\in\pom}[\calC_s(\delta_\om |\nabla \bar F|^2)(\xi)]^{1/2} \lec \|f\|_{\BMO(\sigma_s)}$.
\end{itemize}
 The constant $c_0$ only depends on dimension, the Ahlfors regularity constants, and the corkscrew condition. 
\end{theorem}

\vv
To prove Theorem \ref{th: introtheorem4}, considering $f \in \Lip_c(\pom)$, we apply Theorems \ref{thm:e-approxbmo} and \ref{thm:e-approxLp} to generate approximating functions for the regularized dyadic extension of $f$. We define the extension to be equal to the approximation everywhere except in a neighborhood of the boundary with a ``width" $\delta>0$, where it is set to equal the regularized dyadic extension. Subsequently, we choose $\delta$ as  $\|f\|_{L^p(\sigma_s)}/\|f\|_{\dot M^{1,p}(\sigma_s)}$ in the case of $p\in (1,\infty)$ (resp. $\|f\|_{L^\infty(\sigma_s)}/\Lip f$ for $p=\infty$) and $\|f\|_{\BMO(\sigma_s)}/\Lip f$ in the case of $\BMO(\sigma_s)$,  achieving the desired estimates.

Notably, we don't construct an extension {\it a priori} and subsequently modify it to obtain the Lipschitz extension; instead, we directly modify the $\ve$-approximator of $\upsilon_f$. This is why we don't need to impose any connectivity condition, as in Theorems \ref{th: introtheorem3} and \ref{th: introtheorem2}. Instead, the existence of the ``trace" is readily ensured by the continuity of $\upsilon_f \in \Lip(\oom)$.

\vv

The theorem presented below is a variation of Theorem \ref{th: introtheorem4} for boundary functions belonging to the Campanato space $\Lambda_\beta(\pom)$ for $\beta \in (0,1)$, as well as the space of H\"older continuous functions $\Lip_\beta(\pom)$. In our setting,  any function in $\Lambda_\beta(\pom)$ coincides $\sigma_s$-almost everywhere with a H\"older continuous function, and the two semi-norms are comparable (refer to Remark \ref{rem:campanato}).

\begin{theorem}\label{th: introtheoremholder}
Let $ \om \in \AR(s)$ for $s \in (0, n]$.  If $ f\in \Lambda_\beta(\pom)$ for $\beta \in (0,1)$, there exists a function $ F : \oom \to \R $ and $c_0 \in(0,1/2]$, such that for  any $ c \in (0,c_0]$,  it holds that
\begin{itemize}
\item[(i)] $ F \in  C^\infty(\om)$,  
\item[(ii)] $ \displaystyle  \sup_{\xi \in \pom} \calN^{(\beta)}_{\sharp,c}({F})(\xi)+    \sup_{x\in \om} \delta_\om(x)^{1-\beta} | \nabla {F}(x) |   \lec 
\| f \|_{\Lambda_\beta(\pom)} $,   
\item[(iii)]$ \displaystyle \sup_{\xi \in \pom} \calC^{(\beta)}_{s,c}(\nabla {F})(\xi) \lec  \| f \|_{\Lambda_\beta(\pom)}$,
\item[(iv)] $\ntlim_{x \to \xi} {F}|_{\pom}(x) = f(\xi)$ for $\sigma_s$-a.e $\xi \in \pom$.
\end{itemize}
Moreover, if $f \in \Lip_\beta(\pom)$, then $F \in \Lip_\beta(\oom)$ and $F|_{\pom} = f $ continuously.   The constant $c_0$ only depends on dimension, the Ahlfors regularity constants, and the corkscrew condition. 

\end{theorem}

Remarkably, this theorem marks the first appearance of such a result, despite $\Lambda_\beta(\pom)$ being a natural endpoint in the interpolation scale that contains $L^p(\sigma_s)$ and $\BMO(\sigma_s)$. Its proof is a lot easier compared to those for $L^p(\sigma_s)$ and $\BMO(\sigma_s)$ boundary functions, as the regularized version of the dyadic extension of the boundary data already satisfies the desired properties, eliminating the need for $\ve$-approximability.

\vv

\begin{remark}\label{rem:vector-valued}
The proof of the existence of extensions of complex-valued boundary functions is exactly the same but for the sake of simplicity we prefer to state and prove our results for real-valued boundary functions.  Moreover, if $\vec f:\pom \to \mathbb{C}^{m}$ with $ \vec f = (f_1, \dots, f_m)$, then its extension is just the vector field $\vec F = (F_1, \dots, F_m)$, where $F_j$ is the extension of $f_j$ for each $j \in \{1,2, \dots, m\}$.
\end{remark}

\vv

Lastly, we use Theorem \ref{th: introtheorem4} to obtain connections between Poisson and Boundary Value Problems (see Definitions \ref{def:DP}, \ref{def:PD}, and \ref{def:PR}) for systems of elliptic equations in divergence form with merely bounded complex-valued coefficients. The estimates derived for solutions of elliptic boundary value problems can be perceived as far-reaching extensions of the aforementioned alternative approach to Fefferman’s duality theorem \cite{FS}.  In particular, we will prove the following.

\begin{theorem}\label{thm:applications}
Let $\om \in \AR(n)$ and $L$ be defined in \eqref{eq:inhomogeneous-equation}. If $L^*$ is its formal adjoint,  then the following hold:
\begin{enumerate}
\item If $(\pr)$ is solvable in $\om$ for some $p>1$,  then  $(\dirprime)$  is also solvable, where $1/p+1/p'=1$.
\item If  $(\wt{\pre}^L_q)$ with $H=0$ is solvable in $\om$ for $q \in[1,2]$, then  both $({\wt{\mathrm{PD}}}^{L^*}_{q'})$  with $H=0$ and  $({\wt{\mathrm{D}}}^{L^*}_{q'})$  are solvable in $\om$.
\item If $(\pd^L_p)$  for  $p\in (1,\infty)$    is solvable in $\om$ with $H=0$,  then the Dirichlet problem $(\mathrm{D}^L_p)$ is also solvable in $\om$.  
\item If $(\wt{\pd}^L_{q})$ is solvable in $\om$ with $H=0$ for some $q\in [2,\infty]$,  then $(\wt{\mathrm{D}}^L_q)$ is also solvable in $\om$.  
\end{enumerate}
\end{theorem}

Recently,  the first named author and Tolsa \cite{MT} constructed  an {\it almost harmonic extension} of functions in the  Haj\l asz Sobolev space $\dot M^{1,p}(\sigma_n)$,  which is  the correct analogue of the $L^p$ version of Varopoulos extension for one ``smoothness level'' up.  To be precise,  it was proved in \cite{MT}  that the Carleson functional defined in \eqref{eq:carleson} of the distributional Laplacian of the almost harmonic extension is in $L^p(\sigma_n)$ and  in \cite{mpt22} that the non-tangential maximal function of its gradient is in $L^p(\sigma_n)$  with norms controlled by the $\dot M^{1,p}(\sigma_n)$ semi-norm of the boundary function. The  almost harmonic extension and  its elliptic analogue (see \cite{mpt22}) were very important since they turned out to be the main  ingredients for the solution of the $L^p$-Regularity problem in domains with interior big pieces of chord-arc domains (\cite[Definition 2.12, p. 892]{AHMMT}) for the Laplace operator \cite{MT} and  for elliptic operators satisfying the  Dahlberg-Kenig-Pipher (DKP) condition  \cite{mpt22} respectively.  This solved a 30 year-old question of Kenig.

The Poisson Dirichlet problem $(\pd^L_p)$ (resp.  the Poisson regularity problem $(\pr)$)  with interior data in suitable Carleson spaces for $p>1$ with scale-invariant estimates for the non-tangential maximal function of the (resp.  gradient of the) solution (see Definitions  \ref{def:PD} and \ref{def:PR}),  was first defined in \cite{mpt22} in order to overcome the obstacle that it is not known if  elliptic  operators satisfying the DKP condition have bounded layer potentials.  In particular,  the authors show that, for such operators, $(\pd^{L^*}_{p'})\Rightarrow (\mathrm{R}^L_p)$ for any $p>1$.  To do so, they use the  almost elliptic extension.  One of their results states  the following equivalences for elliptic equations (not systems) with merely bounded coefficients:  
$$
(\mathrm{D}^{L^*}_{p'}) \Leftrightarrow (\pd^{L^*}_{p'}) \Leftrightarrow (\pr)\qquad \textup{if}\,\,p\in(1,\infty).
$$

The equivalence $(\pd^{L^*}_{p'}) \Leftrightarrow (\pr)$ holds for systems as well,  as the proof in \cite{mpt22} does not utilize tools like the maximum principle or the elliptic measure.  However, the remaining results rely significantly on the connection between the weak-$A_{\infty}$ condition of the elliptic measure and the solvability of $(\mathrm{D}^{L^*}_{p'})$, and thus, they only hold for real equations. Inspired by the use of the almost elliptic extension in \cite{mpt22}, we employ the Varopoulos extension constructed in Theorem \ref{th: introtheorem4} to extend some of these results to elliptic systems. We also obtain endpoint results, which are new even for real equations.

Moreover, we introduce the notions of solvability for the Dirichlet and Poisson problems with square function estimates and for the Poisson regularity problem $(\wt{\mathrm{PR}}^{L}_q)$ with data in Coifman-Meyer-Stein tent spaces and establish connections between these problems.  In the recent work of Gallegos, Tolsa, and the first named author \cite{GaMT23},  the authors,  partially inspired by the introduction of $(\wt{\mathrm{PR}}^{L}_q)$ in the present manuscript,  show extrapolation of solvability of the regularity problem $(\mathrm{R}^L_p)$ (see Definition \ref{def:RP}) and the (modified) Poisson regularity problem $(\wt{\mathrm{PR}}^{L}_q)$ in $\AR(n)$ domains for elliptic operators with real and merely bounded coefficients.  Theorem \ref{thm:applications} played a crucial role in the proof of the extrapolation of $(\wt{\mathrm{PR}}^{L}_q)$ as it was used to demonstrate that  $(\wt{\mathrm{PR}}^{L}_q) \Rightarrow \hm_L \in \textup{weak}$-$A_\infty$.  For a detailed history of work in this area, we refer to the introduction of \cite{mpt22}.

It is worth noting that Theorem \ref{th: introtheorem4} plays a crucial role in a forthcoming work by Poggi and the first named author \cite{MP24}. In that work,  it is established that, for  real equations,  not only $(\mathrm{D}^{L}_{p}) \Leftrightarrow (\pd^{L}_{p})$, but the unique solution of the continuous Dirichlet problem with boundary data $g$  is well approximated by a sequence of solutions of $(\pd^{L}_{p})$ with $H=0$ in the following sense:  If $g\in \Lip_c(\pom)$ and $u$ is the unique solution of $(\mathrm{D}^{L}_{p})$, then there exists a sequence $\{\vec F_j\}_{j \geq 1}\subset \Lip_c(\om; \rrn)$ with the following properties: (i) $\sup_{j \geq 1}\| \calC_{q}(|\vec F_j|)\|_{L^{p}(\sigma)} \lec \|g\|_{L^{p}(\sigma)}$; (ii) If $w_j \in Y^{1,2}_0(\om)$ is the unique solution of $Lw_j =-\dv \vec F_j$, it satisfies $w_j \to u$ in the following topologies: a) in the strong topology of $C_{\loc}^{\alpha}(\om)$ for some $\alpha \in (0,1)$; b) in the strong topology of $W_{\loc}^{1,2}(\om)$; c) in the weak-* topology of $\nn_{r, p}(\om)$, for each $r, p \in (1,\infty)$. Moreover,  the  weak-* convergence cannot be improved to convergence in the weak-topology of $\nn_{r, p}(\om)$ for any $r, p \in (1,\infty)$ unless $g=0$.

\vvv

\subsection{Related results}
While writing this paper, we were informed by Bruno Poggi and Xavier Tolsa that in collaboration with Simon Bortz and Olli Tapiola, they have independently obtained in \cite{BOPT23},  as a corollary of their main result studying $\varepsilon$-approximability of solutions to arbitrary elliptic partial differential equations, a less general version of Theorem \ref{th: introtheorem3} which holds for uniform domains with $n$-Ahlfors regular boundaries such that there is an elliptic measure which is $A_{\infty}$ with respect to surface measure. Their assumptions hold, in particular, for the complement of the $4$-corner Cantor set in $\mathbb R^2$, thus they also show that uniform rectifiability is not a necessary condition in order to construct Varopoulos-type extensions.


\vvv

\section{Preliminaries and notation}\label{sec:preliminaries}

We will write $a \lesssim b$ if there is a constant $C>0$ so that $a\leq Cb$ and $a\approx b$ if $\alpha\lesssim b$ and $b\lesssim a$. If we want to indicate the dependence of $C$ on a certain quantity $s$, we write $a \lesssim_s b$.
For a function space $X$ we denote by $X_c$ the space of all the compactly supported functions in $X$.

\subsection{Preliminaries}
In $\rrn$ and for $s \in [0, n+1]$,  we  denote by $\HH^s$ the $s$-dimensional Hausdorff measure and assume that $\HH^{n+1}$ is normalized so that it coincides with $\mathcal{L}^{n+1}$,  the $(n+1)$-dimensional Lebesgue measure in $\rrn$.  We also denote by $\sigma_s:= \HH^s|_{\pom}$ the ``surface" measure of $\om$. When the dimension is clear from the context we drop the dependence on $s$ and just write $\sigma$.

\begin{definition}\label{def:sregular}
If $s \in (0, n+1]$,  a measure $\mu$ in $\mathbb R^{n+1}$ is called {\it $s$-Ahlfors regular} if there exists some constant $C_0>0$ such that 
$$
C_0^{-1} \,r^s\leq\mu(B(x,r))\leq C_0 \,r^s
$$
for all $x\in \supp\mu$ and $0<r<\diam(\supp \mu)$.  If $E \subset \rrn$ is a closed set we say that $E$ is {\it $s$-Ahlfors regular} if $\HH^s|_E$ is $s$-Ahlfors regular.
\end{definition}

\vv

\subsection{Function spaces}

We write $2^*=\frac{2(n+1)}{n-1}$ and $2_\ast=(2^*)'=\frac{2(n+1)}{n+3}$.  Recall that $C^\infty_c(\om)$ is the space of compactly supported smooth functions in $\om$.  For $p\in[1, \infty)$ and a non-negative function $w \in L^1_{\loc}(\om)$ we define the { homogeneous weighted  Sobolev space} $\dot{W}^{1,p}(\om;w)$ to be the space consisted of  $L^1_{\loc}(\om)$ functions whose weak gradients exist  in $\om$ and are in $L^p(\om;w)$.  We also define the {inhomogeneous weighted  Sobolev space} $W^{1, p}(\om;w)$ to be the space of functions in $L^p(\om;w)$ whose weak derivatives exist in $\om$ and are also in $L^p(\om;w)$ and $W^{1, p}_0(\om;w)$ to be the completion of $C^\infty_c(\om)$ under the norm $\|u\|_{W^{1, p}(\om;w)}:=\|u\|_{L^p(\om;w)}+\|\nabla u\|_{L^p(\om;w)}$.
Finally,  we let $Y^{1, 2}_0(\om;w)$ be the completion of $C^\infty_c(\om)$ under the norm $\|u\|_{Y^{1, 2}(\om;w)}:=\|u\|_{L^{2^*}(\om;w)} + \|\nabla u\|_{L^2(\om;w)}$. 

\vv

Let $\Sigma$ be a metric space equipped with a non-atomic doubling measure 
$ \sigma $ on $ \Sigma $, which means that there is a uniform constant $C_\sigma \geq 1$ such that 
$\sigma( B(x, 2r)) \leq C_\sigma \sigma( B(x, r))$, for all $x\in \Sigma$ and $r>0$.    If $E \subset \Sigma$ is a Borel set such that $0<\sigma(E)<\infty$ and $f\in L_{\loc}^1(\sigma)$,  we denote the average of $f$ over $E$ by
\begin{equation}\label{eq:average}
m_{\sigma,  E} f :=\fint_E f\, d\sigma :=\frac{1}{\sigma(E)}\int_E f d\sigma.
\end{equation}
If $\sigma$ is the Lebesgue measure then we simply write $m_{E} f$. 

\vv

For $\beta \in [0,1)$  we define   $\Lambda_\beta (\pom)$ to be  the {\it Campanato space}   consisting of the functions $f \in L^1_{\loc}(\sigma)$ satisfying
\begin{equation}\label{eq:Campanato}
\| f \|_{\Lambda_\beta(\pom)}:=  \sup_{\substack{x \in \supp \sigma \\ r \in (0,  2\diam \pom)}}\frac{1}{r^\beta} \fint_{B(x,r)} |f(y) - m_{\sigma, B(x,r)} f | d\sigma(y) < \infty.
\end{equation}
Note that $\Lambda_0(\sigma) =\BMO(\sigma)$,  the space of functions of {\it bounded mean $\sigma$-oscillation}.
We also define the space of functions of {\it vanishing mean oscillation}\footnote{$\VMO$ was originally introduced by Sarason in \cite{Sar75}.}, which we denote  by $\VMO(\sigma)$,  to be the closure of the space of continuous functions with compact support $C_c(\Sigma)$  in the $\BMO(\sigma)$ norm.  

\vv
We say that $\alpha$ is a {\it $2$-atom} if there exists $x\in \Sigma$ and $0<r<\diam(\Sigma)$ such that
$$
\supp \alpha \subset B(x,r), \quad \|\alpha\|_{L^2(\sigma)} \lec \sigma(B(x,r))^{-1/2} \quad \textup{and}\,\,\,\int \alpha\,d\sigma=0.
$$
We define the {\it atomic Hardy space} $H^1(\sigma)$ as follows: $f \in H^1(\sigma)$ if there exist a sequence $\lambda_j \in \mathbb C$ and a sequence of $2$-atoms $\alpha_j$ such that $f=\sum_j \lambda_j \alpha_j$ in $L^1(\sigma)$.  We say that $f$ has an atomic decomposition.  $H^1(\sigma)$ is a  subspace of $L^1(\sigma)$ and is a Banach space with norm 
$$
\|f\|_{H^1(\sigma)}:=\inf\Big\{ \sum_j |\lambda_j|: \textup{all atomic decompositions} \,\,f=\sum_j \lambda_j \alpha_j\Big\}.
$$
By the work of Coifmann and Weiss, \cite{CW77}, we have that $(H^1(\sigma))^*=\BMO(\sigma)$ and $(\VMO(\sigma))^*=H^1(\sigma)$.

\vv
For $\beta \in (0,1]$  we  define  $\Lip_\beta(\Sigma)$ to be the space of measurable functions that satisfy
\begin{equation}\label{eq:Lipbeta}
\| f\|_{\Lip_\beta(\Sigma)}:=\sup_{\substack{x,y \in \Sigma \\ x\neq y}} \frac{|f(x)-f(y)|}{|x-y|^\beta}<\infty.
\end{equation}
When $\beta=1$  we simply write $\Lip(\Sigma)$ since it is the space of Lipschitz functions.  If $\Sigma$ is locally compact then  it holds that $\Lip_c(\Sigma)$ is dense in $C_c(\Sigma)$ in the supremum norm.  Therefore, it is easy to see that in that case
$$
\overline{{\Lip}_c(\Sigma)}^{\BMO(\sigma)}=\VMO(\sigma).
$$

\begin{remark}\label{rem:campanato}
By a simple inspection of the proof  of \cite[Theorem 4]{MS79},  it is easy to see that if $\Sigma$ is a metric space equipped with a measure $\sigma$ which is $s$-Ahlfors regular, then, if $\beta \in (0,1)$, for every $f \in \Lambda_\beta(\Sigma)$ there exists $g \in \Lip_\beta(\Sigma)$ such that $f(x)=g(x)$ for  $\sigma$-a.e. $x\in\Sigma$ and $\|f\|_{ \Lambda_\beta(\sigma)} \approx \|f\|_{\Lip_\beta(\Sigma)}$.
\end{remark}
\vv

Following \cite{Ha},  we will introduce the {\it Haj\l asz's Sobolev space} on $\Sigma$.  For a Borel function $f:\Sigma\to\R$, we say that a non-negative Borel function $g:\Sigma \to \R$ is a {\it Haj\l asz upper gradient of  $f$} if  
\begin{equation}\label{eq: hajuppergrad}
 |f(x)-f(y)| \leq |x-y| \,(g(x)+g(y))\quad \mbox{ for $\sigma$-a.e. $x, y \in \Sigma$.} 
\end{equation}
We denote the collection of all the Haj\l asz upper gradients of $f$ by $D(f)$.

For $p>0$, we denote by $\dot{M}^{1,p}(\sigma)$ the space of Borel functions $f$ which have 
a Haj\l asz upper gradient in $L^p(\sigma)$, and we let $M^{1,p}(\sigma)$ be the space of functions $f\in L^p(\sigma)$ which have a Haj\l asz upper gradient in $L^p(\sigma)$, i.e.,  $M^{1,p}(\sigma)=  \dot M^{1,p}(\sigma) \cap L^p(\sigma)$.
We  define the semi-norm (as it annihilates constants)
\begin{equation}\label{eq: hajseminorm}
 \| f \|_{ \dot M^{1, p}(\sigma)} = \inf_{g \in D(f)} \| g\|_{L^p(\sigma)}.
 \end{equation}
If $\Sigma$ is bounded, then we define the norm
 \begin{equation}\label{eqnorm}
 \| f\|_{M^{1,p}(\sigma)} =(\diam \Sigma)^{-1} \|f\|_{L^p(\sigma)} +   \inf_{g \in D(f)} \| g\|_{L^p(\sigma)},
 \end{equation}
 while if $\Sigma$ is unbounded, we consider the space $M^{1, p}(\sigma):=\dot M^{1, p}(\sigma)/\R$.
 Observe that,   from the uniform convexity of $L^p(\sigma)$ for $p \in (1,\infty)$, one easily deduces
that the infimum in the definition of the norm $\|\cdot\|_{\dot M^{1,p}(\Sigma)}$ and $\|\cdot\|_{ M^{1,p}(\Sigma)}$, in \eqref{eq: hajseminorm} and \eqref{eqnorm} respectively, is attained and is unique. We denote by $\nabla_{H,p} f$ the function $g$ which attains the infimum which we will call the {\it least Haj\l asz  upper gradient} of $f$.

\vv

\subsection{Maximal operators and Carleson functionals}

 Set $\delta_\om(\cdot):= \dist(\cdot, \om^c)$, $B^x:=B(x, \delta_\om(x))$, and  $c B^x:=B(x, c\,\delta_\om(x))$,  for $c \in (0,\frac{1}{2}]$.    For $f\in L^1_{\loc}(\mu)$ and $x \in \om$,  we denote
\begin{equation*}
\displaystyle m_{q, c}(f)(x):=
\begin{cases}
\displaystyle  m_{c B^x}(|f|^q)^{1/q} &\textup{if}\,\, 1\leq q<\infty,\\
\displaystyle \sup_{y \in \,c B^x} |f(y)| &\textup{if}\,\, q=\infty,
\end{cases}
\end{equation*}
and 
\begin{equation}\label{def: sharp-maxfunction}
m_{\sharp,c}(f)(x):= m_{\infty, c}(f-m_{cB^x}f)(x).
\end{equation}

\vv

We define the {\it centered Hardy-Littlewood maximal operator} for a function $f\in L_{\loc}^1(\sigma)$ as
$$
\mathcal{M}(f)(x):=\sup_{r>0} m_{\sigma, B(x,r)} (|f|), \qquad x\in \Sigma
$$
while the {\it non-centered Hardy-Littlewood maximal operator} is defined to be
$$
\widetilde{\mathcal{M}}(f)(x) := \sup_{ B \ni x} m_{\sigma, B}(|f|),
$$
where the supremum is taken over all balls $B$ containing $x\in\Sigma$.  
The {\it dyadic Hardy-Littlewood maximal operator} with respect to a dyadic lattice $\DD_\sigma$ on $\Sigma$\footnote{For the construction of dyadic lattices in this setting,  see e.g. subsection \ref{dyadic}.} will be denoted 
$$
\mathcal{M}_{\DD_\sigma}f(x) := \sup_{Q\in\DD_\sigma, Q\ni x} m_{\sigma, Q}(|f|),
$$
and if  the measure is clear from the context,  we will just write $\mathcal{M}_{\DD}f$.  We also set 
\begin{equation}\label{eq:dyadicHLnumbers}
Mf(Q) := \underset{\underset{Q\subset R}{R\in {\DD}_\sigma}}{\sup} m_{\sigma, R}(|f|)
\end{equation}
to be a {\it truncated} version of $\mathcal{M}_{\DD}f(x)$.

\vv
From now on, we assume that $\om \in \AR(s)$ in $\rrn$ and $\sigma=\HH^s|_{\pom}$.
\vv

For $\alpha>0$
and $\xi\in\pom$ we define the {\it cone} with vertex $\xi$ and aperture $\alpha>0$ to be the set 
$$
\gamma_\alpha(\xi):=\{ x\in\om:\,  |x-\xi|<(1+\alpha)\dist(x,\pom)\};
$$
 The {\it non-tangential maximal operator} of a measurable function $f: \om \to \R$ for a fixed aperture $\alpha>0$  by
\begin{equation}\label{def: ntmaxfunction}
\calN_\alpha(f)(\xi):=\sup_{x\in\gamma_\alpha(\xi)}|f(x)|, \qquad \xi\in\pom.
\end{equation}

By a straightforward modification of the  classical proof of Feffermann and Stein \cite[Lemma 1]{FS},  one can show the following.
\begin{lemma}\label{lem: aperture}
For  $\om \in \AR(s)$ and for $s \in (0,n]$, there holds $\|\calN_\alpha(f)\|_{L^p(\sigma)}\approx_{\alpha, \beta,s} \|\calN_\beta(f)\|_{L^p(\sigma)}$ 
for all $\alpha, \beta > 0$ and $p\in (0, \infty)$. 
\end{lemma}

For a fixed aperture $\alpha>0$, $\beta \in [0,1)$ and a  constant $c\in(0,\frac{1}{2}]$, we also define the {\it sharp non-tangential maximal opeartor} applied to a measurable function $f: \om \to \R$ by
\begin{equation}\label{def: sharpntmaxfunction}
\calN^{(\beta)}_{\sharp, \alpha,c}(f)(\xi):=\sup_{x\in \gamma_\alpha(\xi)}  \delta_\om(x)^{-\beta}m_{\sharp,c}(f)(x), \quad \xi \in\pom.
\end{equation}

\vv

Setting  $\hm_s(x):=\delta_\om(x)^{s-n}$  for $x \in\om$, we define the {\it Carleson functional} of a function  $ F \in L^1_{\loc}\big(\om, \hm_s(x)\,dx\big)$  by
 \begin{equation}\label{eq:carleson}
\mathscr{C}^{(\beta)}_s ( F)(\xi):= \sup_{r>0}\frac{1}{r^{s+\beta}} \int_{B(\xi,r)\cap \om} | F(x)| \, \hm_{s}(x)\,dx, \quad \xi \in \pom.
\end{equation}
We define  the {\it modified Carleson functional} of a locally bounded  function  $ F $  by means of
\begin{equation}\label{eq: carleson sup}
\calC^{(\beta)}_{s,c}(F)(\xi):= \mathscr{C}^{(\beta)}_s \big(m_{\infty,c}( F) \big)(\xi), \quad \xi\in\pom.
\end{equation}
For any $q \in [1,\infty)$, the {\it $q$-Carleson functional} of a function $F \in L^q_{\loc}\big(\om, dx\big)$  is defined to be
\begin{equation}\label{eq: q-carleson}
{\mathcal C}^{(\beta)}_{s, q,c} (F)(\xi):= \sup_{r>0}\frac{1}{r^{s+\beta}} \int_{B(\xi,r)\cap \om} m_{q, \sigma, c B^x}(|F|) \, \hm_{s}(x)\,dx, \quad \xi \in \pom.
\end{equation}

\vv

\begin{lemma}\label{lem:changeofconstant}  If  $\om \in \AR(s)$ for $s \in (0,n]$,  $\beta \in [0,1)$,  $q \in [1,\infty)$,  $F \in L^q_{\loc}\big(\om, \,dx\big)$, and $0<c_1<c_2 \leq \frac{1}{2}$, then
\begin{equation}\label{eq:changeofconstant} 
{\mathcal C}^{(\beta)}_{s, q, c_2} (F)(\xi)\lesssim_{c_1, c_2} {\mathcal C}^{(\beta)}_{s, q,c_1} (F)(\xi), \quad \textup{for any}\,\, \xi \in \pom.
\end{equation}
\end{lemma}

\begin{proof}
The case $s=n$ and $\beta=0$ was proved in \cite[Lemma 2.2]{mpt22}, while the  proof in the other cases follow by a routine adaptation of the same arguments.
\end{proof}

\vv

If it is clear from the context and in  view of Lemmas \ref{lem: aperture} and \ref{lem:changeofconstant},  most of the times we  will drop the dependence of $\calN_\alpha$,  $\calN^{(\beta)}_{\sharp,\alpha, c}$,   $\mathcal{ C}^{(\beta)}_{s, q,c}$,    and ${\mathcal C}^{(\beta)}_{s, c}$ on $\alpha$ and $c$, and write $\calN$,  $\calN^{(\beta)}_{\sharp}$,  $\mathcal{ C}^{(\beta)}_{s, q}$,   and $\mathcal{ C}^{(\beta)}_{s}$.  If $s=n$ and/or $\beta=0$, we will drop the dependence on $s$ and/or $\beta$ as well.

\vvv

For $p \in (1,\infty)$  we introduce  the Banach spaces 
\begin{align}\label{eq:Np}
\nnp(\om) &:= \{ w :\om \to \R : w\,\textup{is measurable and }\calN(w) \in L^p(\sigma) \}, \\
\ccpsi(\om) &:= \{w\in L^\infty_{\loc}(\om) : \calC_s( w) \in L^p(\sigma) \}, \label{eq:Cp}
\end{align}
equipped with the respective norms 
$$
\| w \|_{\nnp(\om)} := \| \calN(w) \|_{L^p(\sigma)} \quad \textup{and}\quad  \| \vec F\|_{\ccpsi(\om)} :=  \| \calC_s(|\vec F|) \|_{L^p(\sigma)}.
$$  For $p=\infty$ we define
\begin{align}\label{eq:Ninfty}
\nni(\om) &:= \{ w\in C(\om) : \sup_{\xi \in \pom}\calN(w)(\xi)<\infty \} \\
\ccisi(\om) &:= \{ w\in C(\om) : \sup_{\xi \in\pom} \calC_s(w)(\xi)<\infty \} \label{eq:Cinfty}\\
\nnis(\om)&:=\{ w\in C(\om) :  \sup_{\xi \in \pom} \calN_{\sharp,c}(w)(\xi)<\infty\} \label{eq: Ninfinity sharp}
\end{align} 
and equip them  with the respective norms 
$$
\| w \|_{\nn^\infty(\om)} := \sup_{\xi \in \pom} \calN(w)(\xi) \quad \textup{and}\quad  \|\vec F\|_{\ccisi(\om)}:=\sup_{\xi \in\pom} \calC_s( |\vec F|)(\xi),
$$
  and   the semi-norm 
$$
\|w\|_{\nnis(\om)}:= \sup_{\xi \in \pom} \calN_{\sharp,c}(w)(\xi)=\sup_{x \in \om} m_{\sharp,c}(w)(x).
$$

We also  define the space 
\begin{equation}\label{eq: Nsum}
\nn_{\textup{sum}}(\om):= \{ u\in C^1(\om) : (u,   \delta_\om \nabla u)\in \nn^{\infty}_{\sharp}(\om)\times \nn^{\infty}(\om)\}
\end{equation}
and equip it with the semi-norm 
$$
\|u\|_{\nn_{\textup{sum}(\om)}} := \|u\|_{\nn^\infty_{\sharp}(\om)} + \|  \delta_\om \nabla u\|_{\nn^\infty(\om)}. 
$$
 In Corollary \ref{cor: the sum space is complete},  we will show that $(\nn_{\textup{sum}}(\om), \|\cdot\|_{\nn_{\textup{sum}}(\om)}) $ is sequentially complete.

We set
\begin{align*}\label{eq: Sobolev  C1pinfty}
\cc^{1, p}_{s, \infty}(\om) &:= \{ u\in C^1(\om) : \nabla u \in \cc^p_{s, \infty}(\om) \},\quad\textup{for}\,\, p\in (1, \infty]\\
\end{align*}
and  equip it  with  the  semi-norm $\|u\|_{\cc^{1, p}_{s, \infty}(\om)}:= \|\nabla u\|_{\cc^p_{s, \infty}(\om)}$.

\vv

If $G:\om \to \R$ is a measurable function in $\om$, we define the {\it area functional} of $G$, for a fixed aperture $\alpha>0$ as 
\begin{equation}\label{eq: area functional}
\mathcal{A}^{(\alpha)} G (\xi) := \int_{\gamma_\alpha(\xi)} | G(x) | \delta_\om(x)^{-n}\,dx, \quad \xi\in\pom.
\end{equation}

The following lemma is proved in the Appendix.
\begin{lemma}\label{lem: A less than C}
Let $\om \in \AR(s)$ for $s \in (0,n]$, $ u  \in L^1_{\loc}(\om, \hm_s) $,  $ p \in [1,  \infty) $,  and $\alpha \geq 1$. Then there exists $C\geq 1$ such that for any $\xi \in \pom$ and $ r \in (0, 2\diam(\pom))$, it holds that
\begin{equation}\label{eq:A<Clocal}
\| \mathcal{A}^{(\alpha)}(u {\bf 1}_{B(\xi,r)} ) \|_{L^p(\sigma, B(\xi,r))} \lec r^\beta  \| \mathscr{C}^{(\beta)}_s(u  {\bf 1}_{B(\xi,C r)}) \|_{L^p(\sigma, B(\xi,C r))}.
\end{equation}
If $\beta=0$, it also holds 
\begin{equation}\label{eq:A<Cglobal}
\| \mathcal{A}^{(\alpha)}(u) \|_{L^p(\sigma)} \lec  \| \mathscr{C}_s(u) \|_{L^p(\sigma)}.
\end{equation}
Moreover for $\beta=0$ and $1<p\leq \infty$ we have 
\begin{equation}\label{eq:C<Aglobal}
\|\mathscr{C}_s(u)\|_{L^p(\sigma)} \lec \|\mathcal{A}^{(\alpha)}(u)\|_{L^p(\sigma)}.
\end{equation}
\end{lemma}

\vv

We also introduce the {\it modified non-tangential maximal operator} $\wt \calN_{\alpha, c, r}$ for a given aperture $\alpha>0$, a parameter $c\in(0, 1/2]$ and $r\geq 1$: for any $u\in L^r_{\loc}(\om)$ it is defined as 
$$
\wt \calN_{\alpha, c, r}u(\xi):= \sup_{x\in\gamma_\alpha(\xi)}\Big(\fint_{B(x, c\delta_\om(x))}|u(y)|^r\, dy\Big)^{1/r}, \quad \xi\in\pom.
$$
The $L^p$-norms of these non-tangential maximal functions with different aperture $\alpha$ or averaging parameter $c$  are comparable (see \cite[Lemma 2.1]{mpt22}), and for ease the notation we will just write $\wt\calN_r=\wt\calN_{\alpha, c, r}$ when we do not need to specify neither $\alpha$ nor $c$.

For any $q\geq 1$ and $p>1$,  we define the Banach spaces
$$
\cc_{s, q, p}(\om):=\{ H\in L^q_{\loc}(\om) : \calC_{s,q}(H) \in L^p(\sigma) \},
$$
with norm $\|H\|_{\cc_{s,q, p}}=\|\calC_{s.q}(H)\|_{L^p(\sigma)}$, and for $r\in[1, \infty]$, $p>1$ let
$$
\nn_{r, p}(\om):= \{ u\in L^r_{\loc}(\om): \wt \calN_r(u)\in L^p(\sigma)\},
$$
(where we identify $\wt\calN_\infty=\calN$) with norm $\|u\|_{\nn_{r, p}(\om)}=\| \wt \calN_r(u)\|_{L^p(\sigma)}$.  By (the proof of) \cite[Proposition 2.4]{mpt22},    it holds that if either $\om$ is bounded or $\pom$ is unbounded, $ \nn_{q, p}(\om)=(\cc_{s, q', p'}(\om))^*$. When $s=n$ we drop the subscript $s$ from $\cc_{s, q, p}$.

\vv


If $\om\in\AR(s)$, we define the  {\it tent spaces} 
\begin{equation}\label{eq: space T2infty}
T^\infty_{s,2}(\om) :=\{ f \in L^2_{\loc}(\om) :  \mathscr{C}_s(f^2\, \delta_\om^{-1})\in L^\infty(\sigma) \}
\end{equation} 
and 
\begin{equation}\label{eq: space T2p}
T^p_2(\om) := \{g\in L^2_{\loc}(\om) : \left(\mathcal{A}(g^2\, \delta_\om^{-1}) \right)^{1/2} \in L^p(\sigma) \}, \,\, \textup{for} \,\, p\in (0, \infty),
\end{equation}
and we equip them with the respective norms  
$$
\| f\|_{T^\infty_{s,2}(\om)}= \big\|  \mathscr{C}_s\big(f^2\, \delta_\om^{-1}\big)^{1/2}\big\|_{L^\infty(\sigma)} \quad  \text{and} \quad \| g\|_{T^p_2(\om)}=\big\| \left(\mathcal{A}(g^2\, \delta_\om^{-1} \right)^{1/2}\big\|_{L^p(\sigma)}.
$$ 
When $s=n$,  we drop the subscript $s$ from $T^\infty_{2,s}$ and just write $T ^\infty_2$. It is not hard to see that $L^2_c(\om)$ is a dense subspace of $T^p_2(\om)$ for any $p\in[1, \infty)$.

The tent spaces were first  introduced and studied in \cite{CMS85} in the upper-half space $\R^{n+1}_+$ and was 
extended to $\AR(n)$ domains in \cite{MPT}\footnote{Note that the results  are stated in chord-arc domains but an easy inspection of the proofs in \cite{MPT} reveals that neither the Harnack chain condition nor the exterior corkscrew condition are necessary. }.   An important result in this theory is the duality between tent spaces.  Namely,  if $\om \in \AR(n)$,   the pairing
$$
\langle f, g \rangle= \int_\om f(x) \,g(x)\,\frac{dx}{\delta_\om(x)}
$$
realizes $T^\infty_2(\om)$ as the Banach dual of $T^1_2(\om)$.  Moreover,  for $p \in (1,\infty)$,  the same pairing realizes  $T^{p'}_2(\om)$ as the Banach dual of $T^p_2(\om)$, where $1/p+1/{p'}=1$.  In this generality,  this follows from the proof of Theorem 4.2 and  Remarks 4.3 and 4.4 in \cite{MPT}. By an  inspection of the proofs, one can easily show that if $\om\in\AR(s)$ for $s\in(0, n]$,  then  the pairing
$$
\langle f, g \rangle:= \int_\om f(x) \,g(x)\,\frac{dx}{\delta_\om(x)^{n+1-s}}
$$
realizes $T^\infty_{s,2}(\om)$ as the Banach dual of $T^1_2(\om)$.  Analogously,  for $p \in (1,\infty)$,  the same pairing realizes  $T^{p'}_2(\om)$ as the Banach dual of $T^p_2(\om)$, where $1/p+1/{p'}=1$.  
\vvv

\subsection{$L^p$ and uniform $\ve$-approximation}\label{subsec:approx}

 If $\ve>0$,  we say  that a function $w$ is  {\it uniformly $\ve$-approximable},   if  there exist a constant $ C_{\ve} > 0 $ and a function $ \vp = \vp^{\ve} \in C^\infty(\om)$ such that 
\begin{equation}\label{eq:e-approx-sup-1}
\sup_{x \in \om} |w(x) - \vp(x) | + \sup_{x \in \om} \,\delta_\om(x)|\nabla(w-\vp)(x)| \lec \ve 
\end{equation}
and
\begin{equation}\label{eq:e-approx-sup-2}
\sup_{\xi \in \pom} \mathcal{C}_{s,c} (\nabla \vphi)(\xi) \lec \ve^{-2},
\end{equation}
where $\delta_\om(\cdot)=\dist(x\cdot, \om^c)$ and the implicit constants are independent of $\ve$.

\vv

 If $\om \in \AR(s)$,  then,  for fixed $p \in (1, \infty)$,  we  say that a function $w$ is  $\ve$-{\it approximable in $ L^p(\sigma_s)$} if there exists a function $\vp = \vp^{\ve} \in C^\infty(\om)$  such that 
\begin{equation}\label{eq:intro-e-approx1}
\| \calN (w - \vp) \|_{L^p(\sigma_s)} +\| \calN (\delta_\om \nabla(w - \vp) )\|_{L^p(\sigma_s)} \lec_p  \ve  \, \| \calN w \|_{L^p(\sigma_s)}
\end{equation}
and 
\begin{equation}\label{eq:intro-e-approx2}
\| \mathcal{C}_{s,c} (\nabla \vp) \|_{L^p(\sigma_s)} \lec_{p}   \ve^{-2}  \, \| \calN w \|_{L^p(\sigma_s)}.
\end{equation}

\vvv

\subsection{Elliptic systems and Boundary value problems}

In this section we consider domains $\om \in \AR(n)$, $n\geq 1$.  Let $L$ be an elliptic operator acting on column vector-fields $u=(u^1, \dots, u^m)^T$, where $u^\beta:\om \to \mathbb{C}$ for $\beta=1, 2, \dots, m$,  defined as follows:
\begin{equation}\label{eq:inhomogeneous-equation}
L u(x) = - \sum_{i,j=1}^{n+1} \partial_i(A_{ij}(x) \partial_j u(x))=- \sum_{\alpha, \beta=1}^m \sum_{i, j=1}^{n+1}  \partial_i \big(a^{\alpha\beta}_{ij}(x) \, \partial_j u^\beta(x)\big),
\end{equation}
where  $\partial_i=\frac{\partial}{\partial x_i}$, $1\leq i \leq n+1$ and 
$A_{ij}$ are $m \times m$ matrix-valued functions on $\rrn$ with entries $a^{\alpha\beta}_{ij}:\om \to \bC$,  $\alpha, \beta \in \{1, \dots, m\}$ for which there exists $\lambda \in (0,1]$ such that
\begin{align}\label{eq: matrix A}
\sum_{\alpha,\beta=1}^m  \sum_{i, j=1}^{n+1} |a^{\alpha\beta}_{ij}(x)|^2 &\leq \lambda^{-2}, \quad \textup{for a.e.} \, x \in \om\,\, \textup{and}\\
\mathfrak{Re} \sum_{\alpha,\beta=1}^m  \sum_{i, j=1}^{n+1} a^{\alpha \beta}_{ij}(x)\xi_j^\beta \overline{\xi_i^\alpha} 
&\geq \lambda   \sum_{\alpha = 1}^m \sum_{i=1}^{n+1} |\xi_i^\alpha|^2\quad \textup{for a.e.} \, x \in \om.\label{eq: accretivity condition}
\end{align}
For $m=1$ and $a_{ij}:\om \to \R$,  \eqref{eq: accretivity condition} amounts to the standard accretivity condition
\begin{equation}\label{eq: ellipticity cond}
\lambda |\xi|^2 \leq \sum_{i, j=1}^{n+1} a_{ij}(x) \xi_i \dot \xi_j , \quad  \textup{for a.e.} \, x\in\om,  \,\,\, \xi\in\R^{n+1}.
\end{equation}
Notice that the $\alpha$-th component of the column vector  $Lu$ coincides with 
\begin{align}\label{eq: system of elliptic equations}
(L u)^\alpha(x):=-&\sum_{\beta=1}^m \sum_{i, j=1}^{n+1}  \partial_i \big(a^{\alpha\beta}_{ij}(x) \, \partial_j u^\beta(x)\big)
\end{align}
We also define its adjoint operator of $L$ by 
$$
L^* u(x) =- \sum_{\alpha, \beta=1}^m \sum_{i, j=1}^{n+1}  \partial_i \big(\overline{a^{\beta\alpha}_{ji}}(x) \, \partial_j u^\beta(x)\big)),
$$
i.e.,  $L^* = -\dv A^* \nabla$, where $A^*=(\overline{A}_{ij})^T$ or equivalently $(a^{\alpha\beta}_{ij})^*=\overline{a^{\beta\alpha}_{ji}}$.

We assume that   $H: \om \to \bC^m$ is given by $H= (H^1, \dots, H^m)$  and ${\bf{\Xi}}:\om \to \bC^{m(n+1)}$ is given by ${\bf \Xi}:=(\vec{\Xi}^1, \dots, \vec{\Xi}^m)$, where $\vec{\Xi}^\alpha:\om \to  \bC^{n+1}$ and $\vec{\Xi}^\alpha=(\Xi^\alpha_1,\dots, \Xi^\alpha_{n+1})$ for $\alpha=1,\dots m$.  We are interested in solutions of the inhomogeneous equation $Lu=-\dv {\bf \Xi}+H$ in $\om$  in the  sense
\begin{align}\label{eq: system of elliptic equations}
L u(x) = - \sum_{\alpha=1}^m \sum_{i=1}^{n+1} \partial_i  \, \Xi_i^\alpha(x)+ \sum_{\alpha=1}^m H^\alpha(x), \quad \textup{for}\,\, x \in \om.\notag
\end{align}

 \vv

For  $H\in L^{2_\ast}_{\loc}(\om;\bC^{m})$ and ${\bf \Xi} \in L^2_{\loc}(\om;\bC^{m(n+1)})$, we say that the vector field $w\in W^{1, 2}_{\loc}(\om;\bC^{m})$ solves 
$Lw=H-\div {\bf \Xi}$ in the {\it weak sense}, or that $w$ is a {\it weak solution} to the equation $Lw=H-\div{\bf \Xi}$, if for any $\Phi\in C^\infty_c(\om;\bC^{m})$, we have that 
\begin{equation}\label{eq:weaksolution}
\sum_{\alpha, \beta=1}^m  \sum_{i, j=1}^{n+1}  \int_\om a^{\alpha \beta}_{ij}(x)\partial_j w^\beta \, \overline{\partial_i \Phi^\alpha}  =      \sum_{\alpha=1}^m \sum_{i=1}^{n+1}  \int_\om  {\Xi_i^\alpha}   \,\overline{\partial_i \Phi^\alpha} + \sum_{\alpha=1}^m  \int_\om H^\alpha \, \overline{ \Phi^\alpha}.
\end{equation}

\vv

We say that the {\it variational Poisson-Dirichlet problem}  for $L$ is solvable in $\om$ if for every $H \in L^{2_*}(\om ; \bC^m)$ and ${\bf \Xi} \in L^2(\om; \bC^{m(n+1)})$ there exists $u \in W^{1,2}_{\loc}(\om; \bC^m)$ such that
\begin{equation}\label{eq:variational Dirichlet problem}
(\vpd)=
\begin{dcases}
Lu=-\div A\nabla u= -\dv {\bf \Xi} + H  &\textup{weakly in} \,\, \om \\
u \in Y^{1,2}_0(\om).
\end{dcases}
\end{equation}
By Lax-Milgram's theorem,  this problem is always solvable and its  solution is unique.  Moreover, 
$$
\|  u \|_{Y^{1,2}_0(\om)} \lec \|H\|_{L^{2_*}(\om ; \bC^m)} + \| {\bf \Xi}\|_{L^2(\om; \bC^{m(n+1)})}.
$$

\vv

We say that the {\it  variational Dirichlet problem}  for $L$ is solvable in $\om$  if,  for every $\vphi \in \Lip(\pom;\bC^m)$ and  $\Phi \in  \dot W^{1,2}(\om; \bC^m)\cap \Lip(\oom;\bC^m)$ satisfying $\Phi|_\pom=\vphi$,  there exists $w \in \dot W^{1,2}(\om; \bC^m)$  such that
\begin{equation}\label{eq: Dirichlet problem}
(\vd)=
\begin{dcases}
Lw=-\div A\nabla w=0 &\textup{weakly in} \,\, \om \\
w-\Phi \in  Y_0^{1,2}(\om).\quad &\textup{on} \,\, \pom.
\end{dcases}
\end{equation}

If $u$ is the solution of \eqref{eq:variational Dirichlet problem} for ${\bf \Xi}= -A \nabla \Phi \in L^2(\om; \bC^{m(n+1)}))$ and $H=0$,  then,  it is easy to see that  $w=u+\Phi$ is the solution of  \eqref{eq: Dirichlet problem}.

Since the set of compactly supported Lipschitz functions on $\pom$, is dense in the set of compactly supported continuous functions on $\pom$,  we can extend the definition of the Dirichlet problem to $C_c(\pom)$.  Namely,  for any $\vphi \in C_c(\pom;\bC^m)$ and  $\Phi \in  \dot W^{1,2}(\om; \bC^m)\cap C(\oom;\bC^m)$ satisfying $\Phi|_\pom=\vphi$,  there exists $w \in \dot W^{1,2}(\om; \bC^m)$  satisfying \eqref{eq: Dirichlet problem}.

\vv

We set
$$
Y^q(\sigma)=
\begin{cases}
L^q(\sigma) & \textup{if}\,\, q \in (1, \infty)\\
\BMO(\sigma) & \textup{if}\,\, q =\infty,
\end{cases}
$$

\begin{definition}\label{def:DP}
For any $q \in (1, \infty)$,  we say that the {\it Dirichlet problem} with $L^{q}$ boundary data is solvable  for  $L$ in $\om$ (write $(\dqr)$ is solvable in $\om$), if there exists  $C>0$ so that for each 
$g\in \Lip_c(\pom)$, the solution $u$ of \eqref{eq: Dirichlet problem} for $L$ with boundary data $g$ satisfies the estimate 
\begin{equation}\label{eq: Dirichlet solvability estimate}
\|\wt \calN_{2^*} (u)\|_{L^{q}(\sigma)} \leq C \|g\|_{L^{q}(\sigma)},
\end{equation}
where $2^*:=\frac{2(n-1)}{n+1}$.  Similarly,   we will say that $(\tdqr)$  is solvable  for  $L$ in $\om$ if there exists  $C >0$ so that for each  $g\in \Lip_c(\pom)$, the solution $u$ of \eqref{eq: Dirichlet problem} for $L$  with boundary data $g$ satisfies the estimate 
\begin{equation}\label{eq: Dirichlet solvability area}
\| \delta_\om \nabla u \|_{T^q_2(\om)}\leq C \|g\|_{Y^{q}(\sigma)}.
\end{equation}
\end{definition}

\vv

\begin{definition}\label{def:RP}
For any $p\in (0, \infty)$, we say that the (homogeneous) {\it Dirichlet regularity problem} (or just {\it regularity problem}) with boundary data in $\dot{M}^{1, p}(\sigma)$  is solvable for $L$  in $\om$ (write $(\rpr)$ is solvable in $\om$), if there exists $C>0$ so that for each $f\in\Lip_c(\pom)$, the solution $u$ of \eqref{eq: Dirichlet problem} with boundary data $f$ satisfies the estimate 
\begin{equation}\label{eq: Dirichlet regularity estimate}
\|\wt \calN_2(\nabla u)\|_{L^p(\sigma)} \leq C \|f\|_{\dot{M}^{1, p}(\sigma)}. 
\end{equation}
\end{definition}

\vv

Following \cite{mpt22},  we introduce the Poisson-regularity problem with data in $\cc_{q, p}(\om)$.

\begin{definition}\label{def:PD}
For any $p\in(1, \infty)$, we say that the {\it Poisson-Dirichlet problem} $(\pd^L_p)$ is solvable in $\om$ if there exists $C>0$ so that for each $H\in L^\infty_c(\om; \bC^m)$ and ${\bf \Xi} \in L^\infty_c(\om; \bC^{m(n+1)})$, the  solution $v$ of the problem \eqref{eq:variational Dirichlet problem} satisfies the estimate 
\begin{equation}\label{eq: poisson d regularity}
\|\wt \calN_{2^*} (u)\|_{L^p(\sigma)} \leq C \left( \| \calC_{2_\ast}(\delta_\om H)\|_{L^p(\sigma)} + \|\calC_2( {\bf \Xi})\|_{L^p(\sigma)}\right).
\end{equation}
Similarly,  for $p\in (1,\infty[$,  we say that the {\it Poisson-Dirichlet problem} $(\wt \pd_p^L)$ for $H=0$ is solvable in $\om$ if there exists $C>0$ so that for each ${\bf \Xi} \in L^\infty_c(\om; \bC^{m(n+1)})$,  the  solution $u$ of the problem \eqref{eq:variational Dirichlet problem} for $H=0$ satisfies the estimate 
\begin{equation}\label{eq:end-poisson regularity}
\|\delta_\om  \nabla u \|_{T^p_2(\om)} \leq C\, \| \calC_2({\bf \Xi})\|_{L^p(\pom)}.
\end{equation}
\end{definition}

\vv

\begin{definition}\label{def:PR}
For any $p\in(1, \infty)$, we say that the {\it Poisson-regularity problem} $(\pr)$ is solvable in $\om$ if there exists $C>0$ so that for each $H\in L^\infty_c(\om; \bC^m)$ and ${\bf \Xi} \in L^\infty_c(\om; \bC^{m(n+1)})$, the  solution $v$ of the problem \eqref{eq:variational Dirichlet problem} satisfies the estimate 
\begin{equation}\label{eq: poisson regularity}
\|\wt \calN_2(\nabla v)\|_{L^p(\sigma)} \leq C \left( \| \calC_{2_\ast}(H)\|_{L^p(\sigma)} + \|\calC_2(| {\bf \Xi}| / \delta_\om)\|_{L^p(\sigma)}\right).
\end{equation}
Similarly,  for $p\in [1,\infty)$, we say that the {\it Poisson-regularity problem} $(\wt \pre_p^L)$ for $H=0$ is solvable in $\om$ if there exists $C>0$ so that for each ${\bf \Xi} \in L_c^2(\om; \bC^{m(n+1)})$, the  solution $v$ of the problem \eqref{eq:variational Dirichlet problem} for $H=0$ satisfies the estimate 
\begin{equation}\label{eq: poisson regularity}
\|\wt \calN_2(\nabla v)\|_{L^p(\sigma)} \leq C\,  \| {\bf \Xi}\|_{T^p_2(\om)}.
\end{equation}
\end{definition}

\vvv

\subsection{Geometry of domains}

Following Jerison and Kenig \cite{JK82}, we introduce the  notions of corkscrew and Harnack chain conditions.
\begin{definition}\label{def:corkscrew} We say that an open set $\om\subset\R^{n+1}$ satisfies the $c$-{\it corkscrew condition} for $c \in (0,1/2)$, if for every ball $B(\xi,r)$ with $\xi\in\pom$ and $0<r<\diam(\om)$,  there exists a point $x \in \om \cap B(\xi,r)$ such that $B(x, c r)\subset \om\cap B(\xi,r)$.
\end{definition}

\begin{definition} \label{def:harnackchain} 
Given two points $x,x' \in \Omega$, and a pair of numbers $M,N\geq1$, 
an $(M,N)$-{\it Harnack Chain connecting $x$ to $x'$},  is a chain of
open balls
$B_1,\dots,B_N \subset \Omega$, 
with $x\in B_1,\, x'\in B_N,$ $B_k\cap B_{k+1}\neq \emptyset$
and $M^{-1}\diam (B_k) \leq \dist (B_k,\partial\Omega)\leq M\diam (B_k).$
We say that $\Omega$ satisfies the {\it Harnack Chain condition}
if there is a uniform constant $M$ such that for any two points $x,x'\in\om$,
there is an $(M,N)$-Harnack Chain connecting them, with $N$ depending only on $M$ and the ratio
$|x-x'|/\left(\min\big(\delta_{\Omega}(x),\delta_{\Omega}(x')\big)\right)$.  
\end{definition}
It is not hard to see that if $E\subset\R^{n+1}$ is  $s$-Ahlfors regular for $s \in (0,n]$,  then $\R^{n+1}\setminus E$ satisfies the 
$c$-corkscrew condition for some $c \in (0,1/2)$ depending only on the Ahlfors regularity constants.  In the case that $s<n$,  $\R^{n+1}\setminus E$ satisfies the Harnack chain condition as well (see  \cite[Lemma 2.2]{DFM}).

\begin{definition}\label{def: good curve}
A connected rectifiable curve $\gamma:[0,\ell] \to \oom$ connecting $\xi \in \pom$ and $x \in \om$, parametrized by the arc-length $s \in [0,\ell]$ and such that  $\gamma(0)=\xi$ and $\gamma(\ell)=x$, is called {\it $\lambda$-good curve} (or {\it $\lambda$-carrot path}), for some  $\lambda \in (0,1]$,  if $\gamma\setminus\{\xi\} \subset \om$   and   $\delta_\om(\gamma(s)) > \lambda \, s$,  for every $s\in (0, \ell]$.
\end{definition}

\begin{definition}\label{def: pwJohn}
An open set $\Omega \subset \R^{n+1}$ is said  to satisfy the {\it pointwise John condition},  if there exists a constant $\theta \in (0, 1) $ such that  for $\sigma_n\textup{-a.e. } \xi \in\pom$, there exist $x_\xi\in\om$ and $r_\xi > 0$ such that $x_\xi \in B(\xi, 2 r_\xi)$ and $\delta_\om(x_\xi) \geq \theta r_\xi$,  and also  there exists a $\theta$-good curve $\gamma_\xi \subset \om \cap B(\xi, 2 r_\xi)$  connecting the points $\xi$ and $x_\xi$ such that $\ell(\gamma_\xi) \leq \theta^{-1} r_\xi$. We will write that $ \xi \in \JC(\theta) $ if the pointwise John condition holds for the point $\xi$ with constant $\theta\in(0, 1)$. 
\end{definition}

\begin{remark}
Any domain $\om \in \AR(n)$ with $n$-rectifable boundary satisfies the pointwise John condition. 
\end{remark}

Following \cite{HMT} we also introduce the notion of  local John domains, which are also examples of domains satisfying the pointwise John condition.
\begin{definition}\label{def: local John}
An open set $\Omega  \subset \R^{n+1}$ is said  to satisfy the {\it  local John condition} if there is $\theta\in(0,1)$
such that the following holds: For all $x\in\pom$ and $r\in(0,2\diam(\Omega))$ 
there is $y\in B(x,r)\cap \Omega$ such that $B(y,\theta r)\subset \Omega$ with the property that for all
$z\in B(x,r)\cap\pom$ one can find a rectifiable {path} $\gamma_z:[0,1]\to\overline \Omega$ with length
at most $\theta^{-1}|x-y|$ such that 
$$\gamma_z(0)=z,\qquad \gamma_z(1)=y,\qquad \dist(\gamma_z(t),\pom) \geq \theta\,|\gamma_z(t)-z|
\quad\mbox{for all $t\in [0,1].$}$$
\end{definition}
If $\om \in \AR(s)$ for $0<s<n$,  then it clearly satisfies the local John condition as it satisfies the corkscrew and the  Harnack chain conditions. If $s=n$,  any semi-uniform  (and thus any uniform) domain has the local John condition. 

\vv
\begin{definition}\label{def:ntconv} Let $\om$ be a corkscrew domain, $F: \om \to \R$, and $ f : \pom \to \R $.   We say that $F$  {\it converges non-tangentially} to $f$ at $\xi \in \pom$ (write $F \to f \textup{ n.t.  at } \xi$)  if there exists $\alpha>0$ such that for every sequence $x_k \in \gamma_\alpha(\xi)$ for which $x_k \to \xi$ as $k \to \infty$, it holds that $F(x_k) \to f(\xi)$ as $k \to \infty$.   We will also write 
$$
\ntlim_{x \rightarrow \xi}F(x)=f(\xi).
$$
 We will  say  that $F \to f$ {\it  quasi-non-tangentially} at $\xi \in \pom$ (write $F \to f \textup{ q.n.t.  at } \xi$) if there exist $r_\xi>0$,  a corkscrew point $x_\xi \in \om \cap B(\xi,2r_\xi)$, and a $\theta$-good curve $\gamma_\xi \in B(\xi, 2 r_\xi)$ connecting $\xi$ and $x_\xi$,  such that  for any $x_k \in \gamma_\xi$ converging to $\xi$ as $k \to \infty$,  it holds that $ \lim_{k \to \infty} F(x_k)= f(\xi)$. We will also write 
$$
\qntlim_{x \rightarrow \xi}F(x)=f(\xi).
$$
\end{definition}

\vv

\subsection{Dyadic lattices}\label{dyadic}
Given an $s$-Ahlfors-regular measure $\mu$ in $\R^{n+1}$, we consider 
the dyadic lattice of “cubes" built by David and Semmes in \cite[Chapter 3 of Part I]{DS2}. 
The properties satisfied by $\DD_\mu$ are the following. 
Assume first, for simplicity, that $\diam(\supp\mu)=\infty$). Then for each $j\in\mathbb Z$ there exists a family 
$\DD_{\mu,j}$ of Borel subsets of $\supp\mu$ (the dyadic cubes of the $j$-th generation) such that:
\begin{itemize}
\item[$(a)$] each $\DD_{\mu,j}$ is a partition of $\supp\mu$, i.e.\ $\supp\mu=\bigcup_{Q\in \DD_{\mu,j}} Q$ and 
$Q\cap Q'=\varnothing$ whenever $Q,Q'\in\DD_{\mu,j}$ and
$Q\neq Q'$;
\item[$(b)$] if $Q\in\DD_{\mu,j}$ and $Q'\in\DD_{\mu,k}$ with $k\leq j$, then either $Q\subset Q'$ or $Q\cap Q'=\varnothing$;
\item[$(c)$] for all $j\in\mathbb Z$ and $Q\in\DD_{\mu,j}$, we have $2^{-j}\lec\diam(Q)\leq2^{-j}$ and $\mu(Q)\approx 2^{-j s}$;
\item[$(d)$] there exists $C>0$ such that, for all $j\in\mathbb Z$, $Q\in\DD_{\mu,j}$, and $0<\tau<1$,
\begin{equation}\label{small boundary condition}
\begin{split}
\mu\big(\{x\in Q:\, &\dist(x,\supp\mu\setminus Q)\leq\tau2^{-j}\}\big)\\&+\mu\big(\{x\in \supp\mu\setminus Q:\, \dist(x,Q)\leq\tau2^{-j}\}\big)\leq C\tau^{1/C}2^{-j s}.
\end{split}
\end{equation}
This property is usually called the {\em small boundaries condition}.
From (\ref{small boundary condition}), it follows that there is a point $x_Q\in Q$ (the center of $Q$) such that $\dist(x_Q,\supp\mu\setminus Q)\gtrsim 2^{-j}$ (see \cite[Lemma 3.5 of Part I]{DS2}).
\end{itemize}
We set 
$$\DD_\mu:=\bigcup_{j\in\mathbb Z}\DD_{\mu,j}.$$
In case that $\diam(\supp\mu)<\infty$, the families $\DD_{\mu,j}$ are only defined for $j\geq j_0$, with
$2^{-j_0}\approx \diam(\supp\mu)$, and the same properties above hold for 
$\mathcal{D}_\mu:=\bigcup_{j\geq j_0}\DD_{\mu,j}$.

Given a cube $Q\in\mathcal{D}_{\mu,j}$, we say that its {\it side-length }is $2^{-j}$, and we denote it by $\ell(Q)$. Notice that $\diam(Q)\leq\ell(Q)$. 
We also denote 
\begin{equation}\label{defbq}
B(Q):=B(x_Q,c_1\ell(Q)),\qquad B_Q = B(x_Q,\ell(Q)),
\end{equation}
where $c_1>0$ is some fix constant so that $B(Q)\cap\supp\mu\subset Q$, for all $Q\in\mathcal{D}_\mu$. Clearly, we have $Q\subset B_Q$.  For $\lambda>1$, we write
$$\lambda Q = \bigl\{x\in \supp\mu:\, \dist(x,Q)\leq (\lambda-1)\,\ell(Q)\bigr\}.$$


The side-length of a “true cube'' $P\subset\R^{n+1}$ is also denoted by $\ell(P)$. On the other hand, given a ball $B\subset\R^{n+1}$, its radius is denoted by $r(B)$. For $\lambda>0$, the ball $\lambda B$ is the ball concentric with $B$ with radius $\lambda\,r(B)$.

\vv

\subsection{The Whitney decomposition}\label{subsec: Whitney}
Recall that a domain is a connected open set. In the whole paper, $\Omega$ will be an open set in $\R^{n+1}$, with 
$n\geq 1$. We will denote the $n$-Hausdorff measure on $\partial\Omega$ by $\sigma$. 

We 
consider the following Whitney decomposition of $\Omega$ (assuming $\Omega\neq \R^{n+1}$): we have a family 
$\WW(\Omega)$ of dyadic cubes in $\mathbb R^{n+1}$ with disjoint interiors such that
$$
\bigcup_{P\in\mathcal{W}(\Omega)} P = \Omega,
$$
and moreover there are
some constants $\Lambda>20$ and $D_0\geq1$ such the following holds for every $P \in\WW(\Omega)$:
\begin{itemize}
\item[(i)] $10P \subset \Omega$;
\item[(ii)] $\Lambda P \cap \pom \neq \varnothing$;
\item[(iii)] there are at most $D_0$ cubes $P'\in\WW(\Omega)$
such that $10P \cap 10P' \neq \varnothing$. Further, for such cubes $P'$, we have $\frac12\ell(P')\leq \ell(P)\leq 2\ell(P')$.
\end{itemize}
From the properties (i) and (ii) it is clear that $\dist(P,\pom)\approx\ell(P)$ and so there exists $\Lambda'>20$ such that
\begin{equation}\label{eq:delta<ellP}
\dist(x, \pom) \leq \Lambda' \ell(P), \quad \textup{for every}\,\, x\in P.
\end{equation} We assume that
the Whitney cubes are small enough so that
\begin{equation}\label{eqeq29}
\diam(P)< \frac1{20}\,\dist(P,\pom).
\end{equation}
The arguments to construct a Whitney decomposition satisfying the properties above are
standard and can be found for example in \cite{Stein}.

Suppose that $\pom$ is $s$-Ahlfors-regular and consider the dyadic lattice $\DD_\sigma$ defined above.
Then, for each Whitney $P\in \WW(\Omega)$ there is some cube $Q\in\DD_\sigma$ such that 
\begin{equation}\label{definitionb(P)}
\ell(Q)=\ell(P)\quad \textup{and} \quad \dist(P,Q)\approx \ell(Q),
\end{equation}
 with the implicit constant depending on the parameters of $\DD_\sigma$ and on the Whitney decomposition.  For every $P\in\WW(\om)$ there is a uniformly bounded number of cubes $Q\in \DD_\sigma$ depending on $n$ and the $s$-Ahlfors regularity of $\pom$ that satisfy \eqref{definitionb(P)}.  To each $P\in\WW(\om)$  we associate precisely on such cube  $Q\in \DD_\sigma$ and  denote it by $Q=b(P)$.  We say that $Q$ is the {\it boundary cube} of $P$. 

Conversely,
given $Q\in\DD_\sigma$, we let 
\begin{equation}\label{eqwq00}
w(Q) = \bigcup_{P\in\WW(\Omega):Q=b(P)} P.
\end{equation}
In the case of $n$-Ahlfors regular boundary, it is immediate to check that $w(Q)$ is made up at most of a {uniformly} bounded number of cubes $P$, but it may happen that $w(Q)=\varnothing$.

In higher co-dimensions where $s<n$ it is also true that for every boundary cube $Q \in \pom$, 
there exist a uniformly bounded number of Whitney cubes $P \in \WW(\om)$ such that $b(P) = Q$.  For the proof of this fact 
one can see \cite[Lemma 4.16, Lemma 4.18]{MaPo}.

We also denote the ``fattened" Whitney region of $Q$ by 
\begin{equation}\label{eqwqfat}
\wt{w}(Q)=\bigcup_{P\in\WW(\om): Q=b(P)} 1.1P.
\end{equation}

\begin{remark}\label{rem:whitney overlap}
If $x \in \bar P \in \WW(\om)$,  then there exists a constant $C_w>1$ depending only on $n$, the constants of the Whitney decomposition, and the $s$-Alhlfors regularity,  so that for every $P' \in \WW(\om)$ that has the property $x \in 1.1 P'$, it holds that
$$
b(P') \subset B(x_{b(P)}, C_w \ell(P))=:B_{P}.
$$
\end{remark}

\vvv

\section{Regularized dyadic extension of functions on the boundary}\label{sec:dyadic extension}

Let $\om \subset \rrn$, $n \geq 1$,  be an open set and let  $\WW(\Omega)$ be the collection of Whitney cubes in which $\om$ is decomposed as in Subsection \ref{subsec: Whitney}.  Let  $\{\vp_P\}_{P\in\WW(\Omega)}$ be a partition of unity subordinate to the open cover $\{1.1 P\}_{P \in \WW(\Omega)} $  such that
\begin{equation*}
\vp_P\in C^\infty(\rn),\,\, \,|\nabla \vp_P|\lec \frac{1}{\ell(P)}, \,\,\, \supp\vp\subseteq 1.1 P, \,\,\textup{and}\,
\sum_{P\in\WW(\Omega)}\vp_P(x)={\bf1}_{\Omega}(x), x \in \om.
\end{equation*}
For $f\in L^1_{\loc}(\sigma)$,  we define the {\it regularized dyadic extension of $f$} in $\Omega$ by
\begin{equation}\label{eq:dyadext}
\upsilon_f(x):=
\begin{dcases}
\underset{{P\in\WW(\Omega)}}{\sum} m_{\sigma, b(P)}f\, \varphi_P(x) &\textup{if}\,\,  x \in \Omega\\
f(x) &\textup{if}\,\,x\in \partial \Omega.
\end{dcases}
\end{equation} 
where,  in the case that $\om$ is an unbounded domain with compact boundary,  we set $b(P)=\pom$ for every $P \in \WW(\om)$ with  $\ell(P)\geq \diam(\pom)$.

\vv

The fact that $\upsilon_f$ is indeed an extension of $f$ in $\Omega$ (in the non-tangential sense) is proved in the following lemma. 
\begin{lemma}\label{lemma:ntconvergenceofdyadic}
Let $\om \in \AR(s)$ for $s \in(0,n]$ and $f \in L^1_{\loc}(\sigma_s)$. There exists $\alpha>0$  such that 
$$
\ntlim_{\substack{x\to \xi}} \upsilon_f(x) = f(\xi),  \quad \textup{for}\, \,\sigma\textup{-a.e.}\, \,\xi \in \pom,
$$
for some cone $\gamma_\alpha$ where $\alpha>0$ only depends  on $n$,  the  Ahlfors regularity constant,  and the constants of the corkscrew condition.
\end{lemma}

\begin{proof}
By \cite[Theorem 1.33]{EG92},  it holds that 
\begin{equation}\label{eq:Lebdiff}
\lim_{r \to 0} m_{\sigma, B(\xi, r)}( |f - f(\xi)|) = 0, \textup{for}\,\,  \sigma\textup{-a.e.} \,\xi \in \pom.
\end{equation}
Fix $\ve>0$ and $\xi\in\pom$ be a point such that \eqref{eq:Lebdiff} holds.  Let $0<\ve'< \ve$ to be picked later. Then there exists $\delta=\delta(\ve',\xi)>0$ such that $m_{\sigma, B(\xi, r)}( |f - f(\xi)|)<\ve'$ for every $r<\delta$.  Let now $0<\delta' <\delta$ be a small constant which  will be chosen momentarily.  For fixed $x\in \gamma_{\alpha}(\xi) \cap B(\xi, \delta')$,  we have that
$$
| \upsilon_f(x) - f(\xi) |\leq \sum_{P\in\WW(\om)} | m_{\sigma, b(P)}f - f(\xi) |\, \vp_P(x).
$$
Let $P_0 \in \WW(\om)$ be a fixed cube so that $x \in \bar P_0$.  Then the only cubes $P\in\WW(\om)$ that contribute to the sum are the ones that $x\in 1.1P$ and,  by Remark \ref{rem:whitney overlap},  $b(P) \subset B_{P_0}=B(x_{P_0}, C_w \ell(P_0))$.  By the properties of Whitney cubes, there exists a constant $C_w'>C_w$ such that $B_P \subset B(\xi, C_w'\,\ell(P))$. Therefore,  choosing  $\delta'>0$ sufficiently small so that $B(\xi, C_w'\,\ell(P)) \subset B(\xi,  \delta/2)$,  we get that for any such $P$,
\begin{align*}
 |m_{\sigma, b(P)}f - f(\xi)|   \lec m_{\sigma, B_{P_0}}(| f - f(\xi) |)  \lesssim \ve',
\end{align*}
which, in turn,  by the bounded overlaps of the Whitney cubes,  infers that there exists a constant $C>1$ such that 
$$
|\upsilon_f(x) - f(\xi)| < C \, \ve',
$$
concluding the proof of the lemma once we choose  $\ve= C \ve'$. 
\end{proof}

\vv

\begin{lemma}\label{lem: gradufestimates}
Let $\om \in \AR(s)$ for $s \in(0,n]$.  Assume that  $f \in L^1_{\loc}(\sigma)$ and $\upsilon_f$ is the extension defined in \eqref{eq:dyadext}.  For any $P\in\WW(\om)$,  we  have that
\begin{equation}\label{graduf}
\sup_{x \in \bar P }|\nabla \upsilon_f(x)| \lec \ell(P)^{-1} \, m_{\sigma, B_P}(|f|).
\end{equation}
If, additionally, $f\in\Lambda_\beta(\pom)$ for $\beta \in [0,1)$,  then it holds  that
\begin{equation}\label{gradufbmo}
\sup_{x \in \bar P } |\nabla \upsilon_f(x)| \lec \ell(P)^{\beta-1} \, \|f \|_{\Lambda_\beta(\pom)},
\end{equation}
while, if $ f \in \dot{M}^{1, p}(\sigma)$, we get that 
\begin{equation}\label{gradufW}
| \nabla \upsilon_f(x) |\lec m_{\sigma, B_P }(\nabla_{H,p} f),
\end{equation}
where $\nabla_{H,p} f$ is the least Haj\l asz  upper gradient of $f$. 
Moreover,  for any $\xi \in\pom$,
\begin{equation}\label{ntMFuf}
\calN_{\alpha}(\upsilon_f)(\xi) \lec_\alpha \mathcal{M}f(\xi), 
\end{equation}
\end{lemma}

\begin{proof}
Fix $x\in \bar P\in\WW(\om)$.  For the proof of the estimate \eqref{graduf}, just note that
$$
|\nabla \upsilon_f(x)| \lec \sum_{\substack{ P'\in\WW(\om) \\ x\in 1.1P'}} m_{\sigma, b(P')}(|f|) \,  \ell(P')^{-1}
\lec_{D_0} m_{\sigma, B_P}(|f|)\,  \ell(P)^{-1},
$$
where we used that  $\ell(P)\approx   \ell(P')$.
If,  in addition, $f \in \Lambda_\beta(\pom)$, then, using the fact that
$\nabla\big(\sum_{P\in\WW(\om)}\vp_P(x)\big)=0$,  we get
\begin{align*}
|\nabla\upsilon_f(x)|&=\Big| \sum_{P'\in\WW(\om)}(m_{\sigma, b(P')}f - m_{\sigma, b(P)}f) \nabla\vp_{P'}(x)\Big| \\
&\lec \ell(P)^{-1} \sum_{\substack{P'\in\WW(\om) \\ x\in 1.1P'}}|m_{\sigma, b(P')}f - m_{\sigma, b(P)}f| \\ 
&\lec\ell(P)^{-1}\sum_{\substack{P'\in\WW(\om) \\ x\in 1.1P'}}m_{\sigma, b(P')} (|f - m_{\sigma,  B_P}f|) \lec_{D_0}  \| f \|_{\Lambda_\beta(\pom)} \ell(P)^{\beta-1}.
\end{align*}
This obviously implies estimate \eqref{gradufbmo}.

Now,  in order to prove \eqref{gradufW},  let $ f\in \dot{M}^{1, p}(\sigma) $ and  $ x\in \bar{P} $.  Then we have 
\begin{align*}
| \nabla \upsilon_f (x) | &= \big| \sum_{ P' \in \WW(\om) } m_{\sigma, b(P')} f  \, \nabla \vp_{P'} (x) \big| 
\lec \sum_{ \substack{ P'\in \WW(\om) \\ x\in 1.1P' }} \frac{1}{\ell(P')} | m_{\sigma, b(P')} f - m_{\sigma, b(P)} f |  \\
&\lec \sum_{ \substack{ P' \in \WW(\om) \\ x\in 1.1P' }} \frac{1}{\ell(P')}
 \fint_{b(P')} \fint_{b(P)} | x - y | [ \nabla_{H,p} f(x) + \nabla_{H,p} f(y) ] \, d\sigma(x)  \, d\sigma(y) \\
 &\lec \sum_{\substack{ P'\in \WW(\om) \\ x\in 1.1P' }} ( m_{\sigma, b(P')}(\nabla_{H,p} f) + m_{\sigma, b(P)}(\nabla_{H,p} f))
 \lec_{D_0} m_{\sigma, B_P} (\nabla_{H,p} f),
\end{align*}
where $ B_P $ was defined in Remark \ref{rem:whitney overlap}.  

 Finally, fox fixed $\xi\in\pom$, if $x \in \gamma_\alpha(\xi)$,  we have  
$$
|\upsilon_f(x)| \leq \sum_{\substack{P'\in\WW(\om)\\x\in 1.1P'}} m_{\sigma, b(P)}(|f|)\, \vp_P(x) 
\lec m_{\sigma, B_P}(|f|) \lec  m_{\sigma, B(\xi, C_w' \ell(P))}(|f|)  \leq \mathcal{M}f(\xi),
$$
for some constant $C_w'>C_w$ depending on $\alpha$ and the Whitney constants.  This readily proves \eqref{ntMFuf} by taking supremum over all  $x \in \gamma_\alpha(\xi)$.
 \end{proof}

\vv
\begin{lemma}\label{lem: sharpntmaxfunctionuf bmo}
If $f\in\Lambda_\beta(\pom)$ for $\beta \in [0,1)$, then it holds that 
\begin{equation}\label{eq: sharpntMFuf-bmo}
\sup_{\xi \in \pom} \calN^{(\beta)}_{\sharp}(\upsilon_f)(\xi)\lec \| f \|_{\Lambda_\beta(\pom)}.
\end{equation}\end{lemma}
\begin{proof}
Fix $\xi\in\pom$ and take $x\in\gamma_\alpha(\xi)$ with 
$x\in\bar{P_0}$, for some $P_0\in\WW(\om)$.  It is enough to bound the difference 
$$
\sup_{y \in B(x, c\delta_\om(x))}\fint_{B(x, c\delta_\om(x))}|\upsilon_f(y)-\upsilon_f(z)|\, dz.
$$
To this  end, fix a point $z\in B(x, c \delta_\om(x))$, with $z\in \bar{P}_1\in\WW(\om)$ and a point $y \in B(x, c \delta_\om(x)) $ with 
$y\in\bar{P}_2\in\WW(\om)$.  Since $c>0$ is small enough,  the Whitney cubes $P_1, P_2$ and $P_0$ are ``close-enough" to each other, in the sense 
that the intersection of a dilation of these cubes is non-empty. 
Then,  by the properties of Whitney cubes, we get that $\ell(P_1)\approx\ell(P_2)\approx\ell(P_0)\approx \delta_\om(x)$. 
Thus, there exist a large enough constant $\Lambda_0>1$, such that for any $P\in\WW(\om)$ with
$1.1P\cap (P_1\cup P_2)\neq \emptyset$,  we have that
$$
B_{b(P)}\subset B_0:= B(x_{b(P_0)}, \Lambda_0\ell(P_0)).
$$
It holds  that
$$
\sum_{\substack{P\in\WW(\om) \\ P\cap B(x, c\delta_\om(x))\neq \emptyset}}(\vp_P(z)-\vp_P(y)) \, m_{\sigma, B_0}f = 0,
$$
and so 
\begin{align*}
&|\upsilon_f(y)-\upsilon_f(z)| \leq \sum_{\substack{P\in\WW(\om)\\ P\cap B(x, c\delta_\om(x))\neq \emptyset}} 
|\vp_P(y)-\vp_P(z)| |m_{\sigma, b(P)}f - m_{\sigma, B_0}f | \\
&\lec \sum_{\substack{P\in\WW(\om)\\ P\cap B(x, c\delta_\om(x))\neq \emptyset}}
\|\nabla \vp_P\|_{L^\infty}|y-z| |m_{\sigma, b(P)}f - m_{\sigma, B_0} f | \\
&\lec \sum_{\substack{P\in\WW(\om)\\ P\cap B(x, c\delta_\om(x))\neq \emptyset}}
 m_{\sigma, b(P)}(| f - m_{\sigma, B_0}f |)
\lec  \| f \|_{\Lambda_\beta(\pom)} \ell(P_0)^{\beta},
\end{align*}
since there are uniformly bounded many Whitney cubes $P\cap B(x, c\delta_\om(x))\neq \emptyset $.  This readily implies \eqref{eq: sharpntMFuf-bmo}.
\end{proof}

\vv

\begin{lemma}\label{lem:carleson-upsilon}
Let $\om \in \AR(s)$ for $s \in (0,n]$.  If $f \in \Lambda_\beta(\pom)$ for $\beta \in (0,1)$,  then 
\begin{equation}\label{eq:carleson-upsilon}
\sup_{\xi \in \pom} \mathcal{C}^{(\beta)}_s(\nabla \upsilon_f)(\xi) \lec \|f\|_{ \Lambda_\beta(\pom)}.
\end{equation}
\end{lemma}

\begin{proof}
By \eqref{gradufbmo}, it is easy to see that for every $\xi \in \pom$ and $r>0$, it holds that
\begin{align*}
\int_{B(\xi,r) \cap \om}& |\nabla \upsilon_f| \,\omega_s(x)\,dx \lec  \|f \|_{\Lambda_\beta(\pom)}\,\sum_{\substack{P \in \WW(\om)\\ P \cap B(\xi,r) \neq \emptyset}}  \ell(P)^{\beta-1} \ell(P)^{s+1}\\
& \lec  \|f \|_{\Lambda_\beta(\pom)}\,\sum_{\substack{P \in \WW(\om)\\ P \cap B(\xi,r) \neq \emptyset}}  \ell(b(P))^{\beta} \sigma(b(P))\\ &\lec   \|f \|_{\Lambda_\beta(\pom)}\,\sum_{k=0}^\infty 2^{-\beta\,k}r^\beta \sum_{\substack{Q \in \DD_{k,\sigma}\\ Q \subset B(\xi, M_0 r)}} \sigma(Q) \lec  \|f \|_{\Lambda_\beta(\pom)}\, r^\beta \, r^s,
\end{align*}
which implies \eqref{eq:carleson-upsilon}.
\end{proof}
\vv

If the boundary function is in $\Lip_\beta(\pom)$ for $\beta \in (0,1]$, then we can show that $\upsilon_f$ is $\Lip_\beta(\oom)$ by arguments similar to the ones in \cite[Lemma 4.2]{MT}.

\begin{lemma}\label{lem: lemma1}
If $f \in \Lip_\beta(\partial\Omega)$ for $\beta \in (0,1]$ and  $\upsilon_f$ is the regularized dyadic extension defined in \eqref{eq:dyadext},  then it holds that  $\upsilon_f \in \Lip_\beta(\oom)$ with $\Lip_\beta(\upsilon_f)\lesssim \Lip_\beta(f)$. 
\end{lemma}

\begin{proof}
 We start by proving that $\upsilon_f\in \Lip_\beta(\Omega)$.   Fix $x,y\in\Omega$ and  let $P_1,  P_2 \in \WW(\Omega)$ such that $x\in\bar{P_1}$ and 
$y\in\bar{P_2}$.  We split into  cases.

 {\bf Case 1:} Suppose that $2P_1\cap 2P_2\neq\emptyset$.  In this case let $P_0\in\WW(\Omega)$ be the smallest cube such that it holds
$$
2B_{b(P)}\subset 2B_{b(P_0)},  \,\, \textup{for every}\,\, P\in\WW(\Omega),\,\, \textup{with}\, 1.1P\cap(P_1\cup P_2)\neq\emptyset.
$$
Then,  by the properties of Whitney cubes, we get that $\ell(P_1)\approx\ell(P_2)\approx\ell(P_0)$. 
As
$$
\upsilon_f(x)-\upsilon_f(y) = \sum_{P\in\WW(\Omega)}m_{\sigma, b(P)}f \,(\vp_P(x)-\vp_P(y))
$$
and
$$
\sum_{P\in\WW(\Omega)}(\vp_P(x) - \vp_P(y))\,m_{\sigma, b(P_0)} f
= 0,
$$
we can write 
\begin{align*}
\upsilon_f(x)-\upsilon_f(y)&\\
 = \sum_{P\in\WW(\Omega)} &m_{\sigma, b(P)} f \,(\vp_P(x)-\vp_P(y)) -\sum_{P\in\WW(\Omega)}  m_{\sigma, b(P_0)}f \,(\vp_P(x)-\vp_P(y)) \\
= \sum_{P\in\WW(\Omega)}&(\vp_P(x)-\vp_P(y))\left(m_{\sigma, b(P)}f - m_{\sigma, b(P_0)}f \right)
\end{align*} 
and thus we get 
\begin{align*}
|\upsilon_f(x)-\upsilon_f(y)| &\leq  \sum_{P\in\WW(\Omega)}|\vp_P(x)-\vp_P(y)| \left| m_{\sigma, b(P)}f - m_{\sigma, b(P_0)}f \right|.
\end{align*}
Observe now that 
$$
|\vp_P(x)-\vp_P(y)|\leq|\nabla\vp_P||x-y|\lec{\ell(P)}^{-1}|x-y|
$$
while, for fixed $w\in 2B_{b(P_0)}$, we can estimate 
\begin{align*}
\left| m_{\sigma, b(P)}f - m_{\sigma, b(P_0)}f \right| &\leq m_{\sigma, b(P)}(|f(z)-f(w)|) + m_{\sigma, b(P_0)}(|f(z)-f(w)|) \\
&\lec {\Lip}_{\beta}(f)\,\ell(P_0)^{\beta}.
\end{align*}
To deal with the  sum,  we may assume that the cubes $P$ appearing in the sum are such that either $1.1P\cap P_1\neq\varnothing$ or
$1.1P\cap P_2\neq\varnothing$, since otherwise the associated summand vanishes. We denote by $I_0$ the family of such cubes.
So the cubes from $I_0$ are such that $B_{b(P)}\subset 2B_{b(P_0)}$ and they satisfy $\ell(P)\approx\ell(P_0)$.  Combining this observation with the last two estimates and the fact that $|x-y| \lec \ell(P_0)$,   we  obtain 
\begin{align}\label{eq:dyadext-Lip1}
|\upsilon_f(x)-\upsilon_f(y)| &\lec \sum_{P\in\WW(\Omega)}\frac{1}{\ell(P)}|x-y|{\Lip}_{\beta}(f)\,\ell(P_0)^{\beta} 
\lec_{D_0} {\Lip}_{\beta}(f)\,|x-y|^{\beta}.
\end{align}

{\bf Case 2:} Suppose that $2P_1\cap 2P_2=\emptyset$. In this case we have 
\begin{align*}
\upsilon_f(x)-\upsilon_f(y) = \sum_{P\in\WW(\Omega)} m_{\sigma, b(P)}f &\, (\vp_P(x)-\vp_P(y)) \\
= \sum_{P\in\WW(\Omega)}\vp_P(x)(m_{\sigma, b(P)}f - &m_{\sigma, b(P_1)}f ) +\sum_{P\in\WW(\Omega)}\vp_P(y)\left(m_{\sigma, b(P_2)}f - m_{\sigma, b(P)}f \right) \\
&+ \left( m_{\sigma, b(P_1)}f - m_{\sigma, b(P_2)}f \right) =:S_1+S_2+S_3.
\end{align*}
If we use that $\ell(P_i)\lesssim |x-y|$ and the fact that $| m_{\sigma, b(P)}f - m_{\sigma, b(P_i)}f | \lec  {\Lip}_{\beta}(f)\,\ell(P_i) ^{\beta}$,  for $i=1,2$, we can show that
$$ 
|S_1| + |S_2| \lec  {\Lip}_{\beta}(f)\,|x-y|^{\beta}.
$$
 It remains to bound $S_3$.  If $w_1\in b(P_1)$ and $w_2\in b(P_2)$,
 $$
 |S_3| \leq \left| m_{\sigma, b(P_1)}f - f(w_1)\right| + |f(w_1)-f(w_2)| + \left| f(w_2) - m_{\sigma, b(P_2)}f \right|.
$$
Since $f\in \Lip_\beta(\partial\Omega)$, 
\begin{align*}
|f(w_1)-f(w_2)| &\leq {\Lip}_{\beta}(f)\,|w_1-w_2|^{\beta} \leq {\Lip}_{\beta}(f)\,\big(|w_1-x|^{\beta} +|x-y|^{\beta} +|y-w_2|^{\beta} \big)\\
&\lec{\Lip}_{\beta}(f)\,\big(\ell(P_1)^{\beta}  + |x-y|^{\beta} + \ell(P_2)^{\beta} \big)\lec {\Lip}_{\beta}(f)\,|x-y|^{\beta},
\end{align*}
while, for $i=1,2$,  once again using that $f\in {\Lip}_{\beta}(\partial\Omega)$,  it is easy to see that
 \begin{align*}
 \left| m_{\sigma, b(P_i)}(f - f(w_i))\right| &\leq m_{\sigma, b(P_i)} (| f-f(w_i)|) \leq {\Lip}_{\beta}(f)\,\ell(P_i)^{\beta} \lesssim {\Lip}_{\beta}(f)\,|x-y|^{\beta}.
 \end{align*}
 implying that  $ |S_3| \lec \Lip_\beta(f) |x-y|^\beta $.  If we combine the above estimates, we get 
 $$
| \upsilon_f(x)-\upsilon_f(y)| \leq |S_1| + |S_2| + |S_3| \lec {\Lip}_{\beta}(f)\,|x-y|^{\beta}, 
 $$ 
in the case 2 as well and thus for all $x,y \in \Omega$ with $x \neq y$.  This readily implies that $\upsilon_f\in {\Lip}_{\beta}(\Omega)$ with ${\Lip}_{\beta}(\upsilon_f)\leq {\Lip}_{\beta}(f)$. 

It remains to prove that 
\begin{equation}\label{eq:upsilon-holder-boundary}
|\upsilon_f(x) - \upsilon_f(y)|  \lesssim {\Lip}_{\beta}(f)\,|x-y|^{\beta}, \quad \textup{for any } \,\,x\in\pom\,\,\textup{and}\,\, y\in \om.
\end{equation}
To this end, we fix such $x$ and $y$ and estimate
\begin{align*}
|\upsilon_f(x) &- \upsilon_f(y)| = |f(x) - \upsilon_f(y)| 
= \Big| f(x) - \sum_{P\in\WW(\Omega)}m_{\sigma, b(P)}f \,\vp_P(y)\Big| \\
&\leq  \sum_{P\in\WW(\Omega)}\vp_P(y) |f(x) - m_{\sigma, b(P)}f| \leq \sum_{\underset{1.1 P\ni y}{P\in\WW(\Omega):}}\vp_P(y) | f(x) - m_{\sigma, b(P)}f|.
\end{align*}
As  $f\in {\Lip}_{\beta}(\partial\Omega)$,   for every $P \in \WW(\om)$ such that $y \in 1.1P$,  we have that
$$ 
| f(x) - m_{\sigma, b(P)}f|\lec {\Lip}_{\beta}(f) \,\ell(P)^{\beta} \approx  {\Lip}_{\beta}(f) \, \delta_\om(y)^{\beta}  \leq  {\Lip}_{\beta}(f) \,|x-y|^{\beta},
$$
which implies \eqref{eq:upsilon-holder-boundary} by the bounded overlaps of the Whitney cubes that contain $y$,  concluding the proof of the lemma.
\end{proof}

\vv

\begin{proof}[proof of Theorem \ref{th: introtheoremholder}]
It follows by combining Lemmas \ref{lemma:ntconvergenceofdyadic}, \ref{lem: gradufestimates}, \ref{lem: sharpntmaxfunctionuf bmo}, \ref{lem:carleson-upsilon},  and \ref{lem: lemma1} (see also Remark \ref{rem:campanato}).
\end{proof}

\vvv

\section{A Corona decomposition for functions in $L^p$ or $\BMO$}\label{sec: large oscillation cubes}

In this section we will assume that $\om \in \AR(s)$ for $s \in(0,n]$.

\vv

We say that a family of cubes $\mathcal{F} \subset \DD_\sigma$ satisfies a {\it Carleson packing condition} with constant $M>0$,  and we write $\mathcal{F} \in \car(M)$, if for any $S \in \DD_\sigma$, it holds that
\begin{equation}\label{eq:packing}
\sum_{R \in \mathcal{F}: R\subset  S} \sigma(R) \leq M \, \sigma(S),
\end{equation}

\vv

A family $\mathcal{T}\subset\DD_\sigma$ is a {\it tree} if it
verifies the following properties:
\begin{enumerate}
\item $\mathcal{T}$ has a maximal element (with respect to inclusion) $Q(\mathcal{T})$ which contains all the other
elements of $\mathcal{T}$ as subsets of $\R^{n+1}$. The cube $Q(\mathcal{T})$ is the ``root'' of $\mathcal{T}$ and we will  call it  ``top" cube.
\item If $Q,Q'$ belong to $\mathcal{T}$ and $Q\subset Q'$, then any $\mu$-cube $P\in\DD_\sigma$ such that $Q\subset P\subset
Q'$ also belongs to $\mathcal{T}$. 
\end{enumerate}
For a tree $\mathcal T$, if $R=Q(\mathcal{T})$ is a top cube,  we will write $\mathcal{T}=\tree(R)$.

\vv

\begin{definition}\label{def: corona decomposition}
A {\textit {corona decomposition}} of $\sigma$ is a partition of $\DD_\sigma$ into a family of ``good cubes", which we denote by $\mathcal{G}$, and a family of ``bad cubes",  which we denote by $\mathcal{B}$, so that the following hold:
\begin{enumerate}
\item $\DD_\sigma=\mathcal{G} \cup \mathcal{B}$;
\item There is a partition of $ \mathcal{G}$ into trees,  that is,
$$\mathcal{G}= \bigcup_{\mathcal{ T}\subset \mathcal{G}} \mathcal T;$$
\item The collections of maximal cubes $ Q(\mathcal{T})$ of the trees $\mathcal{T}$ satisfies \eqref{eq:packing} for some $M_0>0$;
\item The collection of cubes $ \mathcal{B}$ satisfies  \eqref{eq:packing} for some $M_1>0$.
\end{enumerate}
\end{definition}
\vv

We can also define a localized Corona decomposition in a cube $R_0\in \DD_\sigma$ if, in the definition above, we replace $\mathcal{D}_\sigma$ by $\mathcal{D}_\sigma(R_0)$. 

\vv

We recall the definition of the truncated (at large scales) dyadic Hardy-Littlewood maximal function
$$
Mf(Q)=\sup_{\substack{R\in\DD_\sigma \\ Q\subset R}} m_{\sigma, R}|f|, \quad f  \in L^1_{\loc}(\sigma).
$$

\vv

 Given any $R \in\DD_\sigma$ and for fixed $\ve>0$, we define the collection $\sss(R) \subset \DD(R)$ consisting of cubes  $S\in \DD(R)$ which are maximal (thus disjoint) with respect to the  condition
\begin{equation}\label{stopcond}
| m_{\sigma, R}f - m_{\sigma, S}f | \geq
\begin{dcases}
 \ve Mf(S) &,\textup{if}\,\,  f\in L^1_{\loc}(\sigma)\\
 \ve\|f\|_{\BMO(\sigma)} &,\textup{if}\,\,  f\in \BMO(\sigma).
\end{dcases}
\end{equation}

We fix a cube $R_0\in\DD_\sigma$ and we define the family of the top cubes with respect to $R_0$ as follows:
first we define the families $\ttt_k(R_0)$ for $k\geq 0$ inductively. We set
$$\ttt_0(R_0)=\{R_0\}.$$
Assuming that $\ttt_k(R_0)$ has been defined, we set
$$\ttt_{k+1}(R_0) = \bigcup_{R\in\ttt_k(R_0)}\sss(R),$$
and then we define
$$\ttt(R_0)=\bigcup_{k\geq0}\ttt_k(R_0).$$
We also set 
$$
\tree(R):=\{ Q \in \DD_\sigma (R): \nexists \,\, S \in \sss(R) \,\, \textup{such that}\,\, Q\subset S\}.
$$
and notice that
$$\DD_\sigma(R_0)= \bigcup_{R\in\ttt(R_0)}\tree(R),$$
and this union is disjoint.  This is a localized Corona decomposition in  $R_0$ and notice that, in this case,  $\mathcal{B}=\emptyset$.

\vvv

In the rest of this section,  we will devote all our efforts to  proving that $\ttt(R_0)$ satisfies a Carleson packing condition.
\begin{proposition}\label{pro:packing Top}
For any $R_0 \in \DD_\sigma$,  the family of cubes $\ttt(R_0) \in \car(C \ve^{-2})$ for some $C>0$ depending on the Ahlfors-regularity constants.
\end{proposition}

For the proof of proposition \ref{pro:packing Top} we consider the cases $f\in \BMO(\sigma)$ and $f \in L^1_{\loc}(\sigma)$ separately.

\begin{proof}[Proof of Proposition \ref{pro:packing Top} when $f \in \BMO(\sigma)$]
For any  $R\in\ttt(R_0)$ it holds
$$
| m_{\sigma, S}f - m_{\sigma, R}f |>\ve \| f \|_{\BMO(\sigma)}.
$$
Define
$$
f_R(x):=\sum_{Q\in\tree(R)}\Delta_Qf(x)=\sum_{Q\in\tree(R)}\sum_{Q' \in \ch(Q)} \left(m_{\sigma, Q'}f - m_{\sigma, Q}f \right) 1_{Q'}(x).
$$
If $x\in P \in\sss(R)$, we have that 
$ f_R(x)= m_{\sigma, P}f - m_{\sigma, R}f $ 
and so, $|f_R(x)| >\ve\|f\|_{\BMO(\sigma)}$.  This implies that
\begin{align*}
\ve^2\| f\|_{\BMO(\sigma)}^2\sum_{P\in\sss(R)}\sigma(P) &\leq
\sum_{P\in\sss(R)}\int_P| f_R(x)|^2\, d\sigma(x)  \\
&= \int_{\underset{P\in\sss(R)}{\bigcup P}}|f_R|^2\, d\sigma
\leq \int |f_R|^2\, d\sigma.
\end{align*}
By the above estimate and the orthogonality of $\Delta_Qf$,
\begin{align*}
\sum_{\substack{R\in\ttt(R_0) \\ R \subset S}} \sum_{P\in\sss(R)}\sigma(P) &\leq\frac{1}{\ve^2}\frac{1}{\|f\|_{\BMO(\sigma)}^2}
\sum_{R\in\ttt(R_0)}\|f_R\|_{L^2(\sigma)}^2 \\
&=\frac{1}{\ve^2}\frac{1}{\|f\|_{\BMO(\sigma)}^2}\sum_{R\in\ttt(R_0)}\sum_{Q\in\tree(R)}\|\Delta_Qf\|_{L^2(\sigma)}^2 \\
&\leq \frac{1}{\ve^2}\frac{1}{\|f\|_{\BMO(\sigma)}^2}\sum_{Q\in\DD_\sigma(R_0)}\| \Delta_Qf \|_{L^2(\sigma)}^2 \\
&=\frac{1}{\ve^2}\frac{1}{\|f\|_{\BMO(\sigma)}^2}\| {\bf 1}_{R_0}(f - m_{\sigma, R_0}f)\|_{L^2(\sigma)}^2 
\lec \ve^{-2} \sigma(R_0)
\end{align*} 
which proves \eqref{eq:packing} in the case that $S\in \ttt(R_0)$.  By the same argument as in the end of the proof of Proposition \ref{pro:packing Top} when $f \in L^1_{\loc}(\sigma)$, we obtain  \eqref{eq:packing} for any $S \in \DD_\sigma(R_0)$.
\end{proof}

\vv

To prove the proposition for $f \in L^1_{\loc}$, we first need some auxiliary lemmas. 
\begin{lemma}\label{lem: lemma4.1}
Let $f\in L_{\loc}^1(\sigma)$ and $Q\in\DD_\sigma$. Then it holds
\begin{equation}\label{eq: eq4.1}
\frac{\sigma(Q)}{Mf(Q)^2} \leq 8 \int_Q\frac{1}{\left(\mathcal{M}_{\DD_\sigma}f(x)\right)^2}\, d\sigma(x)
\end{equation}
\end{lemma}
\begin{proof}
This was shown in \cite[ Lemma 4.1]{HR18}  for the Lebesgue measure but the same proof works for any
non-atomic Radon measure and so we skip the details. 
\end{proof}

\vv

Let $\mathcal{F}\subset \DD_\sigma$ be any collection of dyadic cubes. 
Given any cube $Q\in\DD_\sigma$, define its stopping parent $Q^*$ to be the minimal $Q^*\in \mathcal{F}$ such that 
$Q\subsetneq Q^*$. If no such $Q^*$ exists, we set $Q^*:=Q$. Define the stopped square function 
\begin{equation}\label{sqf}
\mS_{\mathcal{F}}f(x):=\Big(\sum_{Q\in \mathcal{F}}| m_{\sigma, Q}f - m_{\sigma, Q^*}f |^2{\bf 1}_Q(x)\Big)^{1/2}.
\end{equation}
In the special case  $\mathcal F=\ttt(R_0)$,  we will simply write $\mS f$.

\begin{lemma}\label{pr: sqfest}
If  $w\in A_\infty(\sigma)$ and $1\leq p < \infty$, then
$$
\| \mS_{\mathcal{F}}f \|_{L^p(\pom; w)} \lec \| M_{\DD}f \|_{L^p(\pom; w)}
$$
uniformly for any collection of dyadic cubes $\mathcal{F}$. 
\end{lemma}
\begin{proof}
This was  proved in \cite[Proposition 3.2]{HR18} for the Lebesgue measure but the same proof works verbatim for $\sigma$.
\end{proof}

We will now proceed  to the proof of Proposition \ref{pro:packing Top}  for $f\in L^1_{\loc}(\sigma)$,  which is based on the one  of \cite[Theorem 1.2(3)]{HR18}. 


\begin{proof}[Proof of Proposition \ref{pro:packing Top}  when $f \in L^1_{\loc}(\sigma)$]
We first  fix $S \in \ttt(R_0)$.  As for any $R\in\ttt(R_0)$ it holds that
$$
| m_{\sigma, R}f - m_{\sigma, R^*}f | >\ve Mf(R) 
$$
we have that 
\begin{align*}
\sum_{\underset{R\subset S}{R\in\ttt(R_0)}}&\sigma(R)\leq
\sum_{\underset{R\subset S}{R\in\ttt(R_0)}}\frac{| m_{\sigma, R}f - m_{\sigma, R^*}f |^2}{\ve^2 Mf(R)^2}\sigma(R)\\
&\lec \sum_{\underset{R\subset S}{R\in\ttt(R_0)}} \frac{| m_{\sigma, R}f - m_{\sigma, R^*}f |^2}{\ve^2}
\int_{S}\frac{{\bf 1}_R(x)}{\mathcal{M}_{\DD_\sigma}f(x)^2}\, d\sigma =\int_{S} \frac{\mS f(x)^2}{\ve^2} 
\frac{d\sigma(x)}{\mathcal{M}_{\DD_\sigma}f(x)^2},
\end{align*}
where,  in the second inequality, we used  Lemma \ref{lem: lemma4.1}.  We  write 
$$
(\mathcal{M}_{\DD_\sigma}f)^{-2} = 1\cdot \left((\mathcal{M}_{\DD_\sigma}f)^{2\gamma}\right)^{1-q}
$$
for $\gamma\in(0,1/2)$ and $q=1+\frac{1}{\gamma}>3$.  Since $f\in L_{\loc}^1(\sigma)$ and $2\gamma\in(0,1)$, using for example 
\cite[Theorem 3.4, p.158]{CG} (whose proof works for doubling Borel measures), we get that 
$\left(\mathcal{M}_{\DD_\sigma}f\right)^{2\gamma}\in A_1(\sigma)$.  As $1\in A_1$ and $q>1$ it follows from \cite[Theorem 2.16, p.407]{CG} (whose proof also works for doubling Borel measures)  that 
$1\cdot\left(\left(\mathcal{M}_{\DD}f \right)^{2\gamma}\right)^{1-q} \in A_q(\sigma)$.
Therefore, $\left(\mathcal{M}_{\DD_\sigma}f\right)^{-2}\in A_q(\sigma)\subset A_\infty(\sigma)$. 
We now apply Lemma
\ref{pr: sqfest} with the collection of cubes $\wt{\mathcal{F}}:=\{R\in\ttt(R_0) : R \subset S\}$ to
 the function
\begin{equation*}
\tilde{f}(x):=
\begin{dcases}
f(x) - m_{\sigma, S}f    &,\textup{if}\,\,x\in S \\
0    &,\textup{if}\,\,x\notin S,
\end{dcases}
\end{equation*}
for the weight $w:=\left(\mathcal{M}_{\DD_\sigma}f\right)^{-2}$ and $p=2$,  and obtain
\begin{align*}
\int |S_{\wt{\mathcal{F}}}\tilde{f}|^2\, w \,d\sigma &\lec \int |\mathcal{M}_{\DD}\tilde{f}|^2 \, w\,d\sigma = \int_{S}|\mathcal{M}_{\DD}\tilde{f}|^2\, w\,d\sigma \lec \int_{S}|\mathcal{M}_{\DD}f|^2 \, w\, d\sigma.
\end{align*}
Thus, since
$|S_{\wt{\mathcal{F}}}\tilde{f}(x)|^2=|S_{\wt{\mathcal{F}}}f(x)|^2$ for all $x \in S$, we infer that
\begin{align*}
\int_{S}|\mS f|^2 w\,d\sigma &\lec 
\int_{S}|\mS_{\wt{\mathcal{F}}}f|^2 \, w\,d\sigma + \int_{S}|\mathcal{M}_{\DD}f|^2 \, w\, d\sigma \leq \int| \mS_{\wt{\mathcal{F}}}\tilde{f}|^2 \, w\, d\sigma + \int_{S}|\mathcal{M}_{\DD}f|^2 \, w\, d\sigma \\
&\lec \int_S |\mathcal{M}_{\DD}f|^2 \, w\, d\sigma = \int_{S}|\mathcal{M}_{\DD}f|^2 
\frac{d\sigma}{\left(\mathcal{M}_{\DD}f\right)^2} =\sigma(S),
\end{align*}
proving \eqref{eq:packing} in the case that $S\in \ttt(R_0)$.

If $S\in\DD_\sigma(R_0) \setminus \ttt(R_0)$,  we can find a maximal collection $\mathcal{F}_0$  of cubes $\wt S\in \ttt(R_0)$ such that 
$$
S=\bigcup_{\wt S\in\mathcal{F}_0} \wt S.
$$
Then, 
$$
\sum_{\substack{R\in\ttt(R_0) \\ R\subset S}}\sigma(R)=
\sum_{\substack{\wt S\in\mathcal{F}_0}} \, \sum_{\substack{R\in\ttt(\wt S): \\ R\subset \wt S}}\sigma(R) \lec 
\sum_{\substack{ \wt S\in\mathcal{F}_0}} \sigma(\wt S) = \sigma(S)
$$
and the proof is now complete. 
\end{proof}

\vvv

\begin{remark}\label{rem:Coronaunbounded} If $\supp \sigma$ is bounded, we can pick  $R_0=\supp \sigma$.  In the case that  $\supp \sigma$ is not bounded we apply a technique described in p.\ 38 of \cite{DS1}: we consider a family of cubes $\{R_j\}_{j\in J}\subset \DD_\sigma$ which are pairwise disjoint, whose union is all of $\supp\sigma$, and which have the property 
that for each $k$ there at most $C$ cubes from $\DD_{\sigma,k}$ not contained in any cube $R_j$. For each
$R_j$ we construct a family $\ttt(R_j)$ analogous to $\ttt(R_0)$.
Then we set 
$$\ttt:=\bigcup_{j\in J} \ttt(R_j)$$
and
 $$ 
 \mathcal{B}:=\{S  \subset\DD_\sigma: \,\textup{there does not exist}\, \, j \in J  \, \textup{such that}\, S \subset R_j\in \ttt\}.
 $$ 
One can easily check that the families $\ttt$ and $\mathcal B$ satisfy a Carleson packing condition.  See  \cite[p.\ 38]{DS1}
for the construction of the family $\{R_j\}$ and additional details. 
\end{remark}
\vvv

\section{$L^p$ and uniform  $\ve$-approximability of the regularized dyadic extension}
\vv


Given $A>1$, we say that two cubes $Q_1, Q_2$ are {\it $A$-close} if 
$$
\frac{1}{A}\diam Q_1 \leq \diam Q_2 \leq A \diam Q_1
$$
and 
$$
\dist(Q_1, Q_2) \leq A (\diam Q_1 + \diam Q_2).
$$

The following lemma was proved in   \cite[p. 60]{DS2}.
\begin{lemma}\label{lem:packing-A0}
If we have a Corona decomposition such that $\ttt \in \car(M_0)$ for some $M_0>0$,  then,  the collection of cubes 
\begin{align*}
\calA_0:=\{Q \in \DD_\sigma: &Q \in \tree(R)\,\textup{for}\, R\in \ttt \, \textup{and}\,\\& \exists \,Q' \in \tree(R')\,\textup{for some}\, R'\neq R\in \ttt,  
\,\textup{such that}\,Q' \,A\textup{-close to}\, Q \}
\end{align*} 
 is in $\car(M_1)$ for some $M_1>0$  depending on $M$, $A$, and the Ahlfors-regularity constants.
 \end{lemma}
 
 \vv
 
 \begin{lemma}[Lemma I.3.27,  p. 59 in \cite{DS2} ]\label{lem:packing-close to B}
 If $\mathcal{F} \subset \DD_\sigma$ is in $\car(M_1)$ for some $M_1>0$, then the family
  \begin{align*}
\mathcal{F}_A:=\{Q \in \DD_\sigma: Q \,\,\textup{is} \,A\textup{-close to some}\,\, Q' \in \mathcal F \}
\end{align*} 
 is in $\car(M_2)$ for some $M_2>0$  depending on $M_1$, $A$, and the Ahlfors-regularity constants.
 \end{lemma}
 
 \vv
 
 Let us define the family of Whitney cubes
\begin{align*}
\mathcal{P}_0:= \{P &\in \WW(\om):  \textup{there exists}\, P'\in \WW(\om)\,\textup{such that} \,1.2P \cap 1.2 P' \neq \emptyset \,\textup{and there exist}\\
& R, R' \in \ttt \, \textup{with} \, R \neq R' \, \textup{such that} \, b(P) \in \tree(R)\, \textup{and} \,b(P') \in \tree(R')\}.
\end{align*}  
Then, by the properties of Whitney cubes,   for every  $P \in \mathcal{P}_0$,  the cubes $P' \in \WW(\om)$ such that $b(P')$ is not in the same tree as $b(P)$ have the following properties:
 \begin{itemize}
 \item $\ell(b(P))/2 \leq \ell(b(P')) \leq 2 \ell(b(P))$
\item $\dist(b(P),b(P')) \leq C_1 \ell(b(P))$.
\end{itemize}
If, for fixed $R \in \ttt$, we define 
$$
\partial\tree(R):= \{ Q \in \tree(R): \,\textup{there exists}\, P \in \mathcal{P}_0 \, \textup{such that}\, b(P)=Q \}
$$
then, there exists $A>1$ large enough depending only on $C_1$ and $n$ such that 
 $$
\bigcup_{R \in \ttt} \partial\tree(R) \subset \calA_0,
 $$
 and,  by Lemma \ref{lem:packing-A0},  $\bigcup_{R \in \ttt} \partial\tree(R)  \in \car(M_1)$. If $\mathcal{F}$ is a family of ``true" dyadic cubes in 
 $\R^{n+1}$, we also define   
 \begin{align*}
\mathcal{N}(\mathcal{F}):= \{P &\in \WW(\om) :  \textup{there exists}\, P'\in \mathcal{F}\,\textup{such that} \,1.2P \cap 1.2 P' \neq \emptyset\}.
\end{align*}  
So, for $\mathcal F=\mathcal{P}_0$, we set 
$$
\partial\tree^*(R):=\{ Q \in \DD_\sigma : \exists\, P \in \mathcal{N}(\mathcal{P}_0)  \,\textup{such that}\,  Q=b(P) \in \tree(R) \}.
$$
It is easy to see that 
$$
\mathcal{J}:=\bigcup_{R \in \ttt}  \partial\tree^*(R)\subset \Big(\bigcup_{R \in \ttt} \partial\tree(R)\Big)_A
$$
and,  by Lemma \ref{lem:packing-close to B},  $\mathcal{J} \in  \car(M_2)$ for some $M_2>0$. Finally,  recall that $\calB$ is a collection of bad cubes satisfying a Carleson packing condition \eqref{eq:packing} and  define
$$
\mathcal{B}_0:= \{ P \in \WW(\om): b(P) \in \mathcal{B} \}.
$$

 \vvv
 We are now ready to define the {\it approximating function} of $\upsilon_f$ by
\begin{align}\label{eq:approximatingdefinition}
u(x)=\sum_{S \in \mathcal{B}_0} & m_{\sigma, b(S)}f\, \vp_S(x)\\
&+ \sum_{R\in \ttt}\Big[ \sum_{\substack{P\in\WW(\Omega) \setminus\mathcal{P}_0  \\ b(P)\in\tree(R)}}
 m_{\sigma, R}f\, \vp_P(x) + \sum_{\substack{P\in\mathcal{P}_0 \\ b(P)\in \tree(R)}} m_{\sigma, b(P)}f\, \vp_P(x)\Big],\notag
\end{align}
using  the Corona decomposition constructed in Section \ref{sec: large oscillation cubes}.   Note that when $\om$ is bounded,  $\ttt=\ttt(\pom)$ and $\mathcal{B}= \emptyset$,  while if $\pom$ is unbounded,  $\ttt$ and $\mathcal{B}$ are the families constructed in Remark \ref{rem:Coronaunbounded}.  Finally,  when $\om$ is an unbounded domain with compact boundary $\pom$,  we modify the definition of the approximating function as follows.
\begin{align}\label{eq:approximatingdefinition-unbounded-compact}
u(x)=& \sum_{P \in \WW(\om): \ell(P) \geq \diam(\pom)} m_{\sigma,  \pom} f \, \vp_P(x) \\
&+ \sum_{R\in \ttt}\Big[ \sum_{\substack{P\in\WW(\Omega) \setminus\mathcal{P}_0  \\ b(P)\in\tree(R)}}
 m_{\sigma, R}f\, \vp_P(x) + \sum_{\substack{P\in\mathcal{P}_0 \\ b(P)\in \tree(R)}} m_{\sigma, b(P)}f\, \vp_P(x)\Big].\notag
\end{align}

\vv

\begin{theorem}\label{thm: approxestimates}
Let $f\in L^1_{\loc}(\sigma)$ and $\ve>0$.  There exists $\alpha_0 \geq 1$ such that for any $\alpha\geq \alpha_0$ and any $\xi\in\pom$,
\begin{align}
\calN_\alpha(u - \upsilon_f)(\xi) +\calN_\alpha(\delta_\om \nabla( u - \upsilon_f) )(\xi) &\lec \ve \mathcal{M}f(\xi), \label{eq:pwMFu}\\
\calN_\alpha(\delta_\om \nabla u)(\xi) &\lec \mathcal{M}f(\xi)+\mathcal{M}(\mathcal{M}f)(\xi),\label{eq: ntmf-gradu estimate}
\\
\calC_s(\nabla u)(\xi) &\lec_\ve  \mathcal{M}(\wt{\mathcal{M}}f)(\xi) + \mathcal{M}(\wt{\mathcal{M}}( \mathcal{M}f))(\xi), \label{eq:pwCarlu}\\
\calC_s(\delta_\om \,|\nabla u|^2)(\xi) &\lec_\ve  \mathcal{M}((\wt{\mathcal{M}}f)^2)(\xi) + \mathcal{M}(\wt{\mathcal{M}}( (\mathcal{M}f)^2))(\xi). \label{eq:pwCarlu-bis}
\end{align}
Here $c_{\ve}$ is a positive constant depending on $\ve$ and $\alpha_0$ depend only on $n$ and the Ahlfors regularity,   the corkscrew condition, and the Whitney constants.

Let $f\in\BMO(\sigma)$ and $\ve>0$. Then, for any $x \in \om$,  it holds that 
\begin{align}\label{eq: pwMFBMO}
|u(x)-\upsilon_f(x)| + \delta_\om(x)|\nabla (u -\upsilon_f)(x)| &\leq \ve \|f\|_{\BMO(\sigma)}\\ 
\delta_\om(x) |\nabla u(x)| &\lec \|f\|_{\BMO(\sigma)},\label{eq: ntmf-gradu estimate bmo}
\end{align}
and for any $\xi \in \pom$,
\begin{equation}\label{eq: pwCarluBMO}
\calC_s(\nabla u)(\xi) + \left[ \calC_s(\delta_\om \,|\nabla u|^2)(\xi)\right]^{1/2}  \lec_{\ve} \|f\|_{\BMO(\sigma)},
\end{equation}
where the implicit constants depend  only on $n$ and the Ahlfors regularity,   the corkscrew condition, and the Whitney constants.

Moreover, if $f\in \Lip_c(\pom)$, then $u \in \Lip_{\textup{loc}}(\om)$ and for any $x\in\om$, 
\begin{equation}\label{eq: grad-approx-lip}
\delta_\om(x)|\nabla u(x)| \lec \Lip(f) \diam(\supp f).
\end{equation}
\end{theorem}

\begin{proof}
We will only  deal with the case that both $\om$ and $\pom$ are unbounded as the other cases can be treated in a similar but easier way. We remark first that if we choose $\alpha_0$ large enough, depending on $n$, the constants of the corkscrew condition and the Whitney decomposition,  the cone is always  non-empty and for every $Q \in \DD_\sigma$ such that $\xi \in Q$, there exists $P \in \WW(\om)$ such that $b(P)=Q$ and $P \subset \gamma_{\alpha_0}(\xi)$.  

For fixed $\xi \in \pom$  we let $x \in \gamma_\alpha(\xi)$ for $a \geq \alpha_0$.  There exists  $P_0\in\WW(\om)$ such that $x\in \bar P_0$ and we either have that  $ P_0 \in \mathcal{B}_0$ or that there is a unique $R_0 \in \ttt$ such that  $b(P_0)\in\tree(R_0)$.  If either  $P_0 \in \mathcal{P}_0$ and there does not exist any $P \in \mathcal{N}(\mathcal{P}_0) \setminus  \mathcal{P}_0 $ such that $x \in 1.1 P$,  or $P_0 \in \mathcal{B}_0$, it is easy to see that $u(x) - \upsilon_f(x)=0$,  while, if $P_0\in\PP_0$ and there 
exists some  $\wt{P}\in\calN(\PP_0)\setminus\PP_0$ such that $x\in 1.1\wt{P}$, then
$$
u(x) - \upsilon_f(x)= \sum_{\substack{P\in \mathcal{N}(\mathcal{P}_0 )  \setminus \mathcal{P}_0 \\ b(P) \in \tree(R_0) }} (m_{\sigma, R_0}f-m_{\sigma, b(P)}f) \,\vp_P(x).
$$
The same is true if $P_0 \in \mathcal{N}(\mathcal{P}_0) \setminus  \mathcal{P}_0 $ and there is $P\in  \mathcal{P}_0$  such that $x \in 1.1 P$. In any other case,  we have that
\begin{equation}\label{eq: u-uf in tree}
u(x) - \upsilon_f(x)= \sum_{\substack{P\in \WW(\om) \\ b(P) \in \tree(R_0)}} (m_{\sigma, R_0}f-m_{\sigma, b(P)}f) \,\vp_P(x).
\end{equation}
Therefore,  since $b(P) \in \tree(R_0)$,  by \eqref{stopcond}, 
\begin{align*}
|u(x)-\upsilon_f(x)| &\leq \ve \sum_{\substack{P\in\WW(\Omega) \\ b(P)\in\tree(R_0)}}  Mf(b(P)) \, \vp_P(x).
\end{align*}

For any $P \in \WW(\om)$ such that $x \in 1.1P \cap \gamma_\alpha(\xi)$,  since $|x-\xi| \approx \delta_\om(x) \approx  \ell(P)$, it holds that $P \subset B(\xi, M \ell(P))$. The same is true for any $S \in \DD_\sigma$ such that $b(P) \subset S$,  i.e., $S \subset B(\xi,  M' \ell(S))$, for a possibly larger constant $M'>0$  depending also on the Ahlfors regularity constants.  Thus,  
\begin{equation}\label{eq:pointwise-u-v(x)}
|u(x)-\upsilon_f(x)|  \lesssim  \ve \sup_{S \supset b(P)} m_{\sigma, B(\xi,  M'\ell(S))}(|f|) \lesssim  \ve \sup_{r\gtrsim \delta_\om(x)} m_{\sigma, B(\xi,  r)}(|f|).
\end{equation}
which, by taking supremum over all $x \in \gamma_\alpha(\xi)$,  implies \eqref{eq:pwMFu}.
By the same arguments and the fact the $\nabla \vp_P(x)\lec \ell(P)^{-1}\approx \delta_\om(x)^{-1}$, we  infer  that 
$$
\nabla (u-  \upsilon_f)(x) \lec \ve \mathcal{M}(f)(\xi) \delta_\om(x)^{-1}, 
$$
which  implies 
\begin{equation}\label{eq: ntmf-grad u-upsilonf est}
\calN_\alpha(\delta_\om \nabla(u-\upsilon_f))(\xi)\lec \ve \mathcal{M}(f)(\xi).
\end{equation}

In the case that $f\in\BMO(\sigma)$,  in view of  \eqref{eq: u-uf in tree} and $\nabla \vp_P(x)\lec \ell(P)^{-1}\approx \delta_\om(x)^{-1}$,  we have that
\begin{equation}\label{eq: ntmf plus delta gradient in bmo}
|u(x)-\upsilon_f(x)|+\delta_\om| \nabla(u-\upsilon_f)(x)|\lec \ve \|f\|_{\BMO}.
\end{equation}

We now turn our attention to the proof of \eqref{eq:pwCarlu} and \eqref{eq: pwCarluBMO}.  Let $x\in \bar P_0 \in \WW(\om)$.  
Then, once again,  either  there exists a unique $R_0 \in \ttt$ such that $b(P_0)\in \tree(R_0)$, or there exists $B_0\in \mathcal{B}_0$ such that $b(P_0)=B_0$.  For the sake of brevity, we denote 
\begin{equation}\label{Bx}
B_x:=c \, B^x=B(x, c\,\delta_\om(x)),
\end{equation}
for a small enough constant $c>0$ to be chosen.  Fix $y\in B_x$ and if
$P\in\WW(\om)$ is such that $y \in 1.1 P$,  then  $x\in 1.2P$.
Indeed,  by \eqref{eq:delta<ellP}, we always have  that $\dist(x, 1.1 P)\leq |x-y|\leq  c \,\delta_\om(x) \leq c\, \Lambda' \, \ell(P_0)$, and if there exists $ P\in \mathcal{W}(\om)$ such that $y \in 1.1 P$ and $x\notin1.2P$ it also holds that 
$\dist(x, 1.1 P)\geq 0.1\ell(P)$. Now, note that since $\frac{1}{2}\ell(P_0)\leq \ell(P) \leq 2 \ell(P_0)$,  we get that $\frac{1}{2}\ell(P_0) \leq c\, \Lambda' \, \ell(P_0)$  and if we  choose  $c=\frac{1}{4\Lambda'}$ we reach a contradiction.

It is easy to see that $\nabla u(y)=0$ if there does not exist any cube $P \in \mathcal{P}_0$  or $P \in \mathcal{B}_0$  
 such that $y \in 1.1 P$.    Using that $ \sum  \nabla \vp_P(y)=0$, we have  that
\begin{align}
\nabla u(y)&= \Big( \sum_{P\in\mathcal{B}_0} + \sum_{\substack{P\in\calN(\PP_0) \\ b(P)\notin \tree(R_0)}} \Big) (m_{\sigma, b(P)}f-m_{\sigma, b(P_0)}f) \nabla \vp_{P}(y)   \label{eq:gradient2-approx}\\
&+ \sum_{\substack{P\in \PP_0 \\ b(P)\in\tree(R_0)}}(m_{\sigma, b(P)}f - m_{\sigma, b(P_0)}f)\nabla \vp_{P}(y)  \notag\\
&+\sum_{\substack{P\in\calN(\PP_0)\setminus \PP_0 \\ b(P)\in \tree(R_0)}}(m_{\sigma, R_0}f-m_{\sigma, b(P)}f)\nabla\vp_P(y).\notag
\end{align}
Therefore, by Remark \ref{rem:whitney overlap}, arguing as in the proof of \eqref{eq:pointwise-u-v(x)} and 
using  \eqref{eq:gradient2-approx},  the fact that $\ell(P_0)\approx\ell(P)$ for any  $P\in\WW(\om)$ such that $1.1P\ni y$, and the 
Carleson packing of the cubes in $\mathcal{B}_0$,  for fixed $\xi\in\pom$ and $r>0$, we  can estimate
\begin{align}
& \int_{B(\xi, r)\cap\om}\, \sup_{y\in B_x} |\nabla u(y)| \frac{dx}{\delta_\om(x)^{n-s}} \notag\\
& \lec \Big( \sum_{P\in\mathcal{B}_0} +\sum_{R \in \ttt} \sum_{\substack{P\in\calN(\PP_0) \\ b(P)\notin \tree(R)}} \Big) \int_{ P\cap B(\xi, r)}
m_{\sigma, B(x_{b(P)}, C_w \ell(P))}(|f|)\frac{dx}{\ell(P)^{n+1-s}} \notag\\
&+ 2\ve\sum_{R\in\ttt} \sum_{\substack{P\in \mathcal{P}_0 \\ b(P) \in \tree(R)}}
\int_{ P\cap B(\xi, r)} \sup_{\rho \gtrsim \ell(P)} m_{\sigma, B(x_{b(P)},  \rho)}(|f|) \frac{dx}{\ell(P)^{n+1-s}}   \notag\\
&+ \ve \sum_{R\in\ttt} \, \sum_{\substack{P\in\calN(\PP_0)\setminus \PP_0 \\ b(P)\in\tree(R)}}
\int_{P\cap B(\xi, r)} \sup_{\rho \gtrsim \ell(P)} m_{\sigma, B(x_{b(P)},  \rho)}(|f|) \frac{dx}{\ell(P)^{n+1-s}} \notag \\
&\lesssim \Big(\sum_{\substack{P\in \mathcal{B}_0  \\ P\cap B(\xi,r) \neq \emptyset}} +\sum_{R \in \ttt} \sum_{\substack{P\in\calN(\PP_0) \\ b(P)\notin \tree(R)  \\ P\cap B(\xi,r) \neq \emptyset}} \Big)\,\sigma(b(P))\,
m_{\sigma, B(x_{b(P)}, C_w \ell(P))}(|f|) \notag \\
& +  \sum_{R\in\ttt}\sum_{\substack{P\in \calN( \PP_0 ) \\ b(P)\in\tree(R) \\ P\cap B(\xi,r) \neq \emptyset}} \,
\sigma(b(P))\, \inf_{\zeta\in B(x_{b(P)}, M\ell(P))}\mathcal{M}f(\zeta)  \notag\\
&\leq \sum_{\substack{Q\in \mathcal{B} \cup  \mathcal{J} \\ Q\subset B(\xi,C' r)}}\sigma(Q) \, m_{\sigma, B(x_Q, C_w \ell(Q))}(|f|)  
+  \sum_{\substack{Q\in \mathcal{J} \\ Q\subset B(\xi,C' r)}}\sigma(Q) \, m_{\sigma, B(x_Q , M \ell(Q))}(\mathcal{M}f ) \notag\\
&\lesssim \int_{B(\xi, Cr)}\sup_{Q \ni z} \, m_{\sigma, B(x_Q, C_w \ell(Q))}(|f|) \, d\sigma(z) 
+ \int_{B(\xi, Cr)}\sup_{Q \ni z} \,  m_{\sigma, B(x_Q, M \ell(Q))}(\mathcal{M}f) \, d\sigma(z) \notag\\
&\lesssim  \int_{B(\xi, Cr)}  \wt{\mathcal{M}}f \, d\sigma + \int_{B(\xi, Cr)}\wt{\mathcal{M}}(\mathcal{M}f) \, d\sigma,\notag
 \end{align}
for $M>1$ possible larger than $C_w$ and where in the antepenultimate  inequality we used that if $P \cap B(\xi,r) \neq \emptyset$, then $b(P) \subset B(\xi, C'r)$ for some large constant $	C'>0$  depending on Ahlfors-regularity and the Whitney constants,  while  the penultimate inequality follows from Carleson's embedding theorem (see \cite[Theorem 5.8, p. 144]{Tol})
since the families  $\mathcal{J}=\cup_{R \in \ttt} \partial^*\tree(R)$ and $\mathcal{B}$ are Carleson  families.  This readily concludes \eqref{eq:pwCarlu},  while the proof of \eqref{eq:pwCarlu-bis} follows by similar arguments.  We omit the details.

If $f\in\BMO(\sigma)$, using \eqref{eq:gradient2-approx}, for $\xi\in \pom$ and $r>0$,
\begin{align*}
\int_{B(\xi, r)\cap\om}\, \sup_{y\in B_x}& |\nabla u(y)|  \frac{dx}{\delta_\om(x)^{n-s}} \lec_{\ve} 
\sum_{R\in\ttt}\sum_{\substack{P \in \calN(\PP_0) \\ b(P')\in\tree(R)}}\int_{P \cap B(\xi, r)}  \|f\|_{\BMO(\sigma)}\,
\hm_s(x)\,dx 
\end{align*}
\begin{align*}
&+\Big( \sum_{P\in\mathcal{B}_0} +\sum_{R \in \ttt} \sum_{\substack{P\in\calN(\PP_0) \\ b(P)\notin \tree(R)}} \Big)  \int_{P \cap B(\xi, r)}  \|f\|_{\BMO(\sigma)} \frac{\hm_s(x)}{\ell(P)}\,dx\\
&\lec \|f\|_{\BMO(\sigma)} \,  \sum_{\substack{P \in \calN(\PP_0) \cup \mathcal{B}_0\\  P \cap B(\xi, r)\neq \emptyset}} \,\sigma(b(P)) 
\leq \|f\|_{\BMO(\sigma)}\, \sum_{\substack{Q\in\mathcal{J}\cup\mathcal{B} \\ Q\subset B(\xi, Mr)}}\sigma(Q) \\
&\lec \|f\|_{\BMO(\sigma)} \, r^s,
\end{align*}
for $M>1$ large enough constant depending on the Ahlfors regularity and the Whitney constants. For the last inequality we used that the families of surface cubes $\mathcal{J}$ and $\mathcal{B}$  satisfy Carleson packing condition from Lemma  \ref{lem:packing-close to B}.  Similarly,  we can show that 
$$
\int_{B(\xi, r)\cap\om}\, \sup_{y\in B_x} \delta_\om(y)|\nabla u(y)|^2  \frac{dx}{\delta_\om(x)^{n-s}} \lec_{\ve}  \|f\|^2_{\BMO(\sigma)} \, r^s
$$
The two estimates above obviously imply \eqref{eq: pwCarluBMO}.

Let $\xi \in \pom$ and $x\in\gamma_\alpha(\xi)$.  There exists $P_0 \in \WW(\om)$ such that $x \in \bar{P}_0$.  Then  using \eqref{eq:gradient2-approx} and the bounded overlaps of the Whitney cubes,  it holds 
\begin{align}\label{eq:gradient-bound-approx}
|\nabla u(x)| &\lesssim \ell(P_0)^{-1} (m_{\sigma,   B_{P_0}}(|f|)  +  \ve \sup_{\rho \gec \ell(P_0)} m_{\sigma, B(x_{b(P_0)},  \rho)}(|f|))\notag\\
&\leq {\delta_\om(x)}^{-1} \big(m_{\sigma,  B(\xi,  C' \ell(P_0))} (|f|) + \inf_{\zeta\in B(\xi,  C' \ell(P_0))}\mathcal{M}f(\zeta)\big) \\
&\lec \delta_\om(x)^{-1}\big( \mathcal{M}f(\xi)+ \mathcal{M}(\mathcal{M}f)(\xi)\big). \notag
\end{align}
By a similar but easier argument,  we can show that
\begin{equation}\label{eq: est-gradu-bmo}
|\nabla u(x)|\lec \|f\|_{\BMO(\sigma)} \delta_\om(x)^{-1}.
\end{equation} 
Since $\sup_{x\in\om}\delta_\om(x)|\nabla u(x)|=\sup_{\xi\in\pom}\sup_{x\in\gamma_\alpha(\xi)}\delta_\om(x)|\nabla u(x)| $, 
it easily follows that the estimates 
\eqref{eq:gradient-bound-approx} and \eqref{eq: est-gradu-bmo} imply the estimates \eqref{eq: ntmf-gradu estimate} and 
\eqref{eq: ntmf-gradu estimate bmo} respectively. 

It remains to prove \eqref{eq: grad-approx-lip} in the case that $f\in \Lip_c(\pom)$. 
Using  \eqref{eq:gradient2-approx}  and the bounded overlaps of the Whitney cubes,  it holds that for
any $x\in \om$,
\begin{equation}\label{eq: gradu-f lip}
|\nabla u(x)|\lec \Lip(f)\diam(\supp f)\ell(P)^{-1} 
+ \ve \sum_{\substack{P\in\calN(\PP_0)\setminus\PP_0 \\ b(P)\in\tree(R_0) \\ x \in 1.1 P} } Mf(b(P))\ell(P)^{-1}. 
\end{equation}
Since $f$ has compact support, if  $\xi_0\notin \supp f$,  for every $Q\supset b(P)$, we have that
$$
m_{\sigma, Q}(|f|)=m_{\sigma, Q} ( |f - f(\xi_0)| )\lec \diam(\supp f)\Lip(f).
$$
Taking supremum over all cubes $Q\supset b(P)$ and using again the bounded overlaps of the Whitney cubes, by
\eqref{eq: gradu-f lip} and the fact that $\delta_\om(x) \approx \ell(P)$ for all $P \in \mathcal{W}(\om)$ such that $x \in 1.1P$, we infer that 
$$
|\nabla u(x)|\lec \Lip(f)\diam(\supp f)\delta_\om(x)^{-1},
$$
and this ends the proof.
\end{proof}

As a corollary we get that if $f \in L^p(\sigma)$,  $p\in(1, \infty)$,  (resp. $f \in \BMO(\sigma)$),  then $\upsilon_f$ is $\ve$-approximable in $L^p$ (uniformly $\ve$-approximable).

\begin{theorem}\label{thm:e-approxbmo} If $f\in\BMO(\sigma)$, for any $\ve>0$, there exists $u=u_\ve \in C^\infty(\om)$, $\alpha_1 \geq 1$,  and  a constant $c_\ve>1$, such that for any $\alpha \geq \alpha_1$, it holds that
\begin{align}\label{eq: Linf-ntMF-bmo}
\sup_{x \in \om} |v_f(x) - \upsilon_f)(x) | + \sup_{x \in \om} \,\delta_\om(x)|\nabla(v_f-\upsilon_f))(x)|  &\lec \ve \|f\|_{\BMO(\sigma)},\\
\sup_{\xi \in \pom}  \calC_s(\nabla u)(\xi) + \sup_{\xi \in \pom} \left[\calC_s(\delta_\om \, |\nabla u|^2)(\xi)\right]^{1/2} &\lec_{\ve} \|f\|_{\BMO(\sigma)},\label{eq: Linf-carlu-bmo}\\
 \sup_{x \in \om} \,\delta_\om(x)|\nabla u(x)|  & \lec \|f\|_{\BMO(\sigma)},\label{eq: Linfty-ntmf-grad u est-bmo}
\end{align}
where the implicit constants depend on $s$, $n$,  and the constants of the Ahlfors regularity,  the corkscrew condition, and the Whitney decomposition.  
\end{theorem}

\begin{proof}
The result follows immediatelly by the  estimates \eqref{eq: pwMFBMO}, 
  \eqref{eq: ntmf-gradu estimate bmo}  and \eqref{eq: pwCarluBMO}.
\end{proof}

\vv

\begin{theorem}\label{thm:e-approxLp}
If $f\in L^p(\sigma)$,  for $p \in (1,\infty]$,  then for any $\ve>0$, there exists $u=u_\ve \in C^\infty(\om)$, $\alpha_0 \geq 1$,  and  a constant $c_\ve>1$, such that for any $\alpha \geq \alpha_0$, it holds that
\begin{align}
\| \calN_\alpha(u-\upsilon_f)\|_{L^p(\sigma)}+\| \calN_\alpha(\delta_\om \nabla(u-\upsilon_f)\|_{L^p(\sigma)} &\lesssim \ve\,\| f\|_{L^p(\sigma)}, \label{eq:LpntMfu}\\
\|\calN_\alpha(\delta_\om \nabla u)\|_{L^p(\sigma)} &\lec \|f\|_{L^p(\sigma)},\label{eq: Lp-ntmf-grad u estimate}\\
\|\calC_s(\nabla u)\|_{L^p(\sigma)} &\lec  \ve^{-2} \, \| f \|_{L^p(\sigma)}, \label{eq:LpCarlu}\\
\|\left[\calC_s(\delta_\om \, |\nabla u|^2)\right]^{1/2}\|_{L^q(\sigma)} &\lec  \ve^{-2} \, \| f \|_{L^q(\sigma)},\,\,\,q \in[2,\infty),\label{eq:LpCarlu-bis}
\end{align}
where the implicit constants depend on $s$, $n$, $p$,  $q$,  and the constants of the Ahlfors regularity, the corkscrew condition, and the Whitney decomposition.  
\end{theorem}

\begin{proof} 
The proof is an immediate consequence of  \eqref{eq:pwMFu},   \eqref{eq: ntmf-gradu estimate},  \eqref{eq:pwCarlu},   \eqref{eq:pwCarlu-bis},  Lemma \ref{lem: aperture},  and the fact that $\mathcal{M}$ and  $\wt{\mathcal{M}}$ are $L^p(\sigma) \to L^p(\sigma)$-bounded for any $p\in(1, \infty)$.  In order to prove \eqref{eq:LpCarlu-bis} for $q=2$,  we should estimate it directly using \eqref{eq:gradient2-approx} as in the proof of \eqref{eq:pwCarlu} along with the fact that
\begin{equation}\label{eq:comparable L2C}
\|\calC_s(\delta_\om|\nabla F|^2)^{1/2})\|_{L^2(\sigma)}^2 \approx \int_{\om} \delta_\om(x) \sup_{y \in B_x}|\nabla F(y)|^2 \,\frac{dx}{\delta_\om(x)^{n-s}}.
\end{equation}
Indeed,
\begin{align*}
\int_{\om} \delta(x) &\sup_{y \in B_x}|\nabla u(y)|^2\, \frac{dx}{\delta_\om(x)^{n-s}} \lec \sum_{P \in \WW(\om)} \int_P \ell(P) \sup_{y \in B_x}|\nabla u(y)|^2 \,\frac{dx}{\delta_\om(x)^{n-s}}\\
& \lec \Big( \sum_{P\in\mathcal{B}_0} +\sum_{R \in \ttt} \sum_{\substack{P\in\calN(\PP_0) \\ b(P)\notin \tree(R)}} \Big) \int_{ P}
m_{\sigma, B(x_{b(P)}, C_w \ell(P))}(|f|)^2\,\frac{dx}{\ell(P)^{n+1-s}} \notag\\
&+ 2\ve^2\sum_{R\in\ttt} \sum_{\substack{P\in \mathcal{P}_0 \\ b(P) \in \tree(R)}}
\int_{ P} \sup_{\rho \gtrsim \ell(P)} m_{\sigma, B(x_{b(P)},  \rho)}(|f|)^2 \frac{dx}{\ell(P)^{n+1-s}}   \notag
\end{align*}
\begin{align*}
&+ \ve^2 \sum_{R\in\ttt} \, \sum_{\substack{P\in\calN(\PP_0)\setminus \PP_0 \\ b(P)\in\tree(R)}}
\int_{P} \sup_{\rho \gtrsim \ell(P)} m_{\sigma, B(x_{b(P)},  \rho)}(|f|)^2 \frac{dx}{\ell(P)^{n+1-s}} \notag \\
&\lesssim \Big(\sum_{P\in \mathcal{B}_0} +\sum_{R \in \ttt} \sum_{\substack{P\in\calN(\PP_0) \\ b(P)\notin \tree(R)}} \Big)\,\sigma(b(P))\,
m_{\sigma, B(x_{b(P)}, C_w \ell(P))}(|f|)^2 \notag \\
& +  \sum_{R\in\ttt}\sum_{\substack{P\in \calN( \PP_0 ) \\ b(P)\in\tree(R}} \,
\sigma(b(P))\,  \sup_{\rho \gtrsim \ell(P)} m_{\sigma, B(x_{b(P)},  \rho)}(|f|)^2   \notag\\
&\lec \sum_{\substack{Q\in \mathcal{B} \cup  \mathcal{J}}} \sigma(Q) \, \sup_{\rho \gtrsim \ell(Q)} m_{\sigma, B(x_{Q},  \rho)}(|f|)^2\notag\\
&\lesssim \int_{\pom}\sup_{Q \ni z}\, \sup_{\rho \gtrsim \ell(Q)} m_{\sigma, B(x_{Q},  \rho)}(|f|)^2 \, d\sigma(z)  \lesssim  \int_{\pom}  \wt{\mathcal{M}}f^2 \, d\sigma \lec \|f\|^2_{L^2(\sigma)}.
\end{align*}
\end{proof}

\vv
\begin{remark}
Note that since $L^\infty(\sigma) \subset \BMO(\sigma)$,  the estimates \eqref{eq:LpntMfu},  \eqref{eq: Lp-ntmf-grad u estimate}, and \eqref{eq:LpCarlu} for $p=\infty$,  follow from \eqref{eq: Linf-ntMF-bmo},   \eqref{eq: Linf-carlu-bmo}, and  \eqref{eq: Linfty-ntmf-grad u est-bmo}.
\end{remark}

\vvv

\section{Construction of Varopoulos-type extensions of $L^p$  and $\BMO$ functions}\label{sec: final extension}

\vv

We will first construct extensions of boundary functions that belong to  $\BMO$.  

\vv

\noindent{\bf Hypothesis} \setword{{[$\wt \TT$]}}{Word:Hypothesis T'}
\begin{itemize}
\item[(i)] There exists a bounded linear {\it trace} operator 
$$
\tr: \cc^{1, \infty}_{s, \infty}(\om) \to \BMO(\sigma)
$$
 such that $\| \tr(w)\|_{\BMO(\sigma)} \lec \| \nabla w \|_{\ccisi(\om)}$.  
\item[(ii)]  If $\upsilon_f$  is the regularized dyadic extension of $f \in \BMO(\sigma)$,  then $\tr( \upsilon_f)(\xi)=f(\xi)$ for $\sigma$-a.e. $\xi \in\pom$, and for any $w \in \ccisi(\om)$, it holds that
\begin{equation}\label{eq:lower bound-ftru-bmo}
\|f-\tr(w)\|_{\BMO(\sigma)} =\|\tr( \upsilon_f - w)\|_{\BMO(\sigma)} \lec  \sup_{x \in \om} |  \upsilon_f(x) - w(x)|.
\end{equation}
\end{itemize}

\begin{theorem}\label{thm: extension in bmo}
Let $\om \in \AR(s)$ for $s \in (0,n]$  satisfying  the \textup{{Hypothesis}  \ref{Word:Hypothesis T'}}.
If $f\in\BMO(\sigma)$, then there exist a function $u:\om\to\R$ and a constant $c_0 \in (0, \frac{1}{2}]$ such that for every $c \in (0, c_0]$, it holds that
\begin{itemize}
\item[(i)] $u\in C^1(\om)$.
\item[(ii)] $\displaystyle \sup_{\xi\in\pom}\calN_{\sharp,c}(u)(\xi)+\sup_{x\in\om}\delta_\om(x) |\nabla u(x)|    \lec \|f\|_{\BMO(\sigma)}$,
\item[(iii)] $\displaystyle \sup_{\xi\in\pom}\calC_{s,c}(\nabla u)(\xi) \lec \|f\|_{\BMO(\sigma)}$, and
\item[(iv)] $\tr(u)(\xi)=f(\xi)$ for $\sigma$-a.e. $\xi\in\pom$.
\end{itemize}
\end{theorem}

\begin{proof}
 If $f\in\BMO(\sigma)$ and $\upsilon_f$  its regularized dyadic extension, we apply  Theorem \ref{thm:e-approxLp} and construct  the $\ve$-approximating function of $\upsilon_f$, which we denote by $u_0$. In light of \eqref{eq: Linf-carlu-bmo}, and Hypothesis \ref{Word:Hypothesis T'},  we have that the trace $\tr(u_0)$ exists and it is in $\BMO(\sigma)$. We set 
$$
f_1:= f - \tr(u_0).
$$
Inductively, for every $k\geq 1$, we define $u_k$ to be the $\ve$-approximating function of $\upsilon_{f_k}$ and set 
$$
f_{k+1} := f_k-\tr(u_k),
$$
Therefore by  \eqref{eq:lower bound-ftru-bmo} and \eqref{eq: Linf-ntMF-bmo},  we have that 
 $$
 \|f_{k+1}\|_{\BMO(\sigma)}\lec \sup_{x \in \om} |  \upsilon_{f_k}(x) - u_k(x)| \lec \ve\|f_k\|_{\BMO(\sigma)},
 $$
 which implies that 
 \begin{equation}\label{eq: fk bmo control}
 \|f_{k+1}\|_{\BMO(\sigma)}\leq C\ve \|f_k\|_{\BMO(\sigma)}\leq \cdots \leq (C\ve)^{k+1} \|f\|_{\BMO(\sigma)}.
 \end{equation}
Assume that $C\ve\leq 1/2$ and set $S_k:=\sum_{j=0}^k u_j$,  for any positive integer $k$. Using \eqref{eq: Linf-ntMF-bmo} and
\eqref{eq: sharpntMFuf-bmo},  and finally \eqref{eq: fk bmo control},  we can estimate
\begin{align*}
&\| S_k-S_m\|_{\nn^\infty_{\textup{sum}}(\om)}= \|S_k-S_m\|_{\nnis(\om)} +\| \delta_\om\nabla S_k-   \delta_\om \nabla S_m\|_{\nni(\om)}\\
&\lec \sum_{j=k+1}^m \Big(\sup_{x \in \om} |  \upsilon_{f_j}(x) - u_j(x)|+\sup_{x \in \om} \delta_\om(x)| \nabla( \upsilon_{f_j} - u_j)(x)|\Big) \\
&+ \sum_{j=k+1}^m \Big(\sup_{\xi\in\pom}\calN_{\sharp}(\upsilon_{f_j})(\xi) + \sup_{x \in \om} \delta_\om(x) |\nabla u_j(x)| \Big) \\
&\lec \sum_{j=k+1}^m \left(\ve \|f_j\|_{\BMO(\sigma)} + \|f_j\|_{\BMO(\sigma)} + \|f_j\|_{\BMO(\sigma)} \right) 
\lec \sum_{j=k+1}^m(C\ve)^j\|f\|_{\BMO(\sigma)} \\
&\leq ( 2^{-k} - 2^{-m} ) \|f\|_{\BMO(\sigma)}.
\end{align*}
Thus, $S_k$ is Cauchy sequence in the $(\nn_{\textup{sum}}(\om), \|\cdot\|_{\textup{sum}})$ and so, by Corollary \ref{cor: the sum space is complete}, there exists $u \in \nn_{\textup{sum}}(\om)$ such that $S_k \to u$ in $\nn_{\textup{sum}}(\om)$ (this implies that $\nabla S_k \to \nabla u$ locally uniformly in $\om$).   By
 \eqref{eq: Linf-carlu-bmo} and \eqref{eq: Linfty-ntmf-grad u est-bmo},  for any $m>k$, we have that
\begin{align*}
\|\nabla S_k- \nabla S_m\|_{\ccisi(\om)} &\leq \sum_{j=k+1}^m \sup_{\xi\in\pom}\calC_{s, c} (\nabla u_j)(\xi) \\
&\lec \sum_{j=k+1}^m\|f_j\|_{\BMO(\sigma)} \leq ( 2^{-k} - 2^{-m} ) \|f\|_{\BMO(\sigma)}.
\end{align*}
Therefore, since $\ccisi(\om)$ is a Banach space,   there exists $\vec F\in\ccisi(\om)$ such that 
 $\nabla S_k \to \vec F$ in  $\ccisi(\om)$ and, by Lemma  \ref{lem: sup gradient bound},  for any fixed $x\in\om$,
$$
\sup_{y\in B_x} |\nabla S_k(y) -  \vec F(y)| \delta_\om(y) \to 0, \quad \textup{as}\,\, k \to \infty,
$$
which implies that $\vec F=\nabla u$ in $\om$. Therefore,   
$$
 \nabla S_k \to  \nabla u(x) \quad \textup{in}\,\,  \ccisi(\om) \cap \nn^\infty_{\textup{sum}}(\om).
$$
Furthermore, we have that 
\begin{equation}\label{eq: carlgradu-bmo}
\sup_{\xi\in\pom}\calC_s(\nabla u)(\xi)\leq\sum_{k=0}^\infty \sup_{\xi\in\pom} \calC_s(\nabla u_k)(\xi)   
\lec_{\ve}  \sum_{k=0}^\infty \|f_k\|_{\BMO(\sigma)} \lec \|f\|_{\BMO(\sigma)},
\end{equation}
and, similarly,
\begin{align} \label{eq: ntMFu-bmo}
\|  u \|_{\nn^\infty_{\textup{sum}}(\om)}  \leq \sum_{k=0}^\infty \|  u_k \|_{\nn^\infty_{\textup{sum}}(\om)}  \lec  \sum_{k=0}^\infty \|f_k\|_{\BMO(\sigma)} \lec  \|f\|_{\BMO(\sigma)}.
\notag\end{align}
Finally,  it holds  that
$$
0=\lim_{k \to \infty} \|f_{k+1}\|_{\BMO(\sigma)}= \|f - \sum_{j=0}^{k} \tr(u_j)\|_{\BMO(\sigma)}
 $$
 and since, by construction, $\sum_{j=0}^k u_j-u \in \cc^{1, \infty}_{s, \infty}(\om)$,  using the linearity of the trace and   Hypothesis \ref{Word:Hypothesis T'}, we have that 
\begin{align*}
\Big\|\tr(u)-\sum_{j=0}^k \tr(u_j)\Big\|_{\BMO(\sigma)} &= \left\|\tr\big(u-S_k \big)\right\|_{\BMO(\sigma)} \lec \left\|\nabla u - \nabla S_k\right\|_{\ccisi(\om)} \overset{k \to \infty}{\longrightarrow} 0.
\end{align*}
This gives that $u$ is an extension of $f\in\BMO(\sigma)$ with $\tr(u)=f $ in $ \BMO(\sigma)$ and so 
\begin{equation}\label{eq: trace of extension in bmo}
\tr(u)(\xi)=f(\xi)+c \,\,\, \textup{for}\,\,\sigma\textup{-almost every}\,\xi\in \pom.
\end{equation}
If we set $\wt u=u-c$,  it is clear that $\wt u $ satisfies (i), (ii),  (iii), and (iv) concluding the proof of the theorem.
\end{proof}

\vv

\begin{remark}\label{rem:v-w}
Note that by the proof of  Theorem \ref{thm: extension in bmo},  it is immediate that the extension $u$ satisfies the estimate
\begin{equation}\label{eq:v-w}
\sup_{x \in \om} |\upsilon_f(x) - u(x)| \lec \|f\|_{\BMO(\sigma)}. 
\end{equation}
\end{remark}

\vv

We shall now deal with boundary data in $L^p(\sigma)$.

\vv

\noindent{\bf Hypothesis} \setword{{[$\TT$]}}{Word:Hypothesis T} 
\begin{itemize}
\item[(i)] For $p \in (1,\infty]$,  there exists a bounded linear {\it trace} operator 
$$
\tr:\nnp(\om) \cap \cc^{1, p}_{s, \infty}(\om)  \to L^p(\sigma)
$$
  such that $\| \tr(w)\|_{L^p(\sigma)} \leq \, \|w\|_{\nnp(\om)}$.  
\item[(ii)]  If $\upsilon_f$  the regularized dyadic extension of $f \in L^p(\sigma)$,  then $\tr(\upsilon_f)(\xi)=f(\xi)$ for $\sigma$-a.e. $\xi \in\pom$, and for any $w \in \nnp(\om)\cap \ccpsi(\om)$, it holds that
\begin{equation}\label{eq:lower bound-ftru}
\|f-\tr(w)\|_{L^p(\sigma)} =\|\tr(\upsilon_f - w)\|_{L^p(\sigma)} \leq \, \|\upsilon_f - w\|_{\nnp(\om)}.
\end{equation}
\end{itemize}

\vv

\begin{theorem}\label{thm: extension in Lp}
Let $\om \in \AR(s)$ for $s \in (0,n]$  satisfying  the \textup{{Hypothesis}  \ref{Word:Hypothesis T}}.
If $ f \in L^p(\sigma) $ with $p\in(1, \infty]$, then there exists a function $ u : \om \to \R $ such that 
\begin{enumerate}
\item[(i)] $ u \in C^1(\om) $,
\item[(ii)] $ \tr(u)(\xi) = f(\xi)$ for $\sigma$-a.e. $\xi \in\pom$,
\item[(iii)] $\| \calN(u) \|_{L^p(\sigma)}  + \|\calN(\delta_\om\nabla u)\|_{L^p(\sigma)} \lec \| f \|_{L^p(\sigma)}$,
\item[(iv)] $  \| \calC_s( \nabla u) \|_{L^p(\sigma)}  \lec \|f\|_{L^p(\sigma)}$.  
\end{enumerate}
\end{theorem}

\begin{proof}
Fix $\ve>0$ to be chosen.  By  Theorem \ref{thm:e-approxLp},  we construct  $u_0$,  the $\ve$-approximating function of $\upsilon_f$, and by Hypothesis \ref{Word:Hypothesis T},  the trace  of $u_0$  exists and $\tr(u_0) \in L^p(\sigma)$. 
We set 
$$ f_1:= f - \tr(u_0) \in L^p(\sigma).$$ 
We then let $u_1$  be the $\ve$-approximating function of $\upsilon_{f_1}$ and set  
$$ f_2:= f_1 -  \tr(u_1) \in L^p(\sigma).$$ 
Inductively,  for every $k\geq1$, we define $u_k$ to be the $\ve$-approximating function of $\upsilon_{f_k}$ and set
$f_{k+1}:=f_k-\tr(u_k)$.
Therefore,  by \eqref{eq:lower bound-ftru} and \eqref{eq:LpntMfu}, we have that
$$
\|f_{k+1}\|_{L^p(\sigma)}  \leq \| \upsilon_{f_k}-u_k \|_{\nnp(\om)} \leq C\, \ve \| f_k \|_{L^p(\sigma)},
$$
which implies that
\begin{align}\label{eq:f_k+1<ef_k}
\|f_{k+1}\|_{L^p(\sigma)} \leq C\, \ve \, \|f_k\|_{L^p(\sigma)}\leq \dots \leq (C \ve)^{k+1} \|f\|_{L^p(\sigma)}.
\end{align}
Thus, if we choose $\ve$ so that $C\ve\leq \frac{1}{2}$ and  set $S_k:=\sum_{j=0}^k u_j $,  then for $k<m$,  using  \eqref{eq:LpntMfu}, \eqref{ntMFuf}, and \eqref{eq:f_k+1<ef_k},   it holds that 
\begin{align*}
\| S_k-  &S_m \|_{\nnp(\om)}\leq \sum_{j=k+1}^m \left(\|\mathcal{N}(u_j-\upsilon_{f_j})\|_{L^p(\sigma)} +
 \|\mathcal{N} \upsilon_{f_j} \|_{L^p(\sigma)}\right)\\
&\lesssim \sum_{j=k+1}^m\left( \ve \| f_j\|_{L^p(\sigma)} + \| f_j\|_{L^p(\sigma)} \right)  \leq (1+\ve)  \sum_{j=k+1}^m (C\ve)^{j} \| f \|_{L^p(\sigma)}\\
&\leq (2^{-k+1}-2^{-m+1})\| f \|_{L^p(\sigma)}.
\end{align*}
Thus,  $S_k$ is a Cauchy sequence in $\nnp(\om)$,  and since $\nnp(\om)$  is  a Banach space,   there exists $u \in \nnp(\om)$ such that $S_k \to u$ in $\nnp(\om)$. It is easy to see that $S_k \to u$ uniformly in $B_x$,  for any $ x\in\om$,  and so we define
\begin{equation}\label{eq: extension}
u(x)=\sum_{k=0}^\infty u_k(x) \quad \textup{for all}\,\,x \in \om.
\end{equation}
Similarly,  we can show that $\nabla S_k=\sum_{j=0}^k \nabla u_j $  is convergent in the Banach space  $\ccpsi(\om)$ (resp.   $\nnp(\om)$), since,  by \eqref{eq:LpCarlu}  (resp. \eqref{eq: Lp-ntmf-grad u estimate}) and \eqref{eq:f_k+1<ef_k}, 
\begin{align*}
\| \nabla S_k &-  \nabla S_m \|_{\ccpsi(\om)} +\| \delta_\om \nabla S_k-  \delta_\om \nabla S_m \|_{\nnp(\om)} \\
&  \leq \sum_{j=k+1}^m \|\mathcal{C}_s(\nabla u_j) \|_{L^p(\sigma)}+ \sum_{j=k+1}^m \|\mathcal{N}(\delta_\om |\nabla u_j|) \|_{L^p(\sigma)} \leq C \, \ve^{-2} \sum_{j=k+1}^m  \| f_j\|_{L^p(\sigma)}.
\end{align*}
Thus, there exists $\vec F_1 \in \ccpsi(\om)$ (resp. $\vec F_2$ so that $\delta_\om \vec F_2\in\nnp(\om)$) such that $\nabla S_k \to \vec F_1$ in $\ccpsi(\om)$ (resp. $\delta_\om\nabla S_k \to \delta_\om\vec F_2$ in $\nnp(\om)$), and,  by Lemma \ref{lem: sup gradient bound}, we have that for any fixed $x\in \om$,
$$
\sup_{y \in B_x} |\nabla S_k-F_i |\delta_\om(y) \to 0, \quad \textup{for}\,\, i=1,2,
$$
which readily implies that $\vec F_1=\vec F_2=:\vec F$ in $\om$. Hence,
$\nabla S_k$ converges to $\sum_{k=0}^\infty \nabla u_k$ uniformly in $B_x$ for every $x\in\om$, and 
by Lemma \ref{lem: gradient uniform convergence}, we deduce that $u \in C^1(\om)$ and 
$$
\sum_{k=0}^\infty \nabla u_k(x) = \nabla u(x)\quad  \textup{for all} \,\, x\in \om.
$$
 In fact, 
 \begin{align*}
\|\mathcal{N}(u)\|_{L^p(\sigma)} + \|\mathcal{N}(\delta_\om \nabla u)\|_{L^p(\sigma)}+ \|\mathcal{C}_s(\nabla u)\|_{L^p(\sigma)} \lesssim \| f \|_{L^p(\sigma)}.
\end{align*}

To show that $u$ is an  extension of $f$, notice first that, in light of \eqref{eq:f_k+1<ef_k},
\begin{align*}
0 &=\underset{k\to\infty}{\lim}\|f_{k+1}\|_{L^p(\sigma)}=\underset{k\to\infty}{\lim}\| f- \tr \Big(\sum_{j=0}^k u_j \Big) \|_{L^p(\sigma)}.
\end{align*}
Since,  by construction, $\sum_{j=0}^k u_j-u \in \nnp(\om) $,  in light of Hypothesis \ref{Word:Hypothesis T} (i), we get that
 $$
\|\tr \Big(\sum_{j=0}^k u_j\Big) - \tr(u)  \|_{L^p(\sigma)}=\|\tr \Big(\sum_{j=0}^k u_j - u \Big) \|_{L^p(\sigma)}\leq  
\| \sum_{j=0}^k u_j  - u\|_{\nnp(\om)} \overset{k \to \infty}{\longrightarrow} 0,
 $$
which entails that $ \tr(u)(\xi)= f(\xi)$ for $\sigma$-almost every $\xi\in \pom$.

\end{proof}

\vv

\begin{proposition}\label{lem: non-tangential convergence}
Let $\om \in \AR(s)$ for $s \in(0,n]$, which for $s=n$ satisfies the pointwise John condition. Then, for any $p \in (1,\infty]$, there exists a bounded linear  trace operator $\tr{_\om}: \nnp(\om) \cap \cc^{1, p}_{s, \infty}(\om)  \to L^p(\sigma)$  satisfying the Hypothesis  \textup{\ref{Word:Hypothesis T}}.   Moreover,  if $\om$ satisifies the local John condition for $s=n$, there exists a bounded linear  trace operator $\tr{_\om}:  \cc^{1, \infty}_{s, \infty}(\om) \to \BMO(\sigma)$.
\end{proposition}

\begin{proof}
For any $ x\in\om $ and fixed $c \in (0,1/2]$, we define
\begin{equation}\label{eq: average of u}
E(x) := \fint_{B(x, c\delta_\om(x))} u(z)\, dz. 
\end{equation}
Fix $\xi\in\pom$ such that $ \xi \in \JC(\theta)$ (see Definition \ref{def: pwJohn}).  Then there exist $r_{\xi}>0$ and $x_{\xi} \in B(\xi, 2r_\xi) \cap\om $ such that $\delta_\om(x_\xi) \geq \theta r_\xi$,  and also there exists a good curve
$ \gamma: [0, 1] \to \R $ in $ B(\xi,  2r_{\xi})\cap\om$ connecting the points $ \xi $ and $ x_\xi $ such that $| \dot{\gamma}(t) | = 1$ 
$\forall t\in[0, 1]$.  For any fixed points $ x_1, x_2 \in \gamma $,  there exist $ t_{1}, t_{2} \in [0, 1] $ such that 
$ x_1 =\gamma( t_{1} ) $ and $ x_2 = \gamma( t_{2} ) $. 
Applying a change of variables and applying the mean value theorem,  we estimate
\begin{align*}
| E(x_1) &- E(x_2) | = \Big| \int_{B(0, 1)} \big( u(x_1+wc\delta_\om(x_1)) - u(x_2+wc\delta_\om(x_2)) \big)\, dw \Big| \\
&= \Big| \int_{B(0, 1)} \big( u(\gamma(t_{1}) + wc\delta_\om(\gamma(t_{1}))) - 
u(\gamma(t_{2}) + wc\delta_\om(\gamma(t_{2}))) \big) \, dw \Big| \\
&= \Big| \int_{B(0, 1)} \int_{t_{1}}^{t_{2}} \nabla u( \gamma(t)+wc\delta_\om(\gamma(t))) \cdot
\nabla \delta_\om(\gamma(t))\,\dot{\gamma}(t)  \, dt\, dw \Big| \\
&\leq \int_{t_{1}}^{t_{2}} \int_{B(0, 1)} \big| \nabla u(\gamma(t) + wc\delta_\om(\gamma(t)) \big|\, dw\, dt
\end{align*}
where we used that $|\dot{\gamma}(t)|=1$ and $| \nabla \delta_\om(\gamma(t)) | \leq 1$ since the function $\dist( \cdot \, , \pom) $ is 
$1$-Lipschitz. 
Note that there exists $M_j\in \NN$ such that $ 2^{-M_j} \leq t_{j} \leq 2^{-M_j + 1} $, for $ j= 1, 2$. 
By the Fundamental Theorem of Calculus, since $u \in C^1(\om)$ and $\gamma(t) + wc\delta_\om(\gamma(t))$ is $(1+c)$-Lipschitz in $t$ for any $w \in B(0,1)$, we have that there exists $s_k \in [2^{-k}, 2^{1-k}]$ such that
\begin{align*}
\int_{t_{1}}^{t_{2}} \int_{B(0, 1)} &\big| \nabla u(\gamma(t) + wc\delta_\om(\gamma(t)) \big| dw dt\\
& = \sum_{k=M_1}^{M_2} \int_{2^{-k}}^{2^{-k+1}} \int_{B(0, 1)} \big| \nabla u(\gamma(t) + wc\delta_\om(\gamma(t)) \big| dw dt \\
&= \sum_{k=M_1}^{M_2}2^{-k}  \int_{B(0, 1)} \big| \nabla u(\gamma(s_k) + wc\delta_\om(\gamma(s_k)) \big| dw   \\
&=\sum_{k=M_1}^{M_2}2^{-k}  \int_{ B(\gamma(s_k), c \delta_\om(\gamma(s_k)))}  |\nabla u(y)| \frac{dy}{ (c \delta_\om(\gamma(s_k)))^{n+1}}\\
&\lec \sum_{k=M_1}^{M_2}  \int_{ B(\gamma(s_k), c \delta_\om(\gamma(s_k)))}  |\nabla u(y)|  \frac{dy}{\delta_\om(y)^n},
\end{align*}
where is the last inequality we used that $\delta_\om(\gamma(s_k)) \approx s_k \approx 2^{-k}$ and that $\delta_\om(y) \lesssim \delta_\om(\gamma(s_k))$ for any $y \in B(\gamma(s_k), c \delta_\om(\gamma(s_k)) )$. Therefore,  there exists a cone $\gamma_\alpha(\xi)$ with apperture depending on $c$ and $\theta$ such that $B(\gamma(s_k), c \delta_\om(\gamma(s_k))) \subset \gamma_\alpha(\xi)$, and, by the bounded overlaps of the balls $B(\gamma(s_k), c \delta_\om(\gamma(s_k))) $, we infer that 
$$
| E(x_1) - E(x_2) | \lec \int_{\gamma_\alpha(\xi) \cap B(\xi, C \delta_\om(x_2)) }  |\nabla u(y)|  \frac{dy}{\delta_\om(y)^n}.
$$

By \eqref{eq:A<Cglobal},  we have that $ \mathcal{A}^{(\alpha)}_s( \nabla u ) \in L^p(\sigma)$  for $p\in(1, \infty)$ and $ \mathcal{A}^{(\alpha)}_s( \nabla u ) \in L^q_{\loc}(\sigma)$ for any $q \in (1,\infty)$ when $p=\infty$. Thus,   $ \mathcal{A}^{(\alpha)}_s(\nabla u)(\xi) < \infty $ for $\sigma$-almost every $\xi\in\pom$ 
 and using the fact that the above estimate holds for any points $ x_1, x_2 \in \gamma_\xi$,  we can assume that  $x_1, x_2 \in B(\xi, \ve)$ for some $\ve>0$ small compared to $r_\xi$.  Therefore,
$$
|E(x_1) - E(x_2)| \lec 
\int_{\gamma_\alpha(\xi)\cap B(\xi, C\ve)} | \nabla u(y)| \, \delta_\om(y)^{-n} \, dy  \lec  \, \mathcal{A}^{(\alpha)}(\nabla u)(\xi) < \infty .
$$ 
By the dominated convergence theorem, we get that 
$ | E(x_1) - E(x_2) | \to 0 $ as $\ve \to 0$, i.e.,  $E(x)$ is Cauchy on $\gamma_\xi$ and thus convergent. This shows that the quasi-non-tangential limit of $E(x)$ at $\xi \in \pom$ exists for $\sigma$-a.e. $\xi \in \pom$ and we can  define the desired trace operator by 
\begin{equation}\label{eq: trace}
\tr{_\om}(u)(\xi):=\qntlim_{x \rightarrow \xi} E(x) \quad \,\,\text{for}\,\,\sigma\textup{-a.e.}\,\xi \in \pom.
\end{equation}
In the case  $s<n$,  we  just define $\tr{_\om}(u)(\xi)=\ntlim_{x \rightarrow \xi} E(x)$ since $\om$ has only  one connected component and any $\xi \in \pom$ can be connected to a corkscrew point by a good curve.  It is clear that $\tr: \cc^{1, p}_{s, \infty}(\om) \to L^p(\sigma)$ is a linear operator, while the fact that $\tr: \nnp(\om) \cap \cc^{1, p}_{s, \infty}(\om) \to L^p(\sigma)$  is bounded if $p \in (1, \infty]$ can be proved pretty easily.  Indeed,  let $\xi\in\pom$ such that $ \xi \in \JC(\theta)$.  For fixed $\ve>0$,  there exists $\delta>0$ such that, if  $x\in B(\xi, \delta) \cap \gamma_\xi$,  it holds that 
$$
|\tr{_\om}(u)(\xi)| \leq |\tr{_\om}(u)(\xi) - E(x)| + m_{\sigma, B(x, c\delta_\om(x))} (|u|) 
< \ve + \calN(u)(\xi).
$$
Letting $ \ve\to 0$ we infer that $ |\tr{_\om}(u)(\xi)| \leq  \calN(u)(\xi) $ for $\sigma$-a.e. $\xi \in \pom$, which readily concludes Hypothesis  \ref{Word:Hypothesis T} (i), while (ii)  readily follows from Lemma \ref{lemma:ntconvergenceofdyadic}.

Assume now that $\om$ satisfies the local John condition when $s=n$ (in the case $s<n$ this is automatic).  Fix $\xi\in\pom$ and $r>0$,  and, by the local John condition,  there exists a corkscrew point  $x_r \in B(\xi,r)$ such that any $\zeta \in B(\xi,r) \cap \pom$ can be connected to $x_r$ by  a good curve. The existence of the trace operator  follows by the same argument as above and we define it the same way.  It remains to show that $\tr: \cc^{1, \infty}_{s, \infty}(\om) \to \BMO(\sigma)$ is bounded.  For $u \in \cc^{1, \infty}_{s, \infty}(\om)$, if  $B_r := B(x_r,  c \delta_\om(x_r))$ is a corkscrew ball centered at $x_r$ with radius $c \delta_\om(x_r) \approx r$,   by the same proof as above,  we can show that 
$$
\big| \tr{_\om}(u)(\zeta) - \fint_{ B_r }u(y)\, dy \big| \lec \mathcal{A}_s(\nabla u {\bf 1}_{B(\xi,C' r)})(\zeta),
\quad \forall \zeta\in B(\xi, r)\cap\pom.
$$
Thus, taking averages over the ball $B(\xi, r)$  and applying \eqref{eq:A<Clocal} in $L^1(B(\xi, C' r))$,  we conclude that 
$$
\fint_{B(\xi, r)} \big| \tr{_\om}(u) -  \fint_{ B_r }u(y)\, dy \big|  d\sigma \lec 
\fint_{ B(\xi, C'' r) } \calC_s( \nabla u {\bf 1}_{B(\xi,C'' r)})\, d\sigma \leq \| \calC_s( \nabla u ) \|_{L^\infty(\sigma)}.
$$ 
This readily implies that $ \| \tr{_\om}(u) \|_{\BMO(\sigma)} \lec  \| \calC_s(\nabla u) \|_{L^\infty(\sigma)} $,  which,  combined with   $\|f\|_{\BMO(\sigma)} \leq 2\|f\|_{L^\infty(\sigma)}$,  \eqref{eq:v-w},    and  Hypothesis  \ref{Word:Hypothesis T} (ii) (we have already proved it above),  proves Hypothesis  \ref{Word:Hypothesis T'}.
\end{proof}

\vv
We state  \cite[Theorem 2, p. 171]{Stein} in the following lemma. 
\begin{lemma}\label{lem: Stein lemma}
Let $E\subset \R^{n+1}$ be a closed set and $\delta_E$ be the distance function with respect to $E$. Then there exist positive 
constants $m_1$ and $m_2$ and a function $\beta_E$ defined in $E^c$ such that
\begin{itemize}
\item[(i)] $m_1\delta_E(x)\leq \beta_E(x) \leq m_2\delta_E(x)$, for every $x\in E^c$, and 
\item[(ii)] $\beta_E$ is smooth in $E^c$ and 
$$
\Big| \frac{\partial^\alpha}{\partial x^\alpha}\beta_E(x) \Big| \leq C_\alpha \beta_E(x)^{1-|\alpha|}.
$$
\end{itemize}
In addition, the constants $m_1$, $m_2$ and $C_\alpha$ are independent of $E$. 
\end{lemma}

Following  \cite[Section 3]{HT20}, we  define a  kernel $\Lambda(\cdot,\cdot) : \om \times \om \to [0,\infty]$,  which will be  necessary in the proof of Theorem \ref{thm: extension in bmo}. To this end,  let $\beta = \beta_{\om^c}$ be the  function constructed  in Lemma \ref{lem: Stein lemma} and let $\zeta \geq 0$ be a smooth non-negative function supported on $ B(0, \frac{c}{4m_2}) $,  satisfying $\zeta\leq 1$ and  $ \int \zeta = 1 $.  For every $ \lambda >0 $, we set 
$$
\zeta_\lambda (x) := \lambda^{-(n+1)}  \zeta \big( x/\lambda\big), 
$$
and define the mollifier
$$ 
\Lambda( x, y):= \zeta_{\beta(x)}(x - y) = \frac{1}{ \beta(x)^{n+1} } \zeta \Big( \frac{2(x-y)}{\beta(x)} \Big).
$$
Observe that,  by construction,  for every $ x\in \om$,
\begin{equation}\label{eq:lambda-prop}
 \supp ( \Lambda (x, \cdot) ) \subset \wt B_x:=B(x, c\,\delta_\om(x)/8) \quad \textup{and}\quad  \int_\om \Lambda(x,y)\, dy = 1.
\end{equation}
Moreover, it is easy to prove that
\begin{equation}\label{eq:derivetive-lambda}
\sup_{y \in  \wt B_x} \Lambda(x,y) \lesssim \delta_\om(x)^{-n-1} \quad \textup{and}\quad \sup_{y \in  \wt B_x} |\nabla_x \Lambda(x,y) | \lesssim \delta_\om(x)^{-n-2}.
\end{equation}
For any $F:\om \to \R$, we  define the {\it smooth modification} of $F$ by
\begin{equation}\label{def:tilde-mollify}
\wt{F}(x) := \int_\om \Lambda(x, y)\, F(y) \,dy.
\end{equation}

\vv

The next lemma was essentially proved in  \cite[Section 3]{HT20} but we provide its proof  for the reader's convenience.

\begin{lemma}\label{lem:modifedFproperties}
Let $\om \subset \rrn$ be an open set satisfying the corkscrew condition. If $F\in C^1(\om ; \rrn) $ and $\wt F$ is the smooth modification of $F$ as defined in \eqref{def:tilde-mollify}, then we have that
\begin{itemize}
\item[(a)] For any $x \in \om$, 
\begin{equation*}
|\wt F(x)| \lec \sup_{ \wt B_x}|F(y)|.
\end{equation*}
\item[(b)]  For any $x \in \om$, 
\begin{equation*}
|\wt F(x_1)-\wt F(x_2)| \lec |x_1- x_2| \,\delta_\om(x)^{-1} \,m_{\sharp, c}(F)(x), \quad \textup{for all} \,\, x_1, x_2 \in  B_x.
\end{equation*}
\item[(c)]  For any $x \in \om$, 
\begin{equation*}
m_{\sharp,c}(\wt F)(x) \lec m_{\sharp,c}(F)(x).
\end{equation*}
\item[(d)] For any $\xi \in \pom$, 
\begin{equation*}
\sup_{x \in \gamma_\alpha(\xi)}\delta_\om(x)| \nabla \wt F(x) | \lec \mathscr{C}_s(\nabla \wt F)(\xi).
\end{equation*}
\item[(e)] For any $\xi \in \pom$,
\begin{equation*}
\calC_{s,c}(\nabla \wt F)(\xi) \lec \mathscr{C}_{s}(\nabla  F)(\xi).
\end{equation*}
\item[(f)] If $\qntlim_{x \to \xi} F(x) = f(\xi)\,\textup{(resp. }\ntlim_{x \to \xi} F(x) = f(\xi) \textup{)}$ for $\sigma$-a.e. $\xi \in \pom$, then $\qntlim_{x \to \xi} \wt F(x) = f(\xi)\,\textup{(resp. }\ntlim_{x \to \xi} \wt F(x) = f(\xi) \textup{)}$ for $\sigma$-a.e. $\xi \in \pom$.
\end{itemize}
\end{lemma}

\begin{proof}
(a) follows by definition. For (b),  we let  $x_1, x_2 \in B_x$. Then,  by triangle inequality, 
\begin{equation}\label{eq:twoballsinone}
B(x_1,  c\, \delta_\om(x_1)/8) \cup B(x_2,  c\, \delta_\om(x_2)/8) \subset \frac{5}{4}    B_x.
\end{equation}
Combining \eqref{eq:lambda-prop}, \eqref{eq:derivetive-lambda}, and \eqref{eq:twoballsinone}, we get that
\begin{align*}
|\wt F(x_1)-\wt F(x_2)|&= \left|\int (\Lambda(x_1, y)-\Lambda(x_2, y))\, \left( F(y) - m_{\frac{5}{4}B_x} F\right) \,dy \right|\\
  &\leq |x_1-x_2| \sup_{z  \in \frac{5}{4}B_x} \sup_{w \in  B_x}  |\nabla_w \Lambda (w,z)| \, m_{\frac{5}{4} B_x} \left( |F - m_{\frac{5}{4} B_x} F |\right)\, \big|\tfrac{5}{4} B_x \big| \\
&\lec |x_1- x_2| \,\delta_\om(x)^{-1} \,m_{\sharp, c}(F)(x),
\end{align*} 
which proves (b) and thus (c).  We turn our attention to (d)  and fix $\xi\in\pom$ and $r>0$.  	For any $z\in B(\xi, r)\cap\om$ and  $x\in B_z$,  using \eqref{eq:lambda-prop}, \eqref{eq:derivetive-lambda}, and   Poincar\'{e} inequality,  we can write 
\begin{align*}
|\nabla \wt F(x)|&=\Big| \int \nabla_x\Lambda(x, y) F(y)\, dy\Big|
=\Big|\int \nabla_x\Lambda(x, y)\big(F(y)-m_{\wt B_x}F \big)\, dy\Big| \\
&\lec \delta_\om(x)^{-n-2}\int_{\wt B_x}\big|F(y)-m_{\wt B_x}F\big|\, dy\lec \fint_{\wt B_x} | \nabla F|\, dy,
\end{align*} 
which immediately implies (d).  To prove (e), we first define  
\begin{align*}
A_k(\xi,r)&:= \{ x \in B(\xi,r) \cap \om: 2^{-k-1} r \leq \delta_\om(x) < 2^{-k} r\}\\
A^*_k(\xi,r)&:= \{ x \in B(\xi,r) \cap \om: 2^{-k-2} r \leq \delta_\om(x) < 2^{-k+1} r\}
\end{align*}
and write 
\begin{align*}
\int_{B(\xi,r) \cap \om}& \sup_{y \in B_x}  |\nabla \wt F(y)| \,dy \leq  \sum_{k=0}^\infty \int_{A_k(\xi,r)} \sup_{y \in B_x}  |\nabla \wt F(y)| \,dy.
\end{align*}
As $\cup_{y \in B_x} B_y \subset 2B_x$, by Fubini's theorem,
$$
\int_{A_k(\xi,r)} \sup_{y \in B_x}  |\nabla \wt F(y)| \,dy \lec \int_{A_k(\xi,r)}  \fint_{2 B_x} | \nabla F(y)|\, dy \lec \int_{A^*_k(\xi,r)} | \nabla F(y)|\, dy,
$$
Summing over $k$ and using that $A^*_k(\xi,r)$ have bounded overlap,  we get $(e)$. Finally, (f) follows from  \cite[Lemma 3.14]{HT20}.
\end{proof}

\vv

\begin{proof}[proof of Theorem \ref{th: introtheorem2}]
It is an immediate consequence of  Theorem  \ref{thm: extension in Lp},  Proposition \ref{lem: non-tangential convergence}, and Lemma \ref{lem:modifedFproperties}.
\end{proof}

\vv

Let us now turn our attention to the proof of Theorem \ref{th: introtheorem3}.  When $s=n$ and $\om$ satisfies the pointwise John condition but  not the local John condition, we   will need the following generalization of Garnett's Lemma, which was  proved in \cite[Lemma 10.1]{HT20}.

\begin{lemma}\label{lem: Garnett's Lemma}
Let $\om\ \in \AR(n)$, $ Q_0 \in \DD_\sigma $, and let 
$ f \in \BMO(\sigma) $ which vanishes on $ \pom \setminus Q_0 $ (if it is non-empty). Then, 
there exists  a collection of cubes $\wt{\mathcal{S}}(Q_0) = \{ Q_j \}_j \subset \DD(Q_0) $ and coefficients $\alpha_j$ such that 
\begin{enumerate}
\item $ \sup_j |\alpha_j | \lec \| f \|_{\BMO(\sigma)} $,
\item $ f = g + \sum_j \alpha_j {\bf 1}_{Q_j} $, where $ g\in L^\infty(\sigma) $ with $ \| g \|_{L^\infty(\sigma)} \lec \| f \|_{\BMO(\sigma)} $ and 
\item $\wt{\mathcal{S}}(Q_0) $ satisfies a Carleson packing condition. 
\end{enumerate}
\end{lemma}

\vv

\begin{proof}[proof of Theorem \ref{th: introtheorem3}]
Recall that if $s<n$ then $\om \in \AR(s)$ is uniform and thus it satisfies the local John condition.  Therefore,  by Theorem \ref{thm: extension in bmo}, Proposition \ref{lem: non-tangential convergence}, and Lemma \ref{lem:modifedFproperties},  we can construct the desired extension of Theorem  \ref{th: introtheorem3} when either $s<n$, or  $s=n$ and $\om$ satisfies, in addition,  the local John condition. We are left with the case $s=n$ so that $\om$ satisfies the pointwise John condition but not the local John condition. By Lemma \ref{lem: Garnett's Lemma}, if $f \in \BMO(\sigma)$ with compact support in $Q_0 \in \DD_\sigma$,  there exists $g \in L^\infty(\sigma)$ and $b=\sum_j \alpha_j {\bf 1}_{Q_j}  \in \BMO(\sigma)$ such that $f=g+b$.  We construct an extension $G$ of $g$ by Theorem \ref{th: introtheorem2} and so it remains to prove the existence of the extension of $b$.  By  \cite[Proposition 1.3]{HT20},  there exists $ B_0:\om \to \R $ such that  $\sup_{\xi \in \pom}\mathscr{C}_n(\nabla  B_0)(\xi) + \sup_{x \in \om} \delta_\om(x) |\nabla  B_0(x)| \lec \|f\|_{\BMO}$ and $ B_0 \to b$ non-tangentially for $\sigma$-a.e. $\xi \in \pom$.  By Lemma \ref{lem:modifedFproperties}, if we set $B=\wt B_0$ (as defined in \eqref{def:tilde-mollify}), we get the desired  extension of $b$.  The  extension of $f$ is given by $G+B$.
\end{proof}

\vvv

\section{Varopoulos-type extensions of compactly supported Lipschitz functions}

We begin by constructing an extension of $L^p$-boundary functions in the next theorem. 

\begin{theorem}\label{thm: Lipschitz extension in Lp}
Let $ \om \in \AR(s)$ for $s \in (0, n]$.  If $ f\in \Lip_c(\pom)$,  there exists a  function $ F : \oom \to \R $ such that for every $ p \in (1, \infty]$,
\begin{itemize}
\item[(i)] $ F \in  C^\infty(\om)\cap \Lip(\oom) $, 
\item[(ii)] $ F|_{\pom} = f $ continuously,
\item[(iii)] $ \| \calN(F) \|_{L^p(\sigma)} +\|\calN(\delta_\om \nabla F)\|_{L^p(\sigma)} \lec \| f \|_{L^p(\sigma)}$,  
\item[(iv)] $  \| \calC_s(\nabla F) \|_{L^p(\sigma)} \lec \|f\|_{L^p(\sigma)} $,
\item[(v)] $\| [\calC_s(\delta_\om |\nabla F|^2)]^{1/2}\|_{L^q(\sigma)} \lec \|f\|_{L^q(\sigma)}$, \quad for \,  $q\in[2, \infty)$.
\end{itemize}
When $p=\infty$ the norms on left hand-side of \textup{(iii)} and \textup{(iv)} are the $\sup$-norms instead of $L^\infty$.
\end{theorem}

\begin{proof}
Let $\{ \vp_P\}_{P\in\WW(\om)}$ be a partition of  unity of $\om$ so that each $\vp_P$ is supported in $1.1P$ and $ \| \nabla \vp_P \|_{\infty} \lec 1/{\ell(P)}$. 
For each $\delta \in (0, \diam(\om))$,  set 
$$
\WW_{\delta}(\om) = \{ P\in\WW(\om) : \ell(P)\geq \delta \}
$$
and 
$$
\vp_\delta = \sum_{P\in\WW_{\delta}(\om)} \vp_P. 
$$
 From the properties of the Whitney cubes, there exists  $C>0$ 
(depending on the parameters of the construction of the Whitney cubes) such that
$$
\vp_\delta(x) = 0,  \quad \textup{if} \,\, \dist(x, \pom) \leq \delta/C
$$
and 
$$
\vp_\delta(x) = 1,  \quad \textup{if} \,\, \dist(x, \pom) \geq C\delta .
$$
Consequently, 
for a suitable constant $C'$ depending on $C$, we infer that 
\begin{equation}\label{eq: supportgradphidelta}
\supp \nabla\vp_\delta \subset \{ x\in\om: \delta/C \leq \dist(x, \pom)  \leq C\delta \} =: S_\delta
\subset \bigcup_{P\in \mathcal{I}_\delta} P,
\end{equation}
where 
$$
\mathcal{I}_{\delta}:= \{ P\in\WW(\om) : \frac{1}{2^{N_0+1}2^{N_1}}\leq \ell(P) \leq \frac{2^{N_1}}{2^{N_0}}\}
$$ 
with $N_0\in\NN$ such that $\frac{1}{2^{N_0+1}} \leq \delta \leq \frac{1}{2^{N_0}}$ and $N_1\in\NN$ satisfies 
$2^{N_1} \leq C \leq 2^{N_1+1}$.

We define 
\begin{equation}\label{modifu}
F(x):= \upsilon_f(x) (1-\vp_\delta(x)) + u(x) \vp_\delta(x),
\end{equation}
where $u$ is the approximation function of $\upsilon_f$ as constructed in Theorem \ref{thm:e-approxLp}.  It holds that
\begin{equation}\label{eq: CgradF}
\calC_s(\nabla F) \leq \calC_s(\nabla u) + \calC_s(|\nabla \vp_\delta| \, (u-\upsilon_f))+ 
\calC_s(|\nabla \upsilon_f| (1 - \vp_\delta))
\end{equation}
and
\begin{equation}\label{eq: CdeltagradF2}
\calC_s(\delta_\om|\nabla F|^2)\lec \calC_s(\delta_\om|\nabla \upsilon_f|^2(1-\vp_\delta)^2) + \calC_s(\delta_\om|u-\upsilon_f|^2|\nabla\vp|^2) + \calC_s(\delta_\om |\nabla u|^2).
\end{equation}
For fixed $\xi\in\pom$ and $r>0$, we have 
\begin{align*}
&\int_{ B(\xi, r) \cap \om} | \nabla \vp_\delta |\, |u-\upsilon_f|\, \frac{ dx }{ \delta_\om(x)^{n-s}}   \lec 
\sum_{ \substack{ P\in \mathcal{I}_\delta \\ P\cap B(\xi, r) \neq \emptyset }} \int_P  
|u - \upsilon_f|\, \frac{ dx }{ \delta_\om(x)^{n+1-s}} \\
&\lec \sum_{ \substack{ P\in \mathcal{I}_\delta \\ P\cap B(\xi, r) \neq \emptyset }} \ell(P)^s
\inf_{\zeta\in b(P)}  \calN_\alpha( u - \upsilon_f )(\zeta)  
\lec \sum_{ \substack{ P\in \mathcal{I}_\delta \\ P\cap B(\xi, r) \neq \emptyset }} 
\int_{ b(P) }  \calN( u - \upsilon_f )  \, d\sigma  \\
&\lec   \sum_{ k= -(N_1+1)}^{N_1} \, \sum_{ \substack{ \ell(P) = 2^k/{2^{N_0}} \\ P\subset B(\xi, Mr) }} \,
\int_{ b(P) }  \calN_\alpha( u - \upsilon_f )  \, d\sigma  \lec_C \int_{ B(\xi, M r) }  \calN_\alpha( u - \upsilon_f )   \, d\sigma,
\end{align*}
for suitably chosen constants $\alpha>1$ and $M>1$ large enough. 
Thus, when $p \in (1, \infty)$,  we will get that
\begin{equation}\label{eq:LpLipextension1}
 \calC_s( |\nabla \vp_\delta| \, (u-\upsilon_f))(\xi)  \lec \mathcal{M}(\calN_\alpha( u - \upsilon_f) )(\xi).
\end{equation}
By similar arguments we can also infer that 
\begin{equation}\label{eq: Cdeltau-uf2}
\calC_s(\delta_\om(u-\upsilon_f)^2 |\nabla \vp_\delta|^2)(\xi) \lec \mathcal{M}(\calN_\alpha(u-\upsilon_f)^2)(\xi).
\end{equation}
When $p=\infty$,  by \eqref{eq: pwMFBMO},  we get $ \sup_{x\in\om}\sup_{y \in B_x} |u(y)-\upsilon_f(y)|  \leq 2 \ve \|f\|_{L^\infty(\sigma)}$, which,  arguing as above,  implies 
\begin{align}\label{eq:Linftycarleson1}
\int_{ B(\xi, r) \cap \om} | \nabla \vp_\delta| \,  &|u-\upsilon_f | \frac{ dx }{ \delta_\om(x)^{n-s}}  \lec  \|f\|_{L^\infty(\sigma)} \sum_{ \substack{ P\in \mathcal{I}_\delta \\ P\cap B(\xi, r) \neq \emptyset }} \ell(P)^s\\
 &\lec  \|f\|_{L^\infty(\sigma)}   \sum_{ k= -(N_1+1)}^{N_1} \, \sum_{ \substack{ \ell(P) = 2^k/{2^{N_0}} \\ b(P)\subset B(\xi, Mr) }} \sigma(b(P)) \lec r^s  \|f\|_{L^\infty(\sigma)}\notag
\end{align}
 and thus
\begin{equation}\label{eq:LpLipextension1bis}
\sup_{\xi \in \pom} C_s(|\nabla \vp_\delta| \, (u-\upsilon_f))(\xi) \lec  \|f\|_{L^\infty(\sigma)}.
\end{equation}

For the last term on the right hand side of \eqref{eq: CgradF},  when $p\in(1, \infty)$,  since $f \in \Lip_c(\pom)$, we have that $ f\in \dot{M}^{1, p} (\sigma) $. 
So, for fixed $\xi\in\pom$ and $r>0$,  if $\nabla_H f$ is the least  upper gradient of $f$, in view of \eqref{gradufW}, we estimate
\begin{align*}
&\int_{ B(\xi, r) \cap \om } |\nabla \upsilon_f(x) | | 1 - \vp_{\delta}(x) | \frac{ dx }{ \delta_\om(x)^{n-s}}  \lesssim 
\sum_{\substack{ P\in\WW(\om) \\ P\cap B(\xi, r) \neq \emptyset \\ \ell(P)\lec \delta }} 
\int_P | \nabla \upsilon_f (x) | \frac{\ell(P)^s}{\ell(P)^n} \, dx
\end{align*}
\begin{align*}
&\lec \sum_{\substack{ P\in\WW(\om) \\ P\cap B(\xi, r) \neq \emptyset \\ \ell(P)\leq C \delta }}  m_{\sigma, b(P)}(\nabla_H f)\, \ell(P)^{s+1} \lec  \sum_{\substack{ P\in\WW(\om) \\ P\cap B(\xi, r) \neq \emptyset \\ \ell(P)\leq C \delta }} {\ell(P)} \sigma(b(P)) 
\inf_{\zeta \in b(P)} \mathcal{M}( \nabla_H f )(\zeta) \\
&\leq \sum_{k \geq N_0-N_1} 2^{-k} \sum_{\substack{ Q\in \DD_\sigma \\ Q \subset B(\xi, Mr) \\ \ell(Q)=2^{-k}}} \, 
\int_{Q} \mathcal{M} ( \nabla_H f )\, d\sigma 
\lec \delta \, m_{\sigma, B(\xi, Mr)} (\mathcal{M}(\nabla_H f))\,r^s,
\end{align*}
which shows that 
\begin{equation}\label{eq:LpLipextension2}
\calC_s \big( |\nabla \upsilon_f| (1 - \vp_\delta) \big)(\xi) \lec  \delta \, \mathcal{M}(\mathcal{M}( \nabla_H f ))(\xi).
\end{equation}
In view of \eqref{graduf} and  \eqref{gradufW}, 
$$
\sup_{x \in P} |\nabla v_f(x)|^2 \lec \ell(P)^{-1} m_{\sigma, B_P}(|f|)\,m_{\sigma, B_P}(\nabla_H f).
$$
Arguing as above and using the latter estimate,  we get 
\begin{equation}\label{eq: Cdeltagraduf2}
\calC_s\big(\delta_\om |\nabla\upsilon_f|^2(1-\vp_\delta)^2\big)(\xi) \lec \delta \mathcal{M}\big(\mathcal{M}(\nabla_Hf)\,\mathcal{M}f \big)(\xi). 
\end{equation}
For  $p=\infty$, we use Lemma \ref{lem: lemma1} to  get 
\begin{equation}\label{eq:LpLipextension2bis}
\calC_s(|\nabla \upsilon_f| (1 - \vp_\delta))(\xi)  \lec \Lip( f ) \delta.
\end{equation}
Indeed, for $\xi\in\pom$ and $r>0$,
\begin{align*}
&\int_{B(\xi, r)\cap \om} |\nabla \upsilon_f| |1-\vp_\delta| \frac{dx}{ \delta_\om(x)^{n-s}} 
\lec \Lip(f) \sum_{\substack{ P\in \WW(\om) \\ P\cap B(\xi, r)\neq \emptyset \\ \ell(P)\leq C \delta}}
\int_P \frac{\ell(P)^s}{\ell(P)^n} \, dx \\
&\lec \Lip(f)  \sum_{\substack{ P\in \WW(\om) \\ P \cap B(\xi, r)\neq \emptyset \\ \ell(P)\leq C \delta}} \ell(P)^{s+1} \leq  \Lip(f) \sum_{k \geq N_0-N_1} 2^{-k} \sum_{\substack{ P\in \WW(\om) \\ P\cap B(\xi, r)\neq \emptyset \\ \ell(P)=2^{-k}}} \sigma(b(P))
\end{align*}
\begin{align*}
\lec \Lip(f) \sum_{k \geq N_0-N_1} 2^{-k} \sum_{\substack{ Q\in \DD_\sigma \\ Q \subset B(\xi, Mr) \\ \ell(Q)=2^{-k}}} \sigma(Q) \lec \Lip(f) \, \delta\,r^s,
\end{align*}
for $M>1$ large enough constant depending on Ahlfors regularity and Whitney constants.

Combining \eqref{eq:pwCarlu}, \eqref{eq: CgradF},  \eqref{eq:LpLipextension1},  \eqref{eq:LpLipextension1bis}, \eqref{eq:LpLipextension2},  and \eqref{eq:LpLipextension2bis},  and choosing 
\begin{equation*}
\delta=
\begin{dcases}
 \| f \|_{L^\infty(\sigma)}/\Lip(f)  &,\textup{if}\,\,  p=\infty\\
 \| f \|_{L^p(\sigma)}/\| f \|_{\dot{M}^{1, p}(\sigma)} &,\textup{if}\,\,  p\in (1,\infty),
\end{dcases}
\end{equation*}
it follows that 
$$
\| \calC_s ( \nabla F ) \|_{L^p(\sigma)} \lec \| f \|_{L^p(\sigma)}, \quad \textup{for}\,\, p \in (1, \infty].
$$
Moreover, combining \eqref{eq:pwCarlu-bis}, \eqref{eq: CdeltagradF2}, \eqref{eq: Cdeltau-uf2}, \eqref{eq: Cdeltagraduf2} and choosing $\delta=\|f\|_{L^q(\sigma)}/\|f\|_{\dot M^{1, q}(\sigma)}$ it follows 
$$
\|\calC_s(\delta_\om|\nabla F|^2)^{1/2}\|_{L^q(\sigma)} \lec \|f\|_{L^q(\sigma)} \quad \textup{for} \,\,\, q\in(2, \infty).
$$
The last estimate can be proved for $q=2$ by direct estimates using \eqref{eq:comparable L2C}. We leave the proof as an exercise.

For the non-tangential estimate note that, since $ \upsilon_f = \vp_\delta \upsilon_f + (1-\vp_\delta)\upsilon_f $, 
we can write
$$
\upsilon_f - F = \vp_\delta ( \upsilon_f - u ).
$$
So, for every $\xi\in\pom$ we have 
$$
\calN(\upsilon_f - F)(\xi) = \sup_{x\in\gamma(\xi)}|\upsilon_f(x)-F(x)| \leq \sup_{x\in\gamma(\xi)}|\upsilon_f(x)-u(x)| 
= \calN(\upsilon_f-u)(\xi)
$$
and thus, by \eqref{eq:LpntMfu}, we get
$$
\| \calN(\upsilon_f - F) \|_{L^p(\sigma)} \lec \ve \| f \|_{L^p(\sigma)}, \quad \textup{for}\,\,p\in(1,\infty). 
$$
Using this and  \eqref{ntMFuf}, for $p\in(1,\infty)$,  we get that 
$$
\|\calN(F)\|_{L^p(\sigma)} \lec \|\calN(\upsilon_f-F)\|_{L^p(\sigma)}+\|\calN(\upsilon_f)\|_{L^p(\sigma)}\lec \|f\|_{L^p(\sigma)}.
$$
Moreover, combining  \eqref{graduf}, \eqref{eq:pwMFu},  \eqref{eq:gradient-bound-approx}, the fact that 
$ |\nabla\vp_\delta(x)| \lec\delta_\om(x)^{-1}$,
and using the $L^p$-boundedness of the Hardy-Littlewood maximal function, we can easily infer that
$$
\|\calN(\delta_\om \nabla F)\|_{L^p(\sigma)}\lec \|f\|_{L^p(\sigma)}, \quad \textup{for}\,\,p\in(1,\infty). 
$$
The estimates for $p=\infty$ can be proved similarly and  the routine details are omitted.

Note that the extension $F$ is Lipschitz in $\oom$. Indeed,  if $E:=\supp f$,  in light of Lemmas \ref{lem: gradufestimates} and   \ref{lem: lemma1}, and Theorem \ref{thm: approxestimates},  we infer that for every $x\in\om$, 
\begin{align*}
|\nabla \upsilon_f(x)(1-\vp_\delta(x))| &\lec \Lip(f),\\
|\nabla u(x) \, \vp_\delta(x)|&\lec \delta_\om(x)^{-1}\|f\|_{L^\infty(\sigma)} |\vp_\delta(x)| \lec \frac{1}{\delta}\Lip(f)\diam E,\\
|(u(x)-\upsilon_f(x))\nabla\vp_\delta(x)|&\lec\frac{1}{\delta} \,
 \inf_{\zeta\in B(\xi_x, 2\delta_\om(x))\cap\pom}\calN(u-\upsilon_f)(\zeta) 
\lec_{\ve} \frac{1}{\delta} \| \mathcal{M}(f) \|_{L^\infty} \\
&\lec \frac{\|f\|_{L^\infty(\sigma)}}{\delta} 
\lec \frac{\Lip(f)}{\delta}\diam E.
\end{align*}
These estimates imply that $\|\nabla F\|_{L^\infty(\om)}\lec_{\delta, \diam E}\Lip(f)$. Moreover, since $\upsilon_f\in\Lip(\oom)$, 
$F = \upsilon_f$ around $\pom$,  and $\upsilon_f|_{\pom}=f$,  we  deduce that $F\in\Lip(\oom)$ with $F|_\pom=f$ and 
$\Lip(F)\lec_{\delta, \diam E}\Lip(f)$, concluding the proof of the theorem.  

\end{proof}

\begin{remark}
Note that the convergence to the boundary is inherited from the one of $\upsilon_f$ and there is no need for an iteration argument as in the proof of Theorem \ref{th: introtheorem2}. 
\end{remark}

\vv

We now turn our attention to the  construction of an extension of $\BMO$-boundary functions.

\begin{theorem}\label{thm: Lipschitz extension in BMO}
Let $ \om \in \AR(s)$ for $s \in (0, n]$. If $f\in\Lip_c(\pom)$, then there exists an extension $ {F} :\oom \to \R $ such that
\begin{enumerate}
\item[(i)] ${F} \in C^\infty(\om)\cap\Lip(\oom)$, 
\item[(ii)] $ {F}|_{\pom} = f $ continuously,
\item[(iii)] $\displaystyle \sup_{\xi\in\pom}\calN_{\sharp}({F})(\xi) +  \sup_{x\in\om}\delta_\om(x)|\nabla {F}(x)|  \lec \| f \|_{\BMO(\sigma)} $, 
\item[(iv)] $\displaystyle \sup_{\xi\in\pom}\calC_s( \nabla {F} )(\xi) \lec \|f\|_{\BMO(\sigma)}$,
\item[(v)] $\displaystyle \sup_{\xi\in\pom}[\calC_s(\delta_\om |\nabla F|^2)(\xi)]^{1/2} \lec \|f\|_{\BMO(\sigma)}$.
\end{enumerate}
\end{theorem}

\begin{proof}

Let $w$ be the approximation of $\upsilon_f$ given by Theorem \ref{thm:e-approxLp} and define 
\begin{equation}\label{eq: modification BMO}
{F}(x) := \upsilon_f(x)( 1 - \vp_\delta(x) ) + w(x) \vp_\delta(x).
\end{equation}
Then,  for any $\xi \in \pom$,
\begin{equation}\label{eq: Carl gradient mod V}
\calC_s( \nabla {F} )(\xi)  \leq \calC_s( \nabla \upsilon_f( 1 - \vp_\delta ))(\xi) + \calC_s(\nabla w )(\xi) + 
\calC_s( ( w - \upsilon_f ) \nabla \vp_\delta) (\xi)
\end{equation}
and 
\begin{equation}\label{eq: CdeltagradF2BMO}
\calC_s(\delta_\om|\nabla F|^2)\lec \calC_s(\delta_\om|\nabla \upsilon_f|^2(1-\vp_\delta)^2) + \calC_s(\delta_\om|w-\upsilon_f|^2|\nabla\vp|^2) + \calC_s(\delta_\om |\nabla w|^2).
\end{equation}

For the second term on the right hand side  of \eqref{eq: Carl gradient mod V} we just use \eqref{eq: Linf-carlu-bmo},  while for the first one, by  \eqref{eq:LpLipextension2bis},
we have that 
\begin{equation}\label{eq: Carl grad uf bmo}
\calC_s(\nabla \upsilon_f (1 - \vp_\delta))(\xi) \lec  \delta \, \Lip(f).
\end{equation}
The third term can be bounded as  in \eqref{eq:Linftycarleson1} and get
\begin{equation}\label{eq: Linfty Carlgradphidelta V-uf}
\sup_{\xi\in\pom} \calC_s( \nabla \vp_\delta ( w - \upsilon_f ) )(\xi) \lec \| f \|_{\BMO}.
\end{equation}
Combining \eqref{eq: Carl gradient mod V}, \eqref{eq: Carl grad uf bmo}, and \eqref{eq: Linfty Carlgradphidelta V-uf}, and choosing $\delta= \|f \|_{\BMO} / \Lip(f)$,  we obtain (iv).  Moreover,  (v) follows from (iv) and the second estimate of (iii).

For the sharp non-tangential maximal function estimate,  note that since $ {F}-\upsilon_f=\vp_\delta(w-\upsilon_f)$, 
using \eqref{eq: sharpntMFuf-bmo} and \eqref{eq: pwMFBMO}, we get that for every $\xi \in\pom$  it holds  that
$$
\calN_{\sharp}({F})(\xi) \leq 2 \calN({F}-\upsilon_f)(\xi)+\calN_{\sharp}(\upsilon_f)(\xi) \leq 2 \calN(w-\upsilon_f)(\xi)
+ \calN_{\sharp}(\upsilon_f)(\xi) \lec \|f\|_{\BMO(\sigma)}. 
$$

It remains to prove that  ${F} \in \Lip(\om)$.  We first show that
\begin{equation}\label{eq:LipFBMOcont}
\| \nabla {F} \|_{L^\infty(\sigma)} \lec \Lip(f).
\end{equation}
To this end, for every $x\in\om$,  by Lemma \ref{lem: lemma1}, we have that
$$
|\nabla \upsilon_f(x) (1 - \vp_\delta(x) ) | \lec \Lip(f).
$$
By  \eqref{eq: pwMFBMO} and the fact that $\delta_\om(x) \approx \delta$ in the support of $ \nabla \vp_\delta$,  we obtain
\begin{align*}
|(w(x)-\upsilon_f(x)) \nabla \vp_\delta(x)| \lec_{\ve} \delta_\om(x)^{-1} \|f\|_{\BMO(\sigma)} \approx \delta^{-1} \|f\|_{\BMO(\sigma)}=\Lip(f).
\end{align*}
while,  by \eqref{eq: ntmf-gradu estimate bmo} it holds that
$$
|\nabla w(x)\vp_\delta(x)| \lec \| f\|_{\BMO(\sigma)} \, \delta_\om(x)^{-1} |\vp_\delta(x)| \lec \delta^{-1} \| f \|_{\BMO(\sigma)}=\Lip(f),
$$
which implies \eqref{eq:LipFBMOcont}.   By construction $F$ is continuous around the boundary  and $ {F}|_\pom = f $ continuously, which implies that $ {F} \in \Lip(\oom)\cap C^\infty(\oom) $ with
$\Lip({F})\lec \Lip(f)$.  Moreover, combining the last two estimates above  and \eqref{gradufbmo},  we get that 
$$
\sup_{x\in\om}\delta_{\om}(x)|\nabla {F}(x)|\lec \|f\|_{\BMO(\sigma)},
$$
which concludes the proof of the theorem.
\end{proof}

\vv

Our last goal is to modify that the extensions constructed in Theorems \ref{thm: Lipschitz extension in Lp} and \ref{thm: Lipschitz extension in BMO} so that they are also in $\dot W^{1,2}(\om;\omega_{s})$.  This will conclude the proof of Theorem \ref{th: introtheorem4}.

\begin{theorem}\label{thm:W12}
Let $\om \in \AR(s)$  for $ s \in (0,n]$.  If $f\in\Lip_c(\pom)$, then there exists an extension  $F_0 \in \dot W^{1,2}(\om; \omega_{s})$ (resp. $\bar F_0  \in \dot W^{1,2}(\om; \omega_{s})$) that satisfies the conclusions $(i)$-$(iv)$ of  Theorem \ref{thm: Lipschitz extension in Lp} (resp.  Theorem \ref{thm: Lipschitz extension in BMO}).
\end{theorem}

\begin{proof}
Let  $f \in \Lip_c(\pom)$, $E:=\supp f$, and $r_0:=\diam E$. Without loss of generality we  may assume that $0 \in E$ and so $E \subset B(0,r_0)$.   Now let $B= B(0,  M r_0)$ for some   $M >1$ large enough depending on the Whitney constants so that for every $P \in \WW(\om)$ satisfying  $\ell(P) \leq M^{-1}  r_0$ and $1.2P \cap (2B \setminus B) \neq \emptyset$  it holds that $b(P) \cap E = \emptyset$.  We denote the collection of all such Whitney cubes by $ \mathscr P_s(E)$ (``$s$" stands for ``small").  We also denote by $ \mathscr P_l(E)$ the collection of $P \in \WW(\om)$ satisfying  $\ell(P) > M^{-1}  r_0$ and $1.2P \cap (2B \setminus B) \neq \emptyset$  (``$l$" stands for ``large").  It is easy to see that 
\begin{equation}\label{eq:packing6.4}
\sum_{\substack{Q \subset  R: Q=b(P)\\ P \in  \mathscr{P}_l(E)}} \sigma(Q) \lec r_0^s\lec \sigma(R).
\end{equation}
Note that if $x \in (2B \setminus B) \cap \om$  and  there exists $P \in {\mathscr P}_s(E)$ such that $x \in 1.1 P$,  then the extension $F$ of Theorem \ref{thm: Lipschitz extension in Lp} satisfies $F(x)=0$  (resp.  $\bar F$ of Theorem \ref{thm: Lipschitz extension in BMO} satisfies $\bar F(x)=0$).  We define now the cut-off function $\psi_{r_0} \in C^\infty_c(\rrn)$ such that $0 \leq \psi_{r_0}  \leq 1$, $\psi_{r_0} =1$ in $\overline B$, $\psi_{r_0} =0$ in $\rrn \setminus 2B$,  and $|\nabla \psi_{r_0}| \lec 1/r_0$. Then we define 
$$
F_0(x):=F(x)\, \psi_{r_0}(x) \qquad \textup{and} \qquad \bar F_0(x):=\bar F(x)\, \psi_{r_0}(x)
$$
It is clear  that $F_0|_{\pom}=f$ (resp.  $\bar F_0|_{\pom}=f$) and remark that 
\begin{equation} \label{eq:support-nabla-F0}
\begin{rcases*}
\supp( F \nabla \psi_{r_0}) \\
\supp( \bar F \nabla \psi_{r_0}) 
\end{rcases*} 
\subset T_{r_0}:= \{x \in \om \cap (2B \setminus B): \dist(x, \pom) \geq c_0 r_0 \},
\end{equation}
 for some $c_0 \in (0,1)$ small enough depending on $M$ and the Whitney constants.  Therefore, for any  $\xi \in \pom$,  if $B(\xi,r) \cap \supp( F \nabla \psi_{r_0}) \neq \emptyset$, then  $r \geq c_1\max\{ r_0, \dist(\xi,  2B \setminus B)\}$ for some constant $c_1 \in (0,1)$ depending on $c_0$.  Moreover, 
\begin{equation}\label{eq:Fpsi}
|F(x)|| \nabla \psi_{r_0}(x)|\lec \delta_\om(x)^{-1}|F(x)|\quad \textup{and}\quad |\bar F(x)|| \nabla \psi_{r_0}(x)|\lec \delta_\om(x)^{-1}|\bar F(x)|.
\end{equation}
 We will only prove the theorem for $F_0$ and unbounded domains with unbounded boundary since  for domains with compact boundary the arguments are similar. 

We first prove that $F_0$ satisfies the conclusions of Theorem \ref{thm: Lipschitz extension in Lp}.  It is easy to see that  $\| \calN(F_0)\|_{L^p(\sigma)} \lec \|f\|_{L^p(\sigma)}$ since $|F_0| \leq |F|$ and the same estimate holds for $F$.   We have that  $\nabla F_0 = \nabla F \, \psi_{r_0} + F \,\nabla \psi_{r_0}$ and it is easy to see that $\|\calN(\delta_\om \nabla F_0)\|_{L^p(\sigma)} \lec \|f\|_{L^p(\sigma)}$ by \eqref{eq:Fpsi} and the  estimates in (ii) and (iii)  for $F$ in Theorem \ref{thm: Lipschitz extension in Lp}.  To  prove the  estimate $\| \mathcal{C}_s(\nabla F_0)\|_{L^p(\sigma)} \lec \|f\|_{L^p(\sigma)}$,  it is enough to show that $\| \mathcal{C}_s(F\, \nabla \psi_{r_0})\|_{L^p(\sigma)}\lec \|f\|_{L^p(\sigma)}$.  Thus,  for any such $r$, we have that
 \begin{align*}
 \int_{B(\xi,r) \cap \om} \,\sup_{y \in B_x} |F(y)|| \nabla \psi_{r_0}(y)| \,\omega_{s}(x) \,dx \lec r_0^{-1} \int_{B(\xi,r) \cap (2B \setminus B) \cap \om} \,\sup_{y \in B_x}  |F(y)|\,\omega_{s}(x) \,dx.
 \end{align*}
By \eqref{eq:approximatingdefinition},  \eqref{stopcond}, \eqref{modifu}, and  the choice of the constant $M$, for any $x \in B(\xi,r) \cap (2B \setminus B) \cap \om$ and for any $y \in B_x$,  it holds that 
  \begin{align*}
|F(y)|& \leq \sum_{R \in \ttt} \sum_{\substack{P \in {\mathscr P}_l(E) \setminus \mathcal{P}_0  \\ b(P) \in \tree(R)}} |m_{\sigma, R}f -m_{\sigma, b(P)}f|  \, \vp_P(y) + \sum_{\substack{P \in {\mathscr P}_l(E)}} |m_{\sigma, b(P)}f|  \, \vp_P(y)\\
&\leq \ve \sum_{R \in \ttt} \sum_{\substack{P \in {\mathscr P}_l(E) \setminus \mathcal{P}_0  \\ b(P) \in \tree(R)}} Mf(b(P)) \, \vp_P(y)+\sum_{\substack{P \in {\mathscr P}_l(E)}} |m_{\sigma, b(P)}f|  \, \vp_P(y),\notag
   \end{align*}
  which, in turn, implies that
    \begin{align*}
 r_0^{-1} \int_{B(\xi,r) \cap (2B \setminus B) \cap \om} \,\,\sup_{y \in B_x} |F(y)|\,\omega_{s}(x) \,dx &\lec r_0^{-1}  \sum_{\substack{P \in {\mathscr P}_l(E)}} \ell(P) \sigma(b(P)) Mf(b(P))\\
& \lec  \int_{CB}  \mathcal{M}(f) \,d\sigma \leq \int_{B(\xi, C'r)} \mathcal{M}f \,d\sigma,
   \end{align*}
   for some constant $C'>1$ depending on $C$ and $M$. This readily yields that for every $\xi \in \pom$,
   $$
   \mathcal{C}_s(F \,| \nabla \psi_{r_0}|)(\xi) \lec \mathcal{M}( \mathcal{M}f)(\xi),
   $$
and the desired estimate  follows for any $p \in (1,\infty]$.  Similarly, we can show that 
$$
 \mathcal{C}_s(\delta_\om |F|^2 \,| \nabla \psi_{r_0}|^2)(\xi) \lec \mathcal{M}( (\mathcal{M}f)^2)(\xi)
 $$
  which implies (v)  of Theorem \ref{thm: Lipschitz extension in Lp} for $q>2$. The case $q=2$ once again can be treated separately using  \eqref{eq:comparable L2C}.  We omit the  details.

It remains to  show  that $F_0 \in \dot W^{1,2}(\om; \omega_{s})$ since it is clear that $F_0 \in C^\infty(\om) \cap \Lip(\oom)$ such that $F_0|_{\pom}=f$.  To this end,  by the  definition of $F$ and the proof of its Lipschitz property,  we get that 
\begin{align*}
\int_{ \om}  |\nabla F_0|^2 \omega_{s}(x)\,dx &\leq \int_{2B \cap \om}  |\nabla F|^2 \omega_{s}(x)\,dx\\
&\lec \Big( \Big(1+\frac{r_0^{2}}{\delta^2} \Big)\min(r_0,\delta)+ \frac{r_0^{3}}{\delta^2} \Big) \, r_0^s \,   (\Lip f)^2.
  \end{align*}
  Moreover,  using \eqref{eq:support-nabla-F0} and the fact that $\supp f =E$,  we can show that
  \begin{align*}
\int_{ \om}  | F \nabla \psi_{r_0}|^2 \omega_{s}(x)\,dx \lec r_0^{-2-n+s} \int_{T_{r_0}}  | F|^2 \,dx\lec  r_0^{s-1} \|f\|_{L^\infty(\pom)}^2 \leq r_0^{s+1}  \,   (\Lip f)^2,
  \end{align*}
  concluding the proof of the Theorem for $L^p$,  $p\in (1,\infty]$.

To demonstrate the theorem for $\bar F_0$,  it  is enough to prove the Carleson estimate as the  estimates for the non-tangential maximal functions are easy to prove and we leave them as an exercise.    By \eqref{eq:support-nabla-F0}, \eqref{eq:approximatingdefinition},   \eqref{stopcond}, \eqref{eq: modification BMO},  and  the choice of the constant $M$, for any $x \in B(\xi,r) \cap (2B \setminus B) \cap \om$ and every $y \in B_x$, it holds that 
    \begin{align*}
|\bar F(y)|& \leq \sum_{R \in \ttt} \sum_{\substack{P \in {\mathscr P}_l(E) \setminus \mathcal{P}_0  \\ b(P) \in \tree(R)}} |m_{\sigma, R}f -m_{\sigma, b(P)}f|  \, \vp_P(y) + \sum_{\substack{P \in {\mathscr P}_l(E)}} |m_{\sigma, b(P)}f|  \, \vp_P(y)\\
&\leq \ve\|f\|_{\BMO(\sigma)}+\sum_{\substack{P \in {\mathscr P}_l(E)}} |m_{\sigma, b(P)}f|  \, \vp_P(y).
   \end{align*}
   It is not hard to see that for every $P\in {\mathscr P}_l(E)$, there exists $P^* \in \WW(\om)$ such that  $\ell(P^*)\approx \dist(P^*, P) \approx r_0$ and $b(P^*) \subset \pom \setminus E$, and it holds that $m_{\sigma, b(P^*)}f=0$.  Thus,  for any $x \in B(\xi,r) \cap (2B \setminus B) \cap \om$,
    \begin{align*}
  |\bar F(y)|&\leq  \ve\|f\|_{\BMO(\sigma)}+\sum_{\substack{P \in {\mathscr P}_l(E)}} |m_{\sigma, b(P)}f-m_{\sigma, b(P^*)}f|  \, \vp_P(x)
  \lec \|f\|_{\BMO(\sigma)},
   \end{align*}
   which,  arguing as above,  implies that $\sup_{\xi \in \pom}   \mathcal{C}_s(F \, \nabla \psi_{r_0})(\xi) \lec \|f\|_{\BMO(\sigma)}$.  This finishes the proof of the theorem since the same proof as above shows that $\bar F_0 \in \dot W^{1,2}(\om;\hm_s)$.
\end{proof}

\vvv

\section{Applications  to Boundary Value Problems for systems of elliptic equations}\label{sec: applications}\label{sec:applications}

We define the {\it variational co-normal derivative} of a solution $v \in \dot W^{1,2}(\om; \bC^m) $ of $L v =-\dv {\bf \Xi} + H$ in $\om$,  and denote it by $\partial_{\nu_A} v$,  to be  the  linear functional defined in terms of the sesquilinear form associated to $L$ as follows:
\begin{align*}
\langle \partial_{\nu_A}  v,  \vphi \rangle := \ell_{v}(\vphi):= {\bf B} &(v,\Phi)=\sum_{\alpha, \beta=1}^m \sum_{i, j=1}^{n+1}   \int_\om a^{\alpha\beta}_{ij}(x) \, \partial_j v^\beta(x) \partial_i \Phi^\alpha\,dx\\
& -\sum_{\alpha=1}^m \sum_{i=1}^{n+1} \int_\om  \Xi_i^\alpha(x)  \partial_i \Phi^\alpha(x)\,dx- \sum_{\alpha=1}^m \int_\om H^\alpha(x)\Phi(x)^\alpha\,dx,
\end{align*}
where $\vphi \in \Lip_c(\pom; \bC^m)$ and $\Phi \in \dot W^{1,2}(\om; \bC^m) \cap \Lip(\oom; \bC^m)$ such that $\Phi|_{\pom} = \vphi$.

\vv

\begin{lemma}\label{lem:ell_v} $\ell_{v}$ is unambiguously defined. 
\end{lemma}
\begin{proof}
If $\Phi^1, \Phi^2 \in \dot W^{1,2}(\om; \bC^m) \cap \Lip(\oom; \bC^m)$ such that $\Phi^1|_{\pom}=\Phi^2|_{\pom} = \vphi$ and $\Phi^1\neq \Phi^2$,  then $\Psi:=\Phi^1-\Phi^2 \in \dot W^{1,2}(\om:\bC^m)\cap \Lip(\oom; \bC^m)$ such that $\Psi|_{\pom} = 0$,  which implies that $\Psi \in Y^{1,2}_0(\om)$. Since $L v =-\dv {\bf \Xi} + H$, we have that ${\bf B}(v,\Psi)=0$ and thus,  ${\bf B} (v,\Phi^1)={\bf B} (v,\Phi^2)$. So  any extension of $\vphi$  belonging to  $\dot W^{1,2}(\om; \bC^m) \cap \Lip(\oom; \bC^m)$ defines the same linear functional $\ell_{v}$.
\end{proof}
\vv

From now on, we assume that $\om \in  \AR(n)$, $n\geq 2$,  and that either $\om$ is bounded or $\pom$ unbounded.  This is because we will  use the duality $ \nn_{q, p}(\om)=(\cc_{s, q', p'}(\om))^*$. 

\vv

In the sequel, for simplicity, we will prove our results specifically for real elliptic equations (i.e., when $m=1$). Nevertheless, the proofs for $m > 1$ and complex-valued coefficients are identical (see also Remark \ref{rem:vector-valued}).
\vv

\subsection{Some connections between Poisson Problems and Boundary Value  Problems}

We set 
$$
X^q(\sigma)=
\begin{cases}
L^q(\sigma) & \textup{if}\,\, q \in (1, \infty)\\
H^1(\sigma) & \textup{if}\,\, q =1,
\end{cases}
$$
where $H^1(\sigma)$ is the atomic Hardy space with respect to $\sigma$.

\begin{proposition}\label{prop: PR implies Rellich}
If   $(\pr)$ is solvable in $\om$ for some $p>1$,  then its solution $u$ satisfies the  one-sided Rellich-type inequality 
\begin{equation}\label{eq: Rellich from PR}
\|\partial_{\nu_A} u\|_{L^p(\sigma)} \lec \|H\|_{\cc_{2_*, p}(\om)} + \|  {\bf \Xi}  /\delta_\om \|_{\cc_{2, p}(\om)}.
\end{equation}
Moreover, if $(\wt \pre_q^L)$,  $q \in [1, 2]$,  is solvable in $\om$ with data $H=0$ and ${\bf \Xi}\in L^2_c(\om)$,  then its solution satisfies the  one-sided Rellich-type inequality 
\begin{equation}\label{eq: Rellich from PR endpoint}
\|\partial_{\nu_A} u \|_{X^q(\sigma)} \lec \| {\bf \Xi}\|_{T^p_2(\om)}, \quad q \in [1, 2].
\end{equation}
\end{proposition}

\begin{proof}
Suppose that $u$ is the solution of $(\pr)$.  Let  $\vp\in\Lip_c(\pom)$ and  $F\in \dot W^{1, 2}(\om) \cap \Lip(\oom)$ be the Varopoulos extension of the $L^p$-boundary data $\vp$ constructed in Theorem \ref{th: introtheorem4}.  Then,  by Lemma \ref{lem:ell_v}, we get  
$$
|\ell_u(\vp)|= |{\bf B} (u, F)| \leq \|A\|_{L^\infty(\om)}\int_\om|\nabla u||\nabla F|+\int_\om|H||F| +\int_\om |{\bf \Xi}| \, |\nabla F|.
$$
By duality (see  \cite[Proposition 2.4]{mpt22}),  \eqref{eq: poisson regularity},  and the properties of the extension $F$, we infer that 
\begin{align*}
\int_\om|\nabla u||\nabla F| &\lec \| \wt \calN_2(\nabla u)\|_{L^p(\sigma)}\|\calC_{2}(\nabla F)\|_{L^{p'}(\sigma)}
\lec (\|H\|_{\cc_{2_\ast, p}} +\|  {\bf \Xi} /\delta_\om\|_{\cc_{2, p}(\om)})\|\vp\|_{L^{p'}}.
\end{align*}
By duality and using Theorem \ref{th: introtheorem4} (ii) and (iii),  we infer that
\begin{align*}
\int_\om |H||F| &+\int_\om  |{ \bf \Xi}| \,| \nabla F|\\
&\lec\|H\|_{\cc_{2_\ast, p}}  \| \wt \calN_{2}(F)\|_{L^{p'}(\sigma)} + \|  {\bf \Xi} /\delta_\om \|_{\cc_{2, p}(\om)} \|\calN_{2^*}(\delta_\om \nabla F)\|_{L^{p'}(\sigma)} \\
&\lec (\|H\|_{\cc_{2_\ast, p}} +\|  {\bf \Xi} /\delta_\om\|_{\cc_{2, p}(\om)})\|\vp\|_{L^{p'}(\sigma)}.
\end{align*}
Thus,  by the above estimates,  the density of $\Lip_c(\pom)$ in $L^p(\sigma)$, and   duality, we infer \eqref{eq: Rellich from PR}.

 Let now  $v$ be the solution of $(\wt\pre_1^L)$ and let  $\vp\in \Lip_c(\pom)$.  If $\wt F$ is the Varopoulos extension of $\vp$  constructed in the second part of Theorem \ref{th: introtheorem4},  by Lemma \ref{lem:ell_v}, we get that
$$
|\ell_v(\vp)|\leq \|A\|_{L^\infty(\om)}\int_\om |\nabla v||\nabla  \wt F| + \int_\om |{\bf \Xi} ||\nabla \wt F|.
$$
By duality,  \eqref{eq: poisson regularity},  and Theorem \ref{th: introtheorem4} (iii) (for $\BMO$),  we have that
$$
\int_\om |\nabla v ||\nabla \wt F| \lec \|\wt \calN_2(\nabla v)\|_{L^1(\sigma)} \sup_{\xi\in\pom}\calC(\nabla \wt F)(\xi) 
\lec \|{\bf \Xi}\|_{T^1_2(\om)}\|\vp\|_{\BMO(\sigma)}.
$$
For the second term,  by $T^\infty_2(\om)= (T^1_2(\om))^*$ and Theorem \ref{th: introtheorem4} (v) (for $\BMO$),  it holds that 
$$
\int_\om |{\bf \Xi}| |\nabla \wt F| \lec \|{\bf \Xi} \|_{T^1_2(\om)} \| \delta_\om \nabla\wt F \|_{T^\infty_2(\om)} \lec \|{\bf \Xi} \|_{T^1_2(\om)}  \| \vp \|_{\BMO(\sigma)}.
$$
Thus 
$$
|\ell_v(\vp)|\lec \|{\bf \Xi}\|_{T^1_2(\om)} \|\vp\|_{\BMO(\sigma)},
$$
which,  since $\overline{\Lip_c(\pom)}^{\VMO(\sigma)}=\VMO(\sigma)=(H^1(\sigma))^*$,  implies  \eqref{eq: Rellich from PR endpoint}. The proof of \eqref{eq: Rellich from PR endpoint} for $q \in (1,2]$,  is similar and  is omitted.  
\end{proof}

\vv

\begin{theorem}\label{th: PR implies D}
 If $(\pr)$ with ${\bf \Xi}=0$ is solvable in $\om$ for some $p>1$,  then $(\dirprime)$ is also solvable in $\om$, where $1/p+1/p'=1$.
\end{theorem}

\begin{proof}
Let $f\in\Lip_c(\pom)$ and let $u$ be the  solution to \eqref{eq: Dirichlet problem} for $L^*$ with data $f$.  Due to the density of $L^\infty_c(\om)$ in $\cc_{2_\ast, p}(\om)$ and duality, we have that 
$$
\|\calN_{2^*}(u)\|_{L^{p'}(\sigma)}\lec \sup_{\substack{H \in L^\infty_c(\om): \\ \| H\|_{\cc_{2_\ast, p}(\om)}=1}} \Big| \int_\om uH \Big|.
$$
Fix such an $H \in L^\infty_c(\om)$ and let $w \in Y^{1, 2}_0(\om)$ be the solution to $(\pr)$ with data ${\bf \Xi}=0$ and $H$. Then,
using the fact that $L^* u=0$ and \eqref{eq: Rellich from PR}, we estimate
\begin{align*}
\Big|\int_\om u H \Big| &=\Big|- \int_\om \nabla u A\nabla w + \int_\pom \partial_{\nu_A}w f \Big| 
= \Big| \int_\pom \partial_{\nu_A}w f \Big| \lec  \|\calC_{2_\ast}(H)\|_{L^p(\sigma)} \| f\|_{L^{p'}(\sigma)},
\end{align*}
which readily implies that
$$
\|\calN_{2^*}(u)\|_{L^{p'}(\sigma)} \lec \|f\|_{L^{p'}(\sigma)}.
$$
\end{proof}

\vv

\begin{theorem}\label{th: PR implies D endpoint}
If  $(\wt{\pre}^L_q)$ with $H=0$ is solvable in $\om$ for $q \in[1,2]$, then  both $({\wt{\mathrm{PD}}}^{L^*}_{q'})$  with $H=0$ and  $({\wt{\mathrm{D}}}^{L^*}_{q'})$  are solvable in $\om$.
\end{theorem}

\begin{proof}
Let $v_1$ be the solution of \eqref{eq:variational Dirichlet problem} with data ${\bf \Xi} \in L^2_c(\om;\rrn)$ and $H=0$ and $v_2$ be the solution of \eqref{eq: Dirichlet problem} with data $\vp \in \Lip_c(\pom)$, and we define $w=v_1+v_2$. Using the tent space duality $(T^{q'}_2(\om))^*=T^q_2(\om)$ along with the density of $L^2_c(\om)$ functions in $T^q_2(\om)$,  it holds that
$$
\|\delta_\om \nabla w\|_{T^{q'}_2(\om)}\approx \sup_{\substack{ {\bf \Psi} \in L^2_c(\om): \\ \| {\bf \Psi}\|_{T^{q}_2(\om)}=1}} \Big|\int_\om \delta_\om(x) \nabla w \, {\bf \Psi}
\frac{dx}{\delta_\om(x)}\Big| = \sup_{\substack{ {\bf \Psi} \in L^2_c(\om): \\ \| {\bf \Psi} \|_{T^q_2(\om)}=1}}\Big| \int_\om \nabla w \, {\bf \Psi} \Big|.
$$
Then,  if $u\in Y^{1, 2}_0(\om)$ is the solution of $(\pre^L_q)$,  with data $\Psi\in L^2_c(\om)$ and $H=0$,  by duality and \eqref{eq: Rellich from PR endpoint}, we have 
\begin{align*}
\Big| \int_\om {\bf \Psi} \, \nabla w \Big| &= \Big| -\int_\om A\nabla u\nabla w + \int_\pom \partial_{\nu_A}u f \Big| 
\leq \Big| \int_\om A^* \nabla w \nabla u\Big| +\Big|   \int_\pom \partial_{\nu_A}u \, f \Big| \\
&= \Big|\int_\om {\bf \Xi} \, \nabla u \Big| + \Big|\int_\pom \partial_{\nu_A}u \, f \Big|\\
&\lec \|\calC_2({\bf \Xi})\|_{L^{q'}(\sigma)} \|\wt{\calN}_2(\nabla u)\|_{L^q(\sigma)} + \| {\bf \Psi } \|_{T^q_2(\om)}\|f\|_{Y^{q'}(\sigma)} \\
&\lec \left( \| {\bf \Xi} \|_{\cc_{2, q'}(\om)} +\|f\|_{Y^{q'}(\sigma)} \right) \|  {\bf \Psi} \|_{T^q_2(\om)},
\end{align*}
which proves the desired estimates  \eqref{eq: Dirichlet solvability area} and \eqref{eq:end-poisson regularity}.
\end{proof}

\vv

\begin{theorem}\label{th: PD implies D}
\begin{itemize}
\item[(i)] If  $(\pd^L_p)$ is solvable in $\om$ with $H=0$ for some $p\in(1,\infty)$,  then $(\mathrm{D}^L_p)$ is also solvable in $\om$.  
\item[(ii)] If $(\wt{\pd}^L_{q})$ is solvable in $\om$ with $H=0$ for some $q\in [2,\infty]$,  then $(\wt{\mathrm{D}}^L_q)$ is also solvable in $\om$.  
\end{itemize}
\end{theorem}

\begin{proof}
Let $f\in \Lip_c(\pom)$ and let $F$ be the Varopoulos extensions of the $L^p$-boundary data $f$ given by Theorem \ref{th: introtheorem4}.  In the construction of the  solution  of \eqref{eq: Dirichlet problem} with data $f$,   we can use $F$ as the Lipschitz extension of $f$.  So
if $u$ is the aforementioned solution, then $u= w+F$,  where $w$ is the solution of \eqref{eq:variational Dirichlet problem}  with ${\bf \Xi}= - A\nabla F \in L^2(\om)$ and $H=0$.  Then,  by Theorem \ref{th: introtheorem4} (ii) and \eqref{eq: poisson d regularity}, we have that
\begin{align*}
\|\wt\calN_{2^*} (u)\|_{L^p(\sigma)} &\leq \|\calN(F)\|_{L^p(\sigma)} + \|\wt\calN_{2^*}(w)\|_{L^p(\sigma)}
\lec \|f\|_{L^p(\sigma)} + \|\calC_2({\bf \Xi})\|_{L^p(\sigma)} \\
&\leq \|f\|_{L^p(\sigma)} + \|A\|_{L^\infty(\om)}\|\calC(\nabla F)\|_{L^p(\sigma)} \lec \|f\|_{L^p(\sigma)},
\end{align*}
showing (i).  Similarly,  assuming  $(\wt{\pd}^L_{q})$ is solvable in $\om$ with $H=0$ for some $q\in [2,\infty]$, we obtain
\begin{align*}
\|\delta_\om \nabla u \|_{T^q_2(\om)} &\leq \|\delta_\om \nabla F \|_{T^q_2(\om)}  +\|\delta_\om \nabla w \|_{T^q_2(\om)} 
\lec \|f\|_{Y^q(\sigma)} + \|\calC_2({\bf \Xi})\|_{L^q(\sigma)} \\
&\leq \|f\|_{Y^q(\sigma)} + \|A\|_{L^\infty(\om)}\|\calC(\nabla F)\|_{L^p(\sigma)} \lec \|f\|_{Y^q(\sigma)},
\end{align*}
where we used  Theorem \ref{th: introtheorem4} (v) (when $q=\infty$,  we use the extension $\wt F$ of Theorem \ref{th: introtheorem4}) and \eqref{eq:end-poisson regularity}, which finishes the proof of the Theorem.
\end{proof}

\vv

\begin{proof}[Proof of Theorem \ref{thm:applications}]
It follows by combining Theorems \ref{th: PR implies D}, \ref{th: PR implies D endpoint}, and  \ref{th: PD implies D}.
\end{proof}

\vv

\subsection{Conditional one-sided Rellich-type inequalities}

\begin{proposition}\label{prop: R imply Rellich D}
Suppose that $(\rpr)$ is solvable in  $\om$ for some $p\geq 1$.  If $u$ is the solution of \eqref{eq: Dirichlet problem} for $L^*$ in $\om$ with data $f\in\Lip_c(\pom)$,  it holds that
\begin{equation}\label{eq: RP implies Rellich for D}
\|\partial_{\nu_{A^*}}u \|_{(\dot{M}^{1, p}(\sigma))^*} \lec \|f\|_{Y_{p'}(\sigma)},
\end{equation} 
where $(\dot{M}^{1, p}(\pom))^*$ stands for the Banach space dual of $\dot{M}^{1, p}(\pom)/\R$ and $Y_{p'}(\sigma)$ is equal to $L^{p'}(\sigma)$ if $p>1$ and $\BMO(\sigma)$ if $p=1$.
\end{proposition}

\begin{proof}
By definition,
$$
\|\partial_{\nu_{A^*}} u\|_{(\dot{M}^{1, p}(\pom))^\ast} = \sup_{\substack{\vp\in\Lip_c(\pom) : \\ \|\vp\|_{\dot{M}^{1, p}(\sigma)}=1}} |\langle \partial_{\nu_{A^*}}u , \vp \rangle |.
$$
Fix $\vp \in \Lip_c(\pom)$ such that $\|\vp\|_{\dot{M}^{1, p}(\sigma)}=1$ and let  $w$ be the solution of $(\rpr)$ with data $\vp$.  Let also $F\in\Lip(\oom)$ be the $L^{p'}$-Varopoulos extension of $f$ as constructed in Theorem 
 \ref{th: introtheorem4}. Then,  by Lemma \ref{lem:ell_v},  we have that 
$$
 \langle \partial_{\nu_{A^*}}u, \vp \rangle =  \int_\om A^*\nabla u \nabla w= \int_\om A \nabla w \nabla (u-F) + \int_\om A \nabla w \nabla F=\int_\om A \nabla w \nabla F,
$$
since $u-F \in Y^{1, 2}_0(\om)$  and $Lw=0$.  Therefore, by duality, Theorem 
 \ref{th: introtheorem4} (ii), and \eqref{eq: Dirichlet regularity estimate}, we infer that
\begin{align*}
| \langle \partial_{\nu_{A^*}}u, \vp \rangle | =\Big| \int_\om A \nabla w \nabla F \Big| &\leq \|A\|_{L^\infty(\om)} \|\wt \calN_2(\nabla u)\|_{L^p(\sigma)} \| \calC_{2}(\nabla F)\|_{L^{p'}} \\
&\lec \|\vp\|_{\dot{M}^{1, p}(\sigma)}\|f\|_{L^{p'}(\sigma)},
\end{align*}
which shows \eqref{eq: RP implies Rellich for D} for $p>1$. The proof in the case $p=1$ is similar and we leave it as an exercise.
\end{proof}

\vv

\begin{proposition}
Let $ q \geq 1$. If $u$ is a  solution of \eqref{eq:weaksolution} for $H \in L^\infty_c(\om)$ and ${\bf \Xi} \in L^\infty_c(\om; \rrn)$ such that  $\calN_q(\nabla u)\in L^p(\sigma)$ for $p>1$, it holds that
\begin{equation}\label{eq: Rellich general solution p>1}
\| \partial_{\nu_{A^*}} u \|_{L^p(\sigma)} \lec \|\nabla u\|_{\nn_{q, p}(\om)} +\|H\|_{\cc_{1, p}(\om)} + \|  {\bf \Xi}  /\delta_\om \|_{\cc_{1, p}(\om)}.
\end{equation}
If $u$ is a  solution of \eqref{eq:weaksolution} for $H=0$ and ${\bf \Xi} \in L^\infty_c(\om; \rrn)$ such that  $\calN_q(\nabla u)\in L^1(\sigma)$, 
\begin{equation}\label{eq: Rellich general solution p=1}
\| \partial_{\nu_{A^*}} u \|_{H^1(\sigma)} \lec \|\nabla u\|_{\nn_{q, 1}(\om)}+\| {\bf \Xi}\|_{T^1_2(\om)}.
\end{equation}
\end{proposition}

\begin{proof}
It follows by the same arguments that prove Proposition \ref{prop: PR implies Rellich}. We skip the details.
\end{proof}

\vvv

\appendix \label{appendix}
\section{}

\vvv
\begin{lemma}\label{lem: A less than C-ap}
Let $\om \in \AR(s)$ for $s \in (0,n]$, $ u  \in L^1_{\loc}(\om, \hm_s) $,  $ p \in [1,  \infty) $,  and $\alpha \geq 1$. Then there exists $C\geq 1$ such that for any $\xi \in \pom$ and $ r \in (0, 2\diam(\pom))$, it holds that
\begin{equation}\label{eq:A<Clocal-ap}
\| \mathcal{A}^{(\alpha)}(u {\bf 1}_{B(\xi,r)} ) \|_{L^p(\sigma|_{B(\xi,r)}) } \lec r^\beta  \| \mathscr{C}^{(\beta)}_s(u  {\bf 1}_{B(\xi,C r)}) \|_{L^p(\sigma|_{B(\xi,r)} )}.
\end{equation}
If $\beta=0$, it also holds 
\begin{equation}\label{eq:A<Cglobal-ap}
\| \mathcal{A}^{(\alpha)}(u) \|_{L^p(\sigma)} \lec  \| \mathscr{C}_s(u) \|_{L^p(\sigma)}.
\end{equation}
Moreover for $\beta=0$ and $1<p\leq \infty$ we have 
\begin{equation}\label{eq:C<Aglobal-ap}
\|\mathscr{C}_s(u)\|_{L^p(\sigma)} \lec \|\mathcal{A}^{(\alpha)}(u)\|_{L^p(\sigma)}.
\end{equation}
\end{lemma}

\begin{proof}
We adapt the proof of  \cite[Proposition 2.4]{HR18} and  argue by duality.  Indeed,  if $1/p + 1/{p'} = 1$,  we let $h\in L^{p'}(\sigma)$ be a non-negative function supported in $B(\xi,r)$ with $\| h \|_{L^{p'}(\sigma)}=1$ and such that  $\| \mathcal{A}(u {\bf 1}_{B(\xi,r)}) \|_{L^p(\sigma|_{B(\xi,r)}) }  \approx \int \mathcal{A}(u {\bf 1}_{B(\xi,r)}) h \,d\sigma$. Thus,
\begin{align*}
\| &\mathcal{A}(u {\bf 1}_{B(\xi,r)}) \|_{L^p(\sigma|_{B(\xi,r)})} \approx \int_{\pom}\Big( \int_{\gamma_\alpha(\xi) \cap B(\xi,r)} |u(y)|\delta_\om(y)^{-n}\, dy \Big)
h(\xi)\, d\sigma(\xi ) \\
&\leq \int_{\om \cap B(\xi,r)}  |u(y)| \delta_\om(y)^{s-n} \Big(\delta_\om(y)^{-s}\int_{B(y, \alpha \delta_\om(y))} h(\xi)\, d\sigma(\xi) \Big)\, dy \\
&= \int_{\om \cap B(\xi,r)} |u(y)| \delta_\om(y)^{s-n} H(y) \,dy\\
& = \int_0^{\infty} \int_{\om \cap B(\xi,r) \cap \{H(y) > \lambda\} } |u(y)|\delta_\om(y)^{s-n} \, dy \, d\lambda, 
\end{align*}
where, we have set
$$ 
H(y) := \delta_\om(y)^{-s}\int_{B(y,\alpha \delta_\om(y))} h(\xi)\, d\sigma(\xi). 
$$
For any  $y \in \om\cap B(\xi,r)$ we let $\hat y$ to be a point in $B(\xi,r) \cap \pom$ such that  $|y - \hat y| = \delta_\om(y)$ and set $B_{\hat y}:= B(\hat y,  (\alpha+1) \delta(y)) \supset B(y, \alpha \delta_\om(y)$.  Define 
$$
E_\lambda:=\{ y \in \om \cap B(\xi,r): H(y) >\lambda\}
$$ and note that for any $y \in E_\lambda$, it holds that $ m_{\sigma, B_{\hat y}} h > c \lambda $ for some $c \in (0,1)$ depending on $\alpha$.   If we set
$$
\widehat E_\lambda:= \{ \zeta \in \pom: \zeta= \hat y\,\, \textup{for some}\,\,y \in E_\lambda\}  \quad \textup{and}\quad \mathscr{B}_\lambda=\{ B_{\hat y}: y \in E_\lambda \},$$
then,    there exists large enough $C>1$,  such that 
$$
\bigcup_{\zeta \in \widehat E_\lambda} B_{\zeta} \cap \pom  \subset \{\zeta \in \pom : \mathcal M h(\zeta)> c \lambda\} \cap B(\xi, Cr).
$$
So,  by Vitali's covering lemma, there exists a subcollection $ \mathscr{G}_\lambda \subset \mathscr{B}_\lambda$ of pairwise disjoint balls such that  
$$
\bigcup_{B' \in  \mathscr{B}_\lambda} B' \subset \bigcup_{B \in  \mathscr{G}_\lambda} 5B.
$$
It is clear that 
$$
E_\lambda \subset  \bigcup_{B \in \mathscr{G}_\lambda} 5B
$$
and thus,
\begin{align*}
\int_{E_ \lambda}  |u(y)|\delta_\om(y)^{s-n}\, dy 
&\leq \sum_{B \in \mathscr{G}_\lambda}\int_{5 B }|u(y)| \delta_\om(y)^{s-n}  \, dy  \\
&\lec\sum_{B \in \mathscr{G}_\lambda} \sigma(B) \, r(B)^\beta\inf_{\zeta \in B \cap \pom} \mathscr{C}^{(\beta)}_s(u {\bf 1}_{B(\xi,C r)})(\zeta)\\
& \lec  r^\beta  \int_{B(\xi, Cr) \cap {\{ \mathcal{M} h > \lambda \}}} \mathscr{C}_s^{(\beta)} (u {\bf 1}_{B(\xi,C r)})(\zeta)\, d\sigma(\zeta).
\end{align*}
Therefore, since $ \| h \|_{L^{p'}(\sigma)} = 1$,
\begin{align*}
\| \mathcal{A}_{s,\alpha}(u {\bf 1}_{B(\xi,r)} ) \|_{L^p(\sigma|_{B(\xi,r)}) } &\lec r^\beta  \int_0^\infty \int_{B(\xi, Cr) \cap {\{ \mathcal{M} h > \lambda \}}}  \mathscr{C}_s^{(\beta)}(u {\bf 1}_{B(\xi,C r)} )(\zeta)\, d\sigma(\zeta) \, d\lambda \\
&\lec 
r^\beta \int_{B(\xi, Cr)} \mathscr{C}_s^{(\beta)}(u{\bf 1}_{B(\xi,Cr)} )(\zeta)\, \mathcal{M} h(\zeta)\, d\sigma(\zeta)\\
& \leq r^\beta \| \mathscr{C}_s^{(\beta)}(u {\bf 1}_{B(\xi, Cr)} ) \|_{L^p(\sigma|_{ B(\xi,r) } ) },
\end{align*}
proving \eqref{eq:A<Clocal-ap}. The proof of \eqref{eq:A<Cglobal-ap} is similar and we omit the details.

Finally,  for $1<p\leq\infty$, we have 
\begin{align*}
\mathcal{M}(\mathcal{A}^{(\alpha)}(u))(x)&=\sup_{r>0}\frac{1}{\mu(B(x,r))}\int_{B(x, r)\cap\pom}\mathcal{A}^{(\alpha)}u(\xi)\,d\sigma(\xi)\\
&\gtrsim \sup_{r>0}\frac{1}{r^s}\int_{B(x, r)\cap \pom} \int_{\gamma_\alpha(\xi)}|u(y)|\delta_\om(y)^{-n}\, dy\, d\sigma(\xi)\\
&\gtrsim \sup_{r>0}\frac{1}{r^s} \int_{B(x, r)\cap\om} \Big(\delta_\om(y)^{-s}\int_{B(y,  (1+\alpha)\delta_\om(y))}d\sigma(\xi)\Big)|u(y)|\delta_\om(y)^{s-n}\, dy\\
&\gtrsim \sup_{r>0}\frac{1}{r^s}\int_{B(x, r)\cap\om}|u(y)|\delta_\om(y)^{s-n}\, dy=\mathscr{C}_s(u)(x).
\end{align*}
This implies that 
$$
\|\mathscr{C}_s(u)\|_{L^p(\sigma)} \lec \|\mathcal{A}^{(\alpha)}(u)\|_{L^p(\sigma)}
$$
and the proof is now complete. 
\end{proof}

\vvv

\begin{lemma}\label{lem: gradient uniform convergence}
Let $B\subset \R^{n+1}$ be a compact and convex set and let $\{f_n\}_{n\geq 1}$ sequence of differentiable 
functions in $B$. 
Let $x_0 \in B$ such that the sequence $\{f_n(x_0)\}_{n=1}^\infty$ is convergent.
If $\nabla f_n \to \vec{F}$ uniformly on $B$, then, there exists a function $f:B \to \R$ which is differentiable at $x_0$ such that
$$
f_n \to f \,\,\textup{uniformly on}\,\,B \quad \textup{and} \quad \vec{F}(x_0)=\nabla f(x_0).
$$
Moreover, if  for every  $x \in B$,  $\{f_n(x)\}_{n=1}^\infty$ is convergent and $\nabla f_n$ is continuous on $B$, then  $\nabla f$ is continuous on $B$ as well.
\end{lemma}

\begin{proof}\footnote{This proof was given by Professor Giovanni Leoni  at  \href{https://math.stackexchange.com/questions/2255618/generalization-of-theorem-from-mathbb-r-to-mathbb-rn}{math.stackexchange/gradient convergence}.}
For any $x \in B$, we write 
\begin{align*}
&\frac{f(x)-f(x_0)-\vec F(x_0)(x-x_0)}{|x-x_0|} = \frac{f(x)-f(x_0)-(f_n(x)-f_n(x_0))}{|x-x_0|} \\
&+\frac{f_n(x)-f_n(x_0)-\nabla f_n(x_0)(x-x_0)}{|x-x_0|} + \frac{(\nabla f_n(x_0)-\vec F(x_0))(x-x_0)}{|x-x_0|}\\
&=: I(x)+II(x)+III(x).
\end{align*}
In order to control $I(x)$,  first note that,  since $\nabla f_n$ converges uniformly on $B$,  for fixed $\ve>0$,  there exists $n_1=n_1(\ve, B)\in\mathbb N$ such that for every $n, m > n_1$ it holds 
\begin{equation}\label{eq: nabla fn difference}
|\nabla f_n(x) - \nabla f_m(x)|< \frac{\ve}{3} \min\{1, \diam(B)^{-1}\}, \quad \textup{for all}\,\, x\in B.
\end{equation}
As $f_n$ is differentiable at $x_0$,  we have that there exists $\delta=\delta(\ve,B, x_0)$ such that 
\begin{equation}\label{eq: nabla fn definition}
\frac{|f_n(x)-f_n(x_0)-\nabla f_n(x_0) \cdot (x-x_0)|}{|x-x_0|}< \frac{\ve}{3}\min\{1, \diam(B)^{-1}\}, 
\end{equation}
for all $x \in  (B(x_0,\delta) \setminus \{x_0\})\cap B$.
By \eqref{eq: nabla fn difference} and \eqref{eq: nabla fn definition},  we have that, for any $n>n_1$, 
\begin{equation}\label{eq: difference quotient fn fm}
\frac{|f_n(x)-f_n(x_0)-f_m(x)+f_m(x_0)|  }{|x-x_0|} < \ve\,\min\{1, \diam(B)^{-1}\},
\end{equation}
for all $x \in  (B(x_0,\delta) \setminus \{x_0\})\cap B$.
 
For every $x, y\in B$, we set $\gamma : [0, 1]\to B$ to be the line segment $\gamma(t)=x+t(y-x)$ for $t\in[0, 1]$. Then $\gamma(0)=x$, $\gamma(1)=y$ and $\dot{\gamma}(t)=y-x$. Since $B$ is convex, the line $\gamma(t)$ lies entirely within $B$ for every $t\in[0, 1]$. 
For every $n, m> n_1$ we estimate
\begin{align*}
|f_n(x)-&f_n(y)-f_m(x)+f_m(y)| =\big|\int_0^1 (\nabla f_n(\gamma(t))\dot{\gamma}(t) - \nabla f_m(\gamma(t))\dot{\gamma(t)})\, dt\big| \\
&\leq \int_0^1 |\nabla f_n(\gamma(t)) - \nabla f_m(\gamma(t))||\dot{\gamma}(t)|\, dt \leq \frac{\ve}{2} \min\{1, \diam(B)^{-1}\}\,|x-y|\leq \frac{\ve}{3}.
\end{align*}
Since the sequence $\{f_n\}$ converges at $x_0$, there exists $n_2=n_2(\ve, x_0) \in \NN$ such that $n_2\geq n_1$  and for every $n> n_2$,
$$
|f_m(x)-f_n(x)|\leq |f_n(x)-f_n(x_0)-f_m(x)+f_m(x_0)| + |f_n(x_0)-f_m(x_0)|<\ve,
$$
for every $x\in B$. That is,  the sequence $\{f_n\}$ is uniformly Cauchy  and so  it  converges to some function $f$ uniformly  in $B$.   Therefore, letting $m\to\infty$ in \eqref{eq: difference quotient fn fm} and using \eqref{eq: nabla fn definition},  we get 
that, for every $n> n_2$,
$$
|I(x)|+|II(x)| < 2\ve, \quad \textup{for all}\,\, x \in  (B(x_0,\delta) \setminus \{x_0\})\cap B.
$$
Moreover, the sequence $\{\nabla f_n\}_{n\geq 1}$ converges to $\vec F$ uniformly in $B$ and thus, there exists
 $ n_3 \in \NN$ such that, for all $ n> n_3$, 
$$
| III(x) | \leq |\nabla f_n(x_0) - \vec F(x_0)| < \ve.
$$
Setting $n_0=\max\{n_1, n_2, n_3\}$ and using the above estimates we get that, for all $n> n_0$, it holds 
$$
\Big| \frac{f(x)-f(x_0)-\vec F(x_0)(x-x_0)}{|x-x_0|} \Big| < 3 \ve,
$$
for all $x \in  (B(x_0,\delta) \setminus \{x_0\})\cap B$, which implies that $f$ if differentiable at $x_0$ with $\nabla f(x_0) = \vec F(x_0)$.  
Finally,  if  $f_n$ is pointwisely convergent on $B$ and $\nabla f_n$ is continuous on $B$, then  $\nabla f$ is continuous on $B$ as the uniform limit of a sequence of  continuous functions.
\end{proof}

\vvv

\begin{lemma}\label{lem: sup gradient bound}
If  $F \in L^1_{\loc}(\om)$, then for any $x \in \om$ it holds  that
$$
|F(x)| \lec \frac{1}{\delta_\om(x)^{1+n/p}} \|\calC_s(F)\|_{L^p(\sigma)}, \quad \textup{for every} \,\, p\in (1, \infty)
$$
and 
$$
|F(x)| \lec \frac{1}{\delta_\om(x)} \sup_{\xi \in \pom} \calC_s(F)(\xi).
$$
\end{lemma}

\begin{proof}
Fix $x\in\om$ and note that if $c'= \frac{c}{c+1}$, then for any $z \in B(x,  c'\delta_\om(x))$   we have that 
$$
|z-x|\leq c' \,\dist(x, \pom) \leq c'\, \dist(z, \pom)+c'\,|z-x|, 
$$
which implies that $|z-x|\leq c \,\delta_\om(z)$, i.e.,  $x \in B_z$.  If $\xi_x \in \pom$ is a point such that 
$\delta_\om(x)=|x-\xi_x|$, it is clear that $ B_{x}\subset B(\xi, 3\delta_\om(x)) \cap \om$  for every $\xi \in B(\xi_x, \delta_\om(x))$.  So
\begin{align*}
\| \calC_s(F) \|_{L^p(\sigma)}&= \Big( \int_{\pom} \Big[ \sup_{r>0}\frac{1}{r^s} \int_{B(\xi, r)\cap\om} \sup_{y\in B_z}
| F(y) | \,\omega_s(z)\,dz \Big]^p \, d\sigma(\xi) \Big)^{1/p} \\
&\gtrsim \Big(\int_{B(\xi_x, \delta_\om(x))\cap\pom}\Big[\frac{1}{\delta_\om(x)^s}
\int_{B(\xi,  3 \delta_\om(x)) }\sup_{y\in B_z}|F(y)|  \, \omega_s(z)\,dz \Big]^p \,d\sigma(\xi) \Big)^{1/p}\\
&\gec \Big(\int_{B(\xi_x, \delta_\om(x))\cap\pom}\Big[\frac{1}{\delta_\om(x)^n}
\int_{B(x, c' \delta_\om(x)) }\sup_{y\in B_z}|F(y)|  \, dz \Big]^p \,d\sigma(\xi) \Big)^{1/p}\\
&\gtrsim \delta_\om(x)^{1+\frac{n}{p}} | F(x)|.
\end{align*}
Note that for $p=\infty$, by the same argument, we can directly infer  that 
$$
\sup_{\xi 
\in \pom} \calC_s(F)(\xi) \gtrsim \delta_\om(x) |F(x)|.
$$
\end{proof}

\vv

\begin{definition}
We define the  spaces
\begin{align}\label{eq: Nsharp dyadic}
\nn^\infty_{\sharp, \DD}(\om)&:=\{w\in C(\om): \sup_{P\in\WW(\om)} \, \sup_{x\in P}|w(x)-\fint_{P}w|<\infty \} \\
\nn^\infty_{\DD}(\om)&:=\{w\in C(\om): \sup_{P\in\WW(\om)} \sup_{x\in P} |w(x) | <\infty\}\label{eq: Ninftydyadic}
\end{align}
and equip them with 
\begin{align*}
\|w\|_{\nn^\infty_{\sharp, \DD}(\om)}&:=\sup_{P\in\WW(\om)} \, \sup_{x\in P}|w(x)-\fint_{P}w|,\,\,\textup{and}\\
\|w\|_{\nn^{\infty}_{\DD}(\om)} &:=\sup_{ P\in\WW(\om) } \sup_{x\in P} |w(x)|,
\end{align*}
respectively.  Note that $\|w\|_{\nn^\infty_{\sharp, \DD}(\om)}$ is a semi-norm, while $\|w\|_{\nn^{\infty}_{\DD}(\om)}$ is a norm.  We also define the  space 
\begin{equation}\label{eq: Nsum dyadic}
\nn_{\textup{sum}, \DD}(\om):= \{ u\in C^1(\om) : (u,   \delta_\om \nabla u)\in \nn^{\infty}_{\sharp, \DD}(\om)\times \nn^{\infty}_{\DD}(\om)\}
\end{equation}
and equip it with the  semi-norm 
$$
\|u\|_{\nn_{\textup{sum}, \DD}(\om)}:= \|u\|_{\nn^{\infty}_{\sharp, \DD}(\om)} + \| \delta_\om \, \nabla u\|_{\nn^{\infty}_\DD(\om)}. 
$$ 
\end{definition}

\vv

\begin{lemma}\label{lem: the sum space is normed}
If $\om \subset \R^{n+1}$ is an open and connected set, then $\nn^\infty_{\sharp, \DD}(\om)/\R$  and $\nn^\infty_{\textup{sum}, \DD}(\om)/\R$ are normed spaces.
\end{lemma}

\begin{proof}
Since it is easy to see that $\|\cdot \|_{\nn^\infty_{\sharp, \DD}(\om)}$ and $\| \cdot \|_{\nn^\infty_{\textup{sum}, \DD}}$ are  semi-norms, we will only show that if $\| u \|_{\nn^\infty_{\sharp, \DD}(\om)}=0$, then there exists $c \in \R$ such that $u=c$ in $\Omega$.  Indeed,  if $\sup_{P \in \WW(\om)} \|u\|_{\Lambda_P}=0$, then  $\max_{x\in \bar P}|u(x)-\fint_Pu|=0$, for every $P \in \WW(\om)$,  which implies $u=\fint_Pu$ on  $\bar P$ for every $P \in \WW(\om)$.  If $P_1, \, P_2\in\WW(\om)$ are such that $\bar P_1 \cap \bar P_2 \neq \emptyset$,  we have that $u=\fint_{P_1}u$ on $\bar P_1$ and $u=\fint_{P_2}u$ on $ \bar P_2$. As there exists $\xi \in \partial P_1\cap\partial P_2$ and $u$ is continuous in $\om$, it holds that $\fint_{P_1}u=\fint_{P_2}u$ for every  $P_1, \, P_2\in\WW(\om)$  such that $\bar P_1 \cap \bar P_2 \neq \emptyset$.  So,  if $\sup_{P \in \WW(\om)} \|u\|_{\Lambda_P}=0$,  there exists a constant $c \in \R$ so that $u=c$ in $\om$, which implies that  $\nn^\infty_{\sharp, \DD}(\om)/\R$  and  $\nn^\infty_{\textup{sum}, \DD}(\om)/\R $ are normed  spaces.
\end{proof}

\vv

\begin{remark}\label{rem: unique limit}
 By Lemma \ref{lem: the sum space is normed},  we have that if $\om$ is  open and connected set,  and a sequence converges in $\nn^\infty_{\sharp, \DD}(\om)/\R$  or $\nn_{\textup{sum}, \DD}(\om)$, then the limit is unique modulo constants.
\end{remark}

\vv

\begin{lemma}\label{lem: the sum space is complete}
Let $\om \subset \R^{n+1}$ be an open  set.  Then both $(\nn^\infty_{\sharp, \DD}(\om), \|\cdot\|_{\nn^\infty_{\sharp, \DD}(\om)})$   and $(\nn_{\textup{sum}, \DD}(\om), \|\cdot\|_{\nn_{\textup{sum}, \DD}(\om)})$ are sequentially complete. 
\end{lemma}

\begin{proof}
Let us first assume that $\om$ is connected.  We define the space
$$
\Lambda_P:=\{u:\om \to \R: u\in C(\bar P) \,\, \textup{and} \,\, \sup_{x\in P}\big|u(x)-\fint_{P}u \big|<\infty\}, 
$$
and equip it with the semi-norm $\|u\|_{\Lambda_{P}}:= \sup_{x\in P}|u(x)-\fint_{P}u|$.  In fact,  by the continuity of $u$ on $\bar P$,  we have that
$$
\|u\|_{\Lambda_{P}}=\max_{x\in \bar P} \Big|u(x)-\fint_Pu\Big|.
$$

We will first prove  that the space $\Lambda_{P}$ is sequentially complete with respect to the semi-norm $\|\cdot\|_{\Lambda_{P}}$ for any fixed  $P\in \WW(\om)$.  To this end,  let $\{u_n\}_{n\in\mathbb{N}}$ be a Cauchy sequence in $\Lambda_{P}$ and fix $\ve>0$. 
Then, there exists $n_0=n_0(\ve, P) \in\mathbb{N}$ such that for every $n, m> n_0$ it holds that 
$\|u_n-u_m\|_{\Lambda_{P}}< \ve$. Consequently, for any $y\in \bar P$, 
$$
 \Big|u_n(y)-u_m(y)-\fint_{P}(u_n(z)-u_m(z))\, dz  \Big|< \ve, 
$$
which means that the sequence $\big\{u_n - \fint_{P}u_n \big\}_{n\in\mathbb{N}}$ is uniformly Cauchy on $\bar P$ and so 
 it converges uniformly on $\bar P$ to some $u_P$.  Thus, there exists a positive integer $ n_1=n_1(\ve, P)$ such that for any $n >n_1$,
 \begin{equation}\label{eq: locally uniform convergence of Cauchy seq}
  \max_{y\in \bar P} \Big|u_n(y) - \fint_{P}u_n - u_P(y) \Big| < \ve/2,
 \end{equation}
and so, for any $n>n_1$,
\begin{align}
&\max_{y\in \bar P}\Big|u_n(y)-u_P(y) - \fint_{P}\left(u_n-u_P\right)\Big|   \label{eq: conv un to uP} \\
&\leq \max_{y\in \bar P}\Big|u_n(y)-\fint_{P}u_n - u_P(y) \Big|+  \Big| \fint_{P}\left(u_P(z)+u_n(z) - u_n(z)\right)\, dz \Big| \notag\\ 
&\leq\max_{y\in \bar P} \Big|u_n(y)-\fint_{P}u_n - u_P(y) \Big| + \fint_{P} \Big|u_n(z) - \fint_{P}u_n - u_P(z)\Big|\, dz < \ve,\notag 
\end{align}
concluding that $\Lambda_{P}$ is sequentially complete.  Moreover,  
\begin{align*}
\Big|\fint_P u_P\Big| \leq\fint_P \Big|u_P - u_n(x) + \fint_P u_n \Big| \leq \max_{x \in \bar P} \Big|u_n(x) - \fint_P u_n-u_P(x) \Big| \overset{n \to \infty}{\longrightarrow} 0,
\end{align*}
which implies that 
\begin{equation}\label{eq: limit average 0}
\fint_P u_P=0.
\end{equation}

 It is easy to see that since the half-open Whitney cubes $P\in\WW(\om)$ are disjoint,  the countable direct sum $ \bigoplus_{ P\in \WW(\om)} \Lambda_{P}$ equipped with the sup norm is sequentially complete.
Indeed, if $\{u_n\}$ be a Cauchy sequence in $ \bigoplus_{ P\in \WW(\om)} \Lambda_{P}$, then for $\ve>0$, there exists $n_0=n_0(\ve) \in \NN$ such that 
$$
 \max_{\bar P}\Big|u_n - \fint_P u_n - u_m + \fint_P u_m \Big| <\ve,\quad \textup{for all} \,\,  m, n>n_0 \,\, \textup{and all}\,\,  P \in \WW(\om).
 $$
 Using that $\lim_{m \to \infty}\big( u_m - \fint_P u_m \big)= u_P$ uniformly on $\bar P$,  we take limits as $m \to \infty$ in the last inequality and infer that 
$$
 \max_{\bar P}\Big|u_n(x) - \fint_P u_n - u_P (x) \Big| \leq  \ve,\quad \textup{for all} \,\,  n>n_0 \,\, \textup{and all}\,\,  P \in \WW(\om).
 $$
Then,  the function defined by $v(x)=u_P(x)$ for every $x \in P$ and every $P \in \WW(\om)$,  satisfies $u_n \to v$ in $ \bigoplus_{ P\in \WW(\om)} \Lambda_{P}$.

We shall now prove that $\nn_{\sharp, \DD}(\om)$ is sequentially complete with respect to the semi-norm $\| \cdot \|_{ \nn^\infty_{\sharp, \DD}(\om)}$.  
 Let  $\{u_n\}$ be a Cauchy sequence in $\nn_{\sharp, \DD}(\om)$.  For fixed $\ve>0$, there exists $n_0=n_0(\ve)\in\mathbb N$ such that for every $n, m> n_0$, we have that $\|u_n-u_m\|_{\nn_{\sharp \DD}(\om)}<\ve$.  Note that   
$$
\nn^\infty_{\sharp, \DD}(\om)=\bigoplus_{ P\in \WW(\om)} \Lambda_{P}  \cap  C(\om)
$$
and as $ \bigoplus_{ P\in \WW(\om)} \Lambda_{P}$ is sequentially complete,   there exists $v$ (as defined above) such that $u_n \to  v$ in the $\| \cdot\|_{\nn^\infty_{\sharp, \DD}(\om)}$ semi-norm.   So  there exists $\{u_P\}_{P \in \WW(\om)}$ satisfying \eqref{eq: limit average 0} such that $u_n \to u_P$ in the $\| \cdot\|_{\Lambda_{P}}$ semi-norm, uniformly in $P \in \WW(\om)$.  Note that $u_P \in C(\bar P)$ for each $P \in \WW(\om)$ as the uniform limit of a sequence of continuous functions.

Our main  goal   is to modify each $u_P$ adding a suitable constant so that we can define a new  function $u \in C(\om)$ which satisfies  $u_n \to u$ in $\|\cdot\|_{\nn^\infty_{\sharp, \DD}(\om)}$.   To do so,  it is important to show that if $P_1, P_2 \in \WW(\om)$ with $\partial P_1 \cap \partial P_2 \neq \emptyset$, then, for every $x\in\partial P_1\cap \partial P_2$,  
\begin{equation}\label{eq: constant in boundary points}
u_{P_1}(x) + C(P_1, P_2)=u_{P_2}(x).
\end{equation}
If
$$
C_n(P_1, P_2):=\fint_{P_1} u_n -\fint_{P_2} u_n,
$$
 since $\{u_n\}$ is Cauchy in $\nn^\infty_{\sharp, \DD}(\om)$,  for any $x \in \partial P_1 \cap \partial P_2$,  it holds that 
\begin{align*}
|C_n(P_1, P_2)&-C_m(P_1, P_2)| \\
&\leq  \Big| u_n(x)-u_m(x) - \fint_{P_1} (u_n-u_m) \Big|+\Big| u_n(x)-u_m(x) - \fint_{P_2} (u_n-u_m) \Big|\\
&\leq 2 \|u_n-u_m\|_{\nn^\infty_{\sharp, \DD}(\om)} \to 0, \quad \textup{as}\,\, m,n \to \infty.
\end{align*}
Therefore,  there exists $C(P_1, P_2)\in \R$ such that $\lim_{n \to \infty}C_n(P_1, P_2)=C(P_1, P_2)$.  So, for every $x\in\partial P_1\cap \partial P_2$,  
\begin{align*}
u_{P_2}(x)-u_{P_1}(x)=&  \left( u_{P_2}(x)-  u_{n}(x)+ \fint_{P_2} u_{n} \right)\\
& - \left( u_{P_1}(x)-  u_{n}(x)+ \fint_{P_1} u_{n} \right)+  C_{n}(P_1, P_2) \overset{n \to \infty}{\longrightarrow}  C(P_1, P_2),
\end{align*}
which shows  \eqref{eq: constant in boundary points}.

If $\pom$ is compact, we fix a (starting) cube $P_0\in\WW(\om)$ such that   $\ell(P_0) \leq \diam (\pom)$ and $\ell(P) \leq \ell(P_0)$ for every $P \in \WW(\om)$.  If $\pom$ is unbounded, we pick as a starting cube some $P_0\in\WW(\om)$  such  that $\ell(P_0)=1$.   Once we have fixed such a cube $P_0$, we define 
$$
\mathcal{G}_1(P_0):=\{P \in \WW(\om): \partial P\cap \partial P_0 \neq \emptyset \}.
$$
For every $P\in \mathcal{G}_1(P_0)$, we  let $u_P$ be the limit of $\{u_n\}$ in the $\Lambda_P$ semi-norm and set 
\begin{equation}\label{eq: wt uP}
{u}^1_P :=u_P + C(P, P_0).
\end{equation}
It is clear that $u_n \to u^1_P$ in the $\Lambda_P$ semi-norm as well, while, in view of \eqref{eq: constant in boundary points},  ${u}^1_P=u_{P_0}$ on $\partial P_0 \cap \partial P$.   Observe that, repeating the proof of  \eqref{eq: constant in boundary points}, we can show that for every $\wt{P}\in\mathcal{G}_1(P)\cap\mathcal{G}_1(P_0)\setminus \{P_0\}$, there exists a constant $C(P, \wt P)$ such that 
\begin{equation}\label{eq: utildeP on common boundary}
{u}^1_{P}+C(P, \wt{P}) = {u}^1_{\wt P} \qquad \textup{on}\,\, \partial P\cap \partial \wt{P}.
\end{equation}
Since $P, \wt{P}\in\mathcal{G}_1(P_0)$ and $\wt{P}\in\mathcal{G}_1(P)$ we have 
$$
{u}^1_{\wt{P}}(z) = u_{P_0}(z) = {u}^1_{P}(z) \quad \forall z\in\partial P_0\cap\partial P\cap\partial\wt{P},
$$
which implies that $C(P, \wt{P}) =0$ and so, by \eqref{eq: utildeP on common boundary}, 
\begin{equation}\label{eq: utildeP on common boundary-bis}
{u}^1_{P}= {u}^1_{\wt P} \qquad \textup{on}\,\, \partial P\cap \partial \wt{P}.
\end{equation}
Setting  $C_P:= C(P,P_0)$, we  can write ${u}^1_P :=u_P + C_P$ and define 
$$
v^1:=
\begin{dcases}
&{u}^1_P, \qquad P\in \mathcal{G}_1(P_0) \\
&u_P, \qquad P\in\WW(\om)\setminus \mathcal{G}_1(P_0).
\end{dcases}
$$
Notice that 
$$
v^1 \in C( \mathcal{O}_1),  \quad \textup{where} \quad \mathcal{O}_1:=\bigcup_{P \in \mathcal{G}_1(P_0)} \bar P,
$$
and $\lim_{n \to \infty} \|u_n - v^1\|_{\nn^\infty_{\sharp, \DD}(\om)}=0$.

Moving forward, for every $P\in\mathcal{G}_1(P_0)$,  we set 
$$
\wt{\mathcal{G}}_1(P):= \mathcal{G}_1(P)\setminus \mathcal{G}_1(P_0)\quad
\textup{and}    \quad 
\mathcal{G}_2(P_0):=\bigcup_{P\in\mathcal{G}_1(P_0)} \wt{\mathcal{G}}_1(P).
$$
If $P\in\mathcal{G}_1(P_0)$ and $P'\in\wt{\mathcal{G}}_1(P)$, we define 
$$
{u}^2_{P'} :=u_{P'} + C(P', P),
$$
 which also satisfies $u_n \to u^2_{P'}$ in the $\Lambda_{P'}$ semi-norm. By the same arguments as above,  we can show that
$$
{u}^2_{P'}={u}^1_{P} \qquad \textup{on}\,\, \partial P \cap \partial P'
$$
and, for every $P''\in\wt{\mathcal{G}}_1(P)$ such that $\partial P' \cap \partial P'' \neq \emptyset$, 
$$
{u}^2_{P'}={u}^2_{P''}\qquad \textup{on}\,\, \partial P' \cap \partial P''.
$$

Moreover,  if $\wt P \in\mathcal{G}_1(P_0) $ such that $\partial P \cap \partial \wt P \neq \emptyset$ and $P' \in \mathcal{G}_1( P) \cap \mathcal{G}_1(\wt P)$,  we define $\wt{u}^2_{P'} :=u_{P'} + C(P', \wt P)$.  Then, for every $x \in \partial P'  \cap  \partial P  \cap  \partial \wt P $, it holds that  
$$
u^1_P(x)=  u^1_{\wt P}(x), \quad  {u}^2_{P'}(x) =u^1_P(x), \quad \textup{and}\quad\wt{u}^2_{P'}(x)=u^1_{\wt P}(x),
$$
which we may combine to deduce that $C(P', P)=C(P', \wt P)$ and consequently
$$
{u}^2_{P'} = \wt{u}^2_{P'}.\footnote{The part of the argument dealing with two descendants and one ancestor is better suited to cases where the descendants have, at most, equal side-lengths with the ancestor. Meanwhile, the part involving two ancestors and one descendant is more relevant for cases where the descendant has a larger side-length than the ancestor. The terms ``ancestor" and ``descendant" are used in relation to the selection process.}
$$
The same is true for every $\wt P \in\mathcal{G}_1(P_0) $ such that $P' \in \wt{\mathcal{G}}_{1}(P) \cap\wt{\mathcal{G}}_{1}(\wt P)$ since, for every such cube,  there is a chain of cubes $P=P_1,  P_2, \dots, P_{N}=\wt P$ such that $P_k\in \mathcal{G}_1(P_0)$ and  $\partial P_k \cap \partial P_{k+1} \neq \emptyset$ for $k \in \{1,\dots, N\}$,  where $N$ is a dimensional constant.  Therefore, we can unambiguously set  $C_{P'}:= C(P',P)$ for any $P\in\mathcal{G}_1(P_0)$ such that   $P' \in \wt{\mathcal{G}}_{1}(P)$ and  so  ${u}^2_{P'} :=u_{P'} + C_{P'}$.  
If 
$$
\mathcal{F}_2(P_0):=\mathcal{G}_1(P_0) \cup \mathcal{G}_2(P_0),
$$
we define
$$
v^2:=
\begin{dcases}
&{u}^1_P, \qquad P\in \mathcal{G}_1(P_0)\\
&{u}^2_P, \qquad P\in \mathcal{G}_2(P_0)\\
&u_P, \qquad P\in\WW(\om)\setminus \mathcal{F}_2(P_0)
\end{dcases}
$$
and notice that 
$$
v^2 \in C( \mathcal{O}_2),  \quad \textup{where} \quad \mathcal{O}_2:=\bigcup_{P \in  \mathcal{F}_2(P_0)} \bar P,
$$
and $\lim_{n \to \infty} \|u_n - v^2\|_{\nn^\infty_{\sharp, \DD}(\om)}=0$.

We proceed by iteration and, for every $P\in\mathcal{G}_{k-1}(P_0)$,  $k \geq 2$, we set 
$$
\wt{\mathcal{G}}_{1}(P):= \mathcal{G}_1(P)\setminus \mathcal{G}_{k-1}(P_0)\quad
\textup{and}    \quad 
\mathcal{G}_k(P_0):=\bigcup_{P\in\mathcal{G}_{k-1}(P_0)} \wt{\mathcal{G}}_{1}(P).
$$ 
If $P\in\mathcal{G}_{k-1}(P_0)$ and $P'\in\wt{\mathcal{G}}_1(P)$,   we define 
$$
{u}^k_{P'} :=u_{P'} + C(P', P)
$$
which also satisfies that $u_n \to u^k_{P'}$ in the $\Lambda_{P'}$ semi-norm.  Arguing as above, we can  prove that
$$
{u}^k_{P'}={u}^{k-1}_{P} \qquad \textup{on}\,\, \partial P \cap \partial P'
$$
 and,  for every $P''\in\wt{\mathcal{G}}_1(P)$ such that $\partial P' \cap \partial P'' \neq \emptyset$, 
$$
{u}^k_{P'}={u}^k_{P''}\qquad \textup{on}\,\, \partial P' \cap \partial P''.
$$
Moreover, we can unambiguously set  $C_{P'}:= C(P',P)$ for any $P\in\mathcal{G}_{k-1}(P_0)$ such that   $P'\in\wt{\mathcal{G}}_1(P)$ and  so  ${u}^k_{P'} :=u_{P'} + C_{P'}$.    
If 
$$
\mathcal{F}_k(P_0):= \bigcup_{j=1}^k\mathcal{G}_j(P_0)
$$
and 
$$
v^k:=
\begin{dcases}
\begin{aligned}
{u}^1_P, \qquad &P\in \mathcal{G}_1(P_0) \\
{u}^2_P, \qquad &P\in \mathcal{G}_2(P_0) \\
& \vdotswithin{P} \\[-0.5ex]\\
{u}^k_P, \qquad &P\in \mathcal{G}_k(P_0) \\
u_P, \qquad &P\in\WW(\om)\setminus \mathcal{F}_k(P_0),
\end{aligned}
\end{dcases}
$$
 it is easy to see that
$$
v^k \in C( \mathcal{O}_k),  \quad \textup{where} \quad \mathcal{O}_k:=\bigcup_{P \in  \mathcal{F}_k(P_0)} \bar P,
$$
and $\lim_{n \to \infty} \|u_n - v^k\|_{\nn^\infty_{\sharp, \DD}(\om)}=0$. 

If we  define $u(x):=\lim_{k\to\infty} v^k(x)$ for  $x \in \om$,  then, by construction and the fact that $\om$ is connected,    it is clear that $u \in C(\om)$ and $\lim_{n \to \infty} \|u_n - u\|_{\nn^\infty_{\sharp, \DD}(\om)}=0$.  Hence,  $\nn^\infty_{\sharp, \DD}(\om)$ is sequentially complete.

It remains to prove that $\nn_{\textup{sum}, \DD}(\om)$ is sequentially complete.  To this end, let  $\{u_n\}$ be a Cauchy sequence in $\nn_{\textup{sum}, \DD}(\om)$. Then for fixed $\ve>0$, there exists $n_0=n_0(\ve)\in\mathbb N$ such that for every $n, m\geq n_0$ we have $\|u_n-u_m\|_{\nn_{\textup{sum}, \DD}(\om)}<\ve$, which implies that 
$$
\|u_n-u_m\|_{\nn^\infty_{\sharp, \DD}(\om)}<\ve \quad  \textup{and} \quad \| \delta_\om(\nabla u_n-\nabla u_m)\|_{\nn^\infty_{\DD}(\om)}<\ve.
$$
Since $\nn^\infty_{\sharp, \DD}(\om)$ is sequentially complete,  if $u=\lim_{n \to \infty} u_n  \in \nn^\infty_{\sharp, \DD}(\om)$, it suffices to show that
\begin{equation}\label{eq:u C1}
u \in C^1(\om)\quad \textup{and}  \quad \lim_{n \to \infty} \| \delta_\om(\nabla u_n-\nabla u)\|_{\nn^\infty_{\DD}(\om)}=0.
\end{equation}

Since $\{  \delta_\om \nabla u_n\}_{n=1}^\infty \subset C(\om)$ is Cauchy in $\nn^\infty_{\DD}(\om)$,  then it  is uniformly Cauchy on $\bar P$ (uniformly in $P\in \WW(\om)$)  and so there exists $\vec w_P^0 \in C(\bar P)$ such that $ \delta_\om \nabla u_n \to \vec w_P^0$ uniformly on $\bar P$  for every $P\in \WW(\om)$.  If $P_1, P_2 \in \WW(\om)$ so that $\partial P_1 \cap\partial P_2 \neq  \emptyset$,  for every $x \in \partial P_1 \cap\partial P_2$, it holds that
$$
\vec w_{P_1}^0(x)-\vec w_{P_2}^0(x)= (\vec w_{P_1}^0(x)- \delta_\om(x)\nabla u_n(x))- (\vec w_{P_1}^0(x)- \delta_\om(x)\nabla u_n(x)) \to 0, \,\, \textup{as}\,\, n \to \infty,
$$
which implies that $\vec w_{P_1}^0(x)=\vec w_{P_2}^0(x)$ for every $x \in \partial P_1 \cap\partial P_2$. Therefore, if we define $\vec w_0(x):= \vec w_{P}^0(x)$ for every $x \in P$  and   all $P \in \WW(\om)$,  it is evident  that $\vec w_0 \in C(\om)$ and $\lim_{n \to \infty}\| \delta_\om \nabla u_n - \vec w_0 \|_{\nn^\infty_{\DD}(\om)}=0$.  Set now $\vec w :=  \delta_\om^{-1} \vec w_0\in C(\om)$ and 
$$
w_P^n:=u_n- \fint_{P}u_n \quad \textup{on} \,\,  \bar P,
$$
and note that 
$$
w_P^n \to u_P \quad  \textup{uniformly on} \,\,\bar P \quad \textup{and}\quad   \nabla w_P^n=\nabla u_n \to \vec w \quad  \textup{uniformly on} \,\,\bar P.
$$
We can now apply Lemma \ref{lem: gradient uniform convergence} on each $\bar P$  to the sequence $\{w_P^n\}_{n=1}^\infty \subset C^1(\bar P)$ and  deduce that, for every $P \in \WW(\om)$,
$$
u_P \,\,  \textup{is differentiable on}\,\,  \bar P\qquad \textup{and} \qquad \nabla u_P=\vec w \in C(\bar P). 
$$
In particular,  $\nabla u_P(z)=\vec w(z)=\nabla u_{P'}(z)$ for every $z \in \partial P \cap \partial P'$.   If  $P \in \WW(\om)$ and $y\in  \partial P$, for fixed $\ve>0$, there exist $\delta_P=\delta_P(\ve,y)>0$ such that for every $x \in \bar P \cap B(y, \delta_P)$, 
\begin{equation}\label{eq:u differentiable P}
\frac{|u_P(x)-u_P(y)-\vec w \cdot (x-y)|}{|x-y|}<\ve.
\end{equation}
Let us fix  a cube $P \in \WW(\om)$ and a point $z \in \partial P$.  We define $\mathcal{Z}$ to be the family of Whitney cubes $P'$ such that $ z\in \partial P'$ (the number of such cubes is at most a fixed dimensional constant) and set  
$$
\delta=\min\Big\{ \dist \Big(z,  \om \setminus \bigcup_{ P' \in \mathcal Z} \bar P' \Big), \min_{P' \in \mathcal Z}\delta_{P'}\Big\}.
$$
Note that for every $x \in  B(z, \delta) \setminus \partial P$,  there exists a unique $P_x \in \mathcal{Z}$ such that $x \in P_x^o$ and, by \eqref{eq:u differentiable P},
$$
\frac{|u(x)-u(z)-\vec w \cdot (x-z)|}{|x-z|}=\frac{|u_{P_x}(x)-u_{P_x}(z)-\vec w \cdot (x-z)|}{|x-z|}<\ve.
$$
If $x \in  B(z, \delta) \cap \partial P$,  since $u(x)=u_{P}(x)+C_{P}$,  by \eqref{eq:u differentiable P}, it holds that
$$
\frac{|u(x)-u(z)-\vec w \cdot (x-z)|}{|x-z|}=\frac{|u_{P}(x)-u_{P}(z)-\vec w \cdot (x-z)|}{|x-z|}<\ve,
$$
showing that $u$ is differentiable at any $z \in  \bigcup_{P \in \WW(\om)} \partial P$. Therefore,   since $u_P$ is differentiable in $P^o$ for any $P \in \WW(\om)$ and $u=u_P+C_P$ in $P^o$, we deduce that $u$ is differentiable in $\om$ and $\nabla u = \vec w \in C(\om)$,  proving \eqref{eq:u C1} and concluding the proof of the lemma when $\om$ is connected.

If $\om$ is not connected,  it can be written as the union of at most countably many disjoint connected components, i.e.,  $\om=\bigcup_{i=1}^\infty \om_i$.  For every $i \in \NN$,  we have proved that $\nn_{\sharp, \DD}(\om_i) $ and $\nn_{\textup{sum}, \DD}(\om_i)$ are sequentially complete and, as the connected components $\om_i$ are mutually disjoint, we have that 
$$
\nn_{\sharp, \DD}(\om)=\bigoplus_{i=1}^\infty \nn_{\sharp, \DD}(\om_i)  \,\, \,\textup{and}\,\,\,  \nn_{\textup{sum}, \DD}(\om)=\bigoplus_{i=1}^\infty \nn_{\textup{sum}, \DD}(\om_i) \quad \textup{are sequentially complete},
$$ 
concluding the proof of Lemma \ref{lem: the sum space is complete}.
\end{proof}

\vv

\begin{lemma}\label{lem: norm equivalence}
Let $\om \subset \R^{n+1}$ be an open set. The semi-norms $\|\cdot\|_{\nn_{\textup{sum}}(\om)}$ and $\|\cdot\|_{\nn_{\textup{sum}, \DD}(\om)}$ are equivalent and the implicit constants only depend on $n$.
\end{lemma}

\begin{proof}
 If $x\in\om$ and $B_x=B(x, c\,\delta_\om(x))$,  for $c \in (0, 1/2]$, then,  by easy volume considerations,  one can prove that there exists a uniformly bounded number of Whitney cubes  $P\in\WW(\om)$ that cover the ball $B_x$.  We denote this collection  of cubes  by $\mathcal B_x$. By the mean value theorem,  for any $P \in \mathcal B_x$, we estimate
\begin{align}
\Big| \fint_{B_x}u-\fint_{P} u \Big| &\leq \fint_{B_x}\fint_{P} |u(z)-u(\zeta)|\, dz\,d\zeta\lec\, \delta_\om(x)\,\max_{P \in \mathcal B_x} \max_{z\in \bar P}|\nabla u(z)|\notag\\
&\lec \max_{P \in \mathcal B_x}\max_{z\in \bar P}\delta_\om(z)|\nabla u(z)| \lesssim \|  \delta_\om \nabla u \|_{\nn^\infty_\DD(\om)}.\notag
\end{align}
Consequently,
\begin{align*}
\|u\|_{\nn_{\sharp}^\infty(\om)}  &\leq \sup_{x\in \om} \sup_{P \in \mathcal B_x}\sup_{y \in P}\Big| u(y) - \fint_{B_x}u \Big| \\
&\lec  \sup_{x\in \om} \sup_{P \in \mathcal B_x}\sup_{y \in P}\Big| u(y) - \fint_{P}u \Big| +\| \delta_\om \nabla u \|_{\nn^\infty_\DD(\om)}\\
&\lesssim \| u \|_{\nn_{\sharp, \DD}^\infty(\om)}  + \|  \delta_\om \nabla u \|_{\nn^\infty_\DD(\om)}.
\end{align*}
Since for any $x\in\om$ there exists a unique half-open cube $P\in \WW(\om)$ such that $x\in P$, we have that
$$
\delta_\om(x)|\nabla u(x)| \leq \sup_{z\in P}\delta_\om(z)|\nabla u(z)|\leq  \| \delta_\om \nabla u \|_{\nn^\infty_\DD(\om)}$$
and so $\|   \delta_\om \nabla u \|_{\nn^\infty(\om)} \lesssim  \|  \delta_\om \nabla u \|_{\nn^\infty_\DD(\om)}$.
Therefore, 
\begin{align}\label{eq: equiv sxesi3}
\|u\|_{\nn_{\textup{sum}}(\om)} &= \|u\|_{\nn_{\sharp}^\infty(\om)} + \| \delta_\om\nabla u \|_{\nn^\infty(\om)}  \\
&\lec \|u\|_{\nn_{\sharp, \DD}(\om)} + \|   \delta_\om \nabla u \|_{\nn^\infty_\DD(\om)} =\|u\|_{\nn_{\textup{sum}, \DD}(\om)}.\notag
\end{align}

For the converse direction,  let us fix  $c\in(0, 1/2]$ so that for every $P \in \WW(\om)$, if $x_P$ is the center of $P$, it holds that $P \subset B(x_P, c\,\delta(x_P))=:B_P$.   By the mean value theorem,  for any $P \in \WW(\om)$,  we have that
$$
\Big|\fint_{B_P} u-\fint_{P} u \Big|\leq \max_{z\in B_P} \delta_\om(z)|\nabla u(z)| \leq  \| \delta_\om \nabla u\|_{\nn^\infty(\om)}
$$
and so
$$
\sup_{x \in P} \Big| u(z)- \fint_P u \Big| \leq   \sup_{x \in P} \Big| u(z)- \fint_{B_P} u \Big| + \| \delta_\om \nabla u\|_{\nn^\infty(\om)} \lec \|u\|_{\nn_{\sharp}^\infty(\om)}  +\|  \delta_\om \nabla u\|_{\nn^\infty(\om)}.
$$
This implies that $\|u\|_{\nn_{\sharp, \DD}(\om)}  \lec \|u\|_{\nn_{\textup{sum}}(\om)}$ and as 
$$
\| \delta_\om \nabla u \|_{\nn^\infty_\DD(\om)} \leq \sup_{P \in \WW(\om)} \sup_{z \in B_P} \delta_\om(z) |\nabla u(z)| \leq \|\delta_\om\nabla u \|_{\nn^\infty(\om)},
$$
we infer that $ \|u\|_{\nn_{\textup{sum}, \DD}(\om)} \lec \|u\|_{\nn_{\textup{sum}}(\om)}$ for any $c \in (0, 1/2)$ such that $P \subset B_P$. 
 It remains to prove that if $0<c_0 <  c$, then 
\begin{align}\label{eq:A5a}
\sup_{x \in \om} \sup_{z \in B(x,c \delta_\om(x))} \Big| u(z) - \fint_{B(x,c \delta_\om(x))} \Big| &\lec \sup_{y \in \om} \sup_{z \in B(y,c_0 \delta_\om(y))} \Big| u(z) - \fint_{B(y,c_0 \delta_\om(y))} \Big| \notag\\
&+\sup_{y \in \om} \sup_{z \in B(y,c_0 \delta_\om(y))}  \delta_\om(z) |\nabla u(z)|
\end{align}
and 
\begin{align}\label{eq:A5b}
\sup_{y \in \om} \sup_{z \in B(y,c \delta_\om(y))}   \delta_\om(z) |\nabla u(z)|  \lec \sup_{y \in \om} \sup_{z \in B(y,c_0 \delta_\om(y))}   \delta_\om(z) |\nabla u(z)|,
\end{align}
where the implicit constants are independent of $c_0$ and $c$.  For any $x \in \om$,  there exists a uniformly bounded number of  balls $\{B_j^x \}_{j=1}^N$ so that $B_j^x=B(x_j, c_0 \delta_\om(x_j))$, where $x_j \in B(x,c \delta_\om(x))$ and $B(x,c\, \delta_\om(x)) \subset \cup_{j=1}^N B_j^x$.   Therefore, by the mean value theorem,  we have that for any $x\in \om$,
\begin{align*}
\Big| \fint_{B(x,c \delta_\om(x))} u - \fint_{B_j} u\Big| &\lec \sup_{j\in \{1,\dots,N\}} \sup_{z \in B_j^x} \delta_\om(z) |\nabla u(z)|\\
& \leq \sup_{y \in \om} \sup_{z \in B(y,c_0 \delta_\om(y))} \delta_\om(z) |\nabla u(z)|.
\end{align*}
Therefore,
\begin{align*}
\sup_{x \in \om}& \sup_{z \in B(x,c \delta_\om(x))} \Big| u(z) - \fint_{B(x,c \delta_\om(x))}u \Big| \leq \sup_{x \in \om} \sup_{j\in \{1,\dots,N\}} \sup_{z \in B_j^x}  \Big| u(z) - \fint_{B(x,c \delta_\om(x))} u \Big| \\
& \leq  \sup_{x \in \om} \sup_{j\in \{1,\dots,N\}} \left( \sup_{z \in B_j^x}  \Big| u(z) - \fint_{B_j} u \Big| + \Big| \fint_{B(x,c \delta_\om(x))} u - \fint_{B_j} u\Big| \right) \\
& \lec\sup_{y \in \om} \sup_{z \in B(y,c_0 \delta_\om(y))}  \Big| u(z) - \fint_{B(y,c_0 \delta_\om(y))} u \Big| + \sup_{y \in \om} \sup_{z \in B(y,c_0 \delta_\om(y))} \delta_\om(z) |\nabla u(z)|,
\end{align*}
which shows \eqref{eq:A5a}.   By similar considerations,  it is easy to prove \eqref{eq:A5b}, concluding  the proof of the  lemma.
\end{proof}

\vv

\begin{corollary}\label{cor: the sum space is complete}
The space $(\nn_{\textup{sum}}(\om), \|\cdot\|_{\nn_{\textup{sum}(\om)}})$ is sequentially complete. 
\end{corollary}

\begin{proof}
The result readily follows  from Lemmas \ref{lem: the sum space is complete} and  \ref{lem: norm equivalence}.
\end{proof}

\vv

An immediate corollary of Lemmas  \ref{lem: the sum space is normed} and  \ref{lem: the sum space is complete} is the following.
\begin{corollary}\label{cor: the sum space is Banach}
If $\om \subset \R^{n+1}$ is an open set and $\om=\bigcup_{i=1}^\infty \om_i$,  where $\{\om_i\}_{i=1}^\infty$ are the connected components of $\om$,  then    $\bigoplus_{i=1}^\infty [\nn_{\textup{sum}}(\om_i) /\R]$ is a  Banach space.  
\end{corollary}

\vvv

\vvv

\frenchspacing
\bibliographystyle{alpha}

\newcommand{\etalchar}[1]{$^{#1}$}
\def\cprime{$'$}

\end{document}